\documentclass[11pt]{article}

\let\oldmarginpar\marginpar
\renewcommand\marginpar[1]{\-\oldmarginpar[\raggedleft\footnotesize #1]%
{\raggedright\footnotesize #1}}

\usepackage{lineno}
\usepackage{mathptmx}
\usepackage{amsmath}
\usepackage{amscd}
\usepackage{amssymb}
\usepackage{amsthm}
\usepackage{xspace}
\usepackage[all,tips]{xy}
\usepackage[dvips]{graphicx}
\usepackage{verbatim}
\usepackage{syntonly}
\usepackage{hyperref}
\usepackage{arydshln}
\usepackage{lineno,color}

\theoremstyle{plain}
\newtheorem{thm}{Theorem}[section]
\newtheorem{cor}[thm]{Corollary}
\newtheorem{prop}[thm]{Proposition}
\newtheorem{lemma}[thm]{Lemma}

\theoremstyle{definition}

\newtheorem{defn}[thm]{Definition}

\DeclareMathOperator{\Height}{Height}
 
\DeclareMathOperator{\Aut}{Aut}

 \DeclareMathOperator{\Ad}{Ad}
\DeclareMathOperator{\ad}{ad}

\DeclareMathOperator{\GCD}{gcd}
\DeclareMathOperator{\lcm}{lcm}

\DeclareMathOperator{\D}{D}

\DeclareMathOperator{\Farb}{F}
\DeclareMathOperator{\Conj}{Conj}
\DeclareMathOperator{\Log}{Log}

\DeclareMathOperator{\Free}{N}
\DeclareMathOperator{\Bas}{B}
\DeclareMathOperator{\Depth}{D}
\DeclareMathOperator{\Heis}{H}
\DeclareMathOperator{\CDepth}{CD}

\newcommand{\eps}{\varepsilon}

\newcommand{\bdef}{\overset{\text{def}}{=}}

\newcommand{\al}{\alpha}
\newcommand{\be}{\beta}
\newcommand{\ga}{\gamma}
\newcommand{\Ga}{\Gamma}

\newcommand{\la}{\lambda}
\newcommand{\La}{\Lambda}
\newcommand{\De}{\Delta}

\newcommand{\nid}{\noindent}

\newcommand{\innp}[1]{\left< #1 \right>}

\newcommand{\set}[1]{\left\{#1\right\}}

\newcommand{\pr}[1]{\left( #1 \right) }

\newcommand{\Fr}[1]{\ensuremath{\mathfrak{#1}}}

\newcommand{\N}{\ensuremath{\mathbb{N}}}
\newcommand{\Q}{\ensuremath{\mathbb{Q}}}
\newcommand{\R}{\ensuremath{\mathbb{R}}}
\newcommand{\Z}{\ensuremath{\mathbb{Z}}}

\newcommand{\map}[3]{#1 : #2 \rightarrow #3}

\newcommand{\modulo}[2]{#1 \: \left( \: \text{mod } \: #2 \right)}

\newcommand{\nsub}{\trianglelefteq}
\usepackage{fancyhdr}

\fancypagestyle{plain}

\begin{document}


\title{\textbf{Effective Separability of \\ Finitely Generated Nilpotent Groups}}
\author{Mark Pengitore\thanks{Purdue University, West Lafayette, IN. E-mail: \tt{mpengito@purdue.edu}}} 
\maketitle


\begin{abstract}
\nid We give effective proofs of residual finiteness and conjugacy separability for finitely generated nilpotent groups. In particular, we give precise asymptotic bounds for a function introduced by Bou-Rabee that measures how large the quotients that are needed to separate non-identity elements of bounded length from the identity which improves the work of Bou-Rabee. Similarly, we give polynomial upper and lower bounds for an analogous function introduced by Lawton, Louder, and McReynolds that measures how large the quotients that are needed to separate pairs of distinct conjugacy classes of bounded word length using work of Blackburn and Mal'tsev.
\end{abstract}

\section{Introduction}

We say that $\Gamma$ is \emph{residually finite} if for each $\gamma \in \Gamma - \set{1}$ there exists a surjective homomorphism to a finite group $\map{\varphi}{\Gamma}{Q}$ such that $\varphi(\gamma) \neq 1$. Mal'tsev \cite{residual_linear} proved that if $\Gamma$ is a residually finite finitely presentable group, then there exists a solution to the word problem of $\Ga$. We say that $\Gamma$ is \emph{conjugacy separable} if for each non-conjugate pair $\gamma, \eta$ in $\Gamma$ there exists a surjective homomorphism to a finite group $\map{\varphi}{\Gamma}{Q}$ such that $\varphi(\gamma)$ and $\varphi(\eta)$ are not conjugate. Mal'tsev \cite{residual_linear} also proved that if $\Ga$ is a conjugacy separable finitely presentable group, then there exists a solution to the conjugacy problem of $\Ga$

Residual finiteness, conjugacy separability, subgroup separability, and other residual properties have been extensively studied and used to great effect in resolving important conjectures in geometry, such as the work of Agol on the Virtual Haken conjecture.  Much of the work in the literature has been to understand which groups satisfy various residual properties. For example, free groups, polycyclic groups, surface groups, and fundamental groups of compact, orientable $3$-manifolds have all been shown to be residually finite and conjugacy separable \cite{Blackburn,Formanek,conj_3_manifolds,hempel,remeslennikov,stebe}. Recently, there have been several papers that have made effective these separability properties for certain classes of groups. The main purpose of this article is to improve on the effective residual finiteness results of \cite{BouRabee10} and establish effective conjugacy separability results, both for the class of finitely generated nilpotent groups.

\subsection{Residual Finiteness}
For a finitely generated group $\Gamma$ with a finite generating subset $S$, \cite{BouRabee10} (see also \cite{Rivin}) introduced a function $\Farb_{\Gamma,S} (n)$ on the natural numbers that quantifies residual finiteness. Specifically, the value of $\Farb_{\Ga,S}
(n)$ is the maximum order of a finite group needed to distinguish a non-identity element from the identity as one varies over non-identity elements in the $n$-ball. Numerous authors have studied the asymptotic behavior of $\Farb_{\Gamma,S}(n)$ for a wide collection of groups $\Gamma$ (see \cite{BouRabee11,BouRabee10,BK,brhp,buskin,KM,Patel,Rivin}). 

To state our results, we require some notation. For two non-decreasing functions $\map{f,g}{\mathbb{N}}{\mathbb{N}}$, we write $f \preceq g$ if there exists a $C \in \N$ such that $f (n) \leq C g (C n)$ for all $n \in \N$.  We write $f \approx g$ when $f \preceq g$ and $g \preceq f$. For a finitely generated nilpotent group $\Ga$, we denote $T(\Ga)$ to be the normal subgroup of finite order elements. As we will see in Subsection \ref{word_defn}, the dependence of $\Farb_{\Ga,S}(n)$ is mild; subsequently, we will suppress the dependence of $\Farb_{\Ga}$ on the generating subset in this subsection.

For finitely generated nilpotent groups, Bou-Rabee \cite[Thm 0.2]{BouRabee10} proved that $\Farb_{\Ga} (n) \preceq (\log (n))^{h (\Gamma)}$ where $h (\Gamma)$ is the Hirsch length of $\Gamma$. Our first result establishes the precise asymptotic behavior of $\Farb_{\Gamma} (n)$.

\begin{thm}\label{word_precise} Let $\Ga$ be an infinite, finitely generated nilpotent group. There exists a $\psi_{\text{RF}}(\Ga) \in \N$ such that $\Farb_{\Ga}(n) \approx \pr{\log(n)}^{\psi_{\text{RF}}(\Ga)}$. Additionally, one can compute $\psi_{\text{RF}}(\Ga)$ given a basis for $(\Ga / T(\Ga))_c$ where $c$ is the step length of $\Ga / T(\Ga)$.
\end{thm}	
	
The proof of Theorem \ref{word_precise} is done in two steps. To establish the upper bound, we appeal to some structural properties of finitely generated nilpotent groups. To establish the lower bound, we construct a sequence $\set{\gamma_i} \subseteq \Gamma$ such that the order of the minimal finite group that separates $\gamma_i$ from the identity is bounded below by $C (\log (C  \|\gamma_i\|_S))^{\psi_{\text{RF}}(\Ga)}$ for some $C \in \N$.	
	
The following is a consequence of the proof of Theorem \ref{word_precise}.
\begin{cor}\label{cong_rf}
	Let $\Ga$ be a finitely generated nilpotent group. Then $\Farb_{\Ga}(n) \approx \pr{\log(n)}^{h(\Ga)}$ if and only if $h(Z(\Ga / T(\Ga))) =1$.
\end{cor}	
	
We now introduce some terminology. Suppose that $G$ is a connected, simply connected nilpotent Lie group with Lie algebra $\Fr{g}$. We say that $G$ is \emph{$\Q$-defined} if  $\Fr{g}$ admits a basis with rational structure constants. The \emph{Mal'tsev} completion of a torsion free, finitely generated nilpotent group $\Ga$ is a connected, simply connected, $\Q$-defined nilpotent Lie group $G$ such that $\Ga$ embeds into as a cocompact lattice.

The next theorem demonstrates that the asymptotic behavior of $\Farb_{\Ga}(n)$ is an invariant of the Mal'tsev completion of $\Ga / T(\Ga)$.
 
\begin{thm}\label{maltsev_invariant}
	Suppose that $\Ga_1$ and $\Ga_2$ are two infinite, finitely generated nilpotent groups such that  $\Ga_1 / T(\Ga_1)$ and $\Ga_2 / T(\Ga_2)$ have isomorphic Mal'tsev completions. Then $\Farb_{\Ga_1}(n) \approx \Farb_{\Ga_2}(n)$.
\end{thm}

The proof of theorem \ref{maltsev_invariant} follows by an examination of a cyclic series that comes from a refinement of the upper central series and its interaction with the topology of the Mal'tsev completion.

Since the $3$-dimensional integral Heisenberg group embeds into every infinite, non-abelian nilpotent group, Theorem \ref{word_precise}, Theorem \ref{maltsev_invariant}, \cite[Thm 2.2]{BouRabee10}, and \cite[Cor 2.3]{BouRabee10} allow us to characterize $\R^d$ within the collection of connected, simply connected $\Q$-defined nilpotent Lie groups by the asymptotic behavior of residual finiteness of a cocompact lattice.
\begin{cor}\label{abelian_rf}
	Let $G$ be a connected, simply connected, $\Q$-defined nilpotent Lie group. Then $G$ is Lie isomorphic to $\R^{\text{dim}(G)}$ if and only if $\Farb_{\Ga}(n) \precnsim (\log (n))^3$ where $\Ga \subseteq G$ is any cocompact lattice.
\end{cor}

For the last result of this section, we need some terminology. We say that a group $\Ga$ is \emph{irreducible} if there is no non-trivial splitting of $\Ga$ as a direct product. For a function of the form $f(n) = (\log(n))^m$, we call $m$ the \emph{polynomial in logarithm degree of growth} for $f(n)$.

\begin{thm}\label{applications 2} \text{ } \begin{itemize}
\item[(i)] For $c \in \N$, there exists $m(c) \in \N$ satisfying the following. For each $\ell \in \N$ there exists an irreducible, torsion free, finitely generated nilpotent group $\Ga$ of step length $c$ and $h(\Ga) \geq \ell$ such that $\Farb_{\Ga}(n) \preceq (\log (n))^{m(c)}$. 

\item[(ii)] Every natural number not equal to $2$ can be realized as the polynomial in logarithm degree of growth for $\Farb_{\Ga}(n)$ where $\Ga$ is an irreducible, torsion free, finitely generated nilpotent group. 
\item[(iii)]
Suppose $2 \leq c_1 < c_2$ are natural numbers. For each $\ell \in \N$, there exist irreducible, torsion free, finitely generated nilpotent groups $\Ga_\ell$ and $\De_\ell$ of step lengths $c_1$ and $c_2$, respectively, such that $\Farb_{\Ga_\ell}(n) \approx \Farb_{\De_\ell}(n)$.

\item[(iv)] For natural numbers $c > 1$ and $m \geq 1$, there exists an irreducible, torsion free, finitely generated nilpotent group $\Ga$ of step length $c$ such that $(\log(n))^m \preceq \Farb_{\Ga}(n)$.
\end{itemize}
\end{thm}

For Theorem \ref{applications 2}(i), we consider free nilpotent groups of a fixed step length and increasing rank. We make use of central products of filiform nilpotent groups for Theorem \ref{applications 2}(ii) - (iv).

Using Theorem \ref{applications 2}, we are able to relate the polynomial in logarithm degree of growth of the residual finiteness function with well known invariants of the class of finitely generated nilpotent groups. Theorem \ref{applications 2}(i) implies the polynomial in logarithm degree of growth of $\Farb_{\Ga}(n)$ does not depend on the Hirsch length of $\Ga$. Similarly, Theorem \ref{applications 2}(iv) implies there is no upper bound in terms of step length of $\Ga$ for the polynomial in logarithm degree of growth of $\Farb_{\Ga}(n)$. On the other hand, the step size of $\Ga$ is not determined by the polynomial in logarithm growth of $\Farb_{\Ga}(n)$ as seen in Theorem \ref{applications 2}(iii). 

\subsection{Conjugacy Separability}
We now turn our attention to effective conjugacy separability. Lawton--Louder--McReynolds \cite{LLM} introduced a function $\Conj_{\Gamma,S}(n)$ on the natural numbers that quantifies conjugacy separability. To be precise, the value of $\Conj_{\Ga,S}(n)$ is the maximum order of the minimal finite quotient needed to separate a pair of non-conjugate elements as one varies over non-conjugate pairs of elements in the $n$-ball. Since the dependence of $\Conj_{\Ga,S}(n)$ on $S$ is mild (see Lemma \ref{conj_gen_set}), we will suppress the generating subset throughout this subsection.

To the author's knowledge, the only previous work on the asymptotic behavior of $\Conj_{\Ga}(n)$ is due to Lawton--Louder--McReynolds \cite{LLM}. They demonstrate that if $\Ga$ is a surface group or a finite rank free group, then $\Conj_{\Ga}(n) \preceq n^{n^2}$ \cite[Cor 1.7]{LLM}. In this subsection, we initiate the study of the asymptotic behavior of $\Conj_{\Ga}(n)$ for the collection of finitely generated nilpotent groups.

Our first result is the precise asymptotic behavior of $\Conj_{\Heis_{2m+1}(\Z)}(n)$ where $\Heis_{2m+1}(\Z)$ is the $(2m+1)$-dimensional integral Heisenberg group.

\begin{thm}\label{precise_heisenberg_calc}
$\Conj_{\Heis_{2m+1}(\Z)} (n) \approx n^{2m+1}$.
\end{thm}

For general nilpotent groups, we establish the following upper bound for $\Conj_{\Ga}(n)$.

\begin{thm}\label{upper}
Let $\Ga$ be a finitely generated nilpotent group. Then $\Conj_{\Ga}(n) \preceq n^{k}$ for some $k \in \N$.
\end{thm}

Blackburn \cite{Blackburn} was the first to prove conjugacy separability of finitely generated nilpotent groups. Our strategy for proving Theorem \ref{upper} is to effectivize \cite{Blackburn}.

For the same class of groups, we have the following lower bound which allows us to characterize virtually abelian groups within the class of finitely generated nilpotent groups. Moreoever, we obtain the first example of a class of groups for which the asymptotic behavior of $\Farb_{\Ga}(n)$ and $\Conj_{\Ga}(n)$ are shown to be dramatically different.

\begin{thm}\label{lower}
Let $\Ga$ be a finitely generated nilpotent group. \begin{itemize}
	\item[\textit{(i)}] If $\Gamma$ contains a normal abelian subgroup of index $m$, then $\log (n) \preceq \Conj_{\Gamma}(n) \preceq (\log (n))^{m}$.
	\item[\textit{(ii)}] Suppose that $\Gamma$ is not virtually abelian. Then there exists a $\psi_{\text{Lower}}(\Ga) \in \N$ such that $n^{\psi_{\text{Lower}}(\Ga)}  \preceq \Conj_{\Ga} (n).$ Additionally, one can compute $\psi_{\text{Lower}}(\Ga)$ given a basis for $(\Ga / T(\Ga))_c$ where $c$ is the step length of $\Ga / T(\Ga)$.
\end{itemize} 
\end{thm}

The proof of Theorem \ref{lower}(i) is elementary. We prove Theorem \ref{lower}(ii) by finding an infinite sequence of non-conjugate elements $\set{\gamma_i,\eta_i}$ such that the order of the minimal finite group that separates the conjugacy classes of $\gamma_i$ and $\eta_i$ is bounded below by $C \: n_i^{\psi_{\text{Lower}}(\Ga)}$ for some $C \in \N$ where $\|\ga_i\|_S,\|\eta_i\|_S \approx n_i$ for some finite generating subset $S$.

We have the following theorem which is similar in nature to Theorem \ref{maltsev_invariant}.
\begin{thm}\label{conj_malsev_invariant}
	Let $\Ga$ and $\De$ be infinite, finitely generated nilpotent groups of step size greater than or equal to $2$, and suppose  that $\Ga/ T(\Ga)$ and $\De / T(\De)$ have isomorphic Mal'tsev completions. Then $n^{\psi_{\text{Lower}}(\Ga)} \preceq \Conj_{\De}(n)$ and $n^{\psi_{\text{Lower}}(\De)} \preceq \Conj_{\Ga}(n).$
\end{thm}

We apply Theorem \ref{lower} to construct nilpotent groups that help demonstrate the various asymptotic behaviors that the growth of conjugacy separability may exhibit.
\begin{thm}\label{applications}
For natural numbers $c >1$ and $k \geq 1$, there exists an irreducible, torsion free, finitely generated nilpotent group $\Ga$ of step length $c$ such that $n^k \preceq \Conj_{\Ga}(n)$. 
\end{thm}
Theorem \ref{applications} implies that the conjugacy separability  function does not depend of the step length of the nilpotent group. We consider central products of filiform nilpotent groups for Theorem \ref{applications}.

\subsection{Acknowledgements}
I would like to thank my advisor Ben McReynolds for all his help and my mentor Priyam Patel for her guidance. I am indebted to Karel Dekimpe for his suggestions for Proposition \ref{important_estimate}. I also want to thank Khalid Bou-Rabee for looking at earlier drafts of this article and for his many insightful comments. Finally, I want to thank Alex Barrios, Rachel Davis, Jonas Der\'e, Artur Jackson, Gijey Gilliam, Brooke Magiera, and Nick Miller for conversations on this article.

\section{Background}
%
We will assume the reader is familiar with finitely generated groups, Lie groups and Lie algebras. 
\subsection{Notation and conventions}
We let $\lcm \set{r_1,\hdots, r_m}$ be the lowest common multiple of $\{r_1,\cdots,r_m\} \subseteq \Z$  with the convention that $\lcm(a) = |a|$ and $\lcm (a,0) = 0$. We let $\gcd(r_1,\cdots,r_m)$ be the greatest common multiple of $\{r_1,\cdots,r_m\} \subseteq \Z$ with the convention that $\gcd(a,0) = |a|$.
 
We denote $\|\ga\|_S$ as the word length of $\ga$ with respect to $S$ and denote the identity of $\Ga$ as $1$. We denote the order of $\ga$ as an element of $\Ga$ as $\text{Ord}_\Ga (\ga)$. We write $\ga \sim \eta$ when there exists a $g \in \Ga$ such that $g^{-1} \: \ga \: g = \eta$. For a normal subgroup $\De \nsub \Ga$, we set $\map{\pi_\De}{\Ga}{\Ga / \De}$ to be the natural projection and write $\bar{\ga} = \pi_\De(\ga)$ when $\De$ is clear from context. For a subset $X \subseteq \Ga$, we denote $\innp{X}$ to be the subgroup generated by $X$.

We define the commutator of $\ga$ and $\eta$ as $[\ga,\eta] = \ga^{-1} \: \eta^{-1} \: \ga \: \eta$. We denote the $m$-fold commutator of $\{\ga_i\}_{i=1}^m \subseteq \Ga$  as $[\ga_1,\hdots,\ga_m]$ with the convention that $[\ga_1,\hdots,\ga_m] = [[\ga_1,,\hdots, 
\ga_{m-1}],\ga_m]$.
 
We denote the center of $\Ga$ as $Z(\Ga)$ and the centralizer of $\ga$ in $\Ga$ as $C_{\Ga}(\ga)$. We define $\Ga_i$ to be the $i$-th term of the lower central series and $Z^i(\Ga)$ to be the $i$-th term of the upper central series. For $\ga \in \Ga - \set{1}$, we denote $\Height(\ga)$ as the minimal $j \in \N$ such that $\pi_{Z^{j-1}(\Ga)}(\ga) \neq 1$.

We define the abelianization of $\Ga$ as $\Ga_{\text{ab}}$ with the associated projection given by $\pi_{\text{ab}} = \pi_{[\Ga,\Ga]}$. For $m \in \N$, we define $\Ga^m \cong \innp{\ga^m \: | \: \ga \in \Ga}$  and denote the associated projection as $\sigma_m = \pi_{\Ga^m}$.

When given a basis $X = \{X_i\}_{i=1}^{\text{dim}_\R (\Fr{g})}$ for $\Fr{g}$, we denote $\|\sum_{i=1}^{\text{dim}_\R (\Fr{g})} \al_i \: X_i \|_X = \sum_{i=1}^{\text{dim}_\R (\Fr{g})} |\al_i|$. For a Lie algebra $\Fr{g}$ with a Lie ideal $\Fr{h}$, we define $\map{\pi_{\Fr{h}}}{\Fr{g}}{\Fr{g} / \Fr{h}}$ to be the natural Lie projection.

For a $\R$-Lie algebra $\Fr{g}$, we denote $Z(\Fr{g})$ to the center of $\Fr{g}$, $\Fr{g}_i$ to be the $i$-th term of the lower central series, and $Z^i(\Fr{g})$ to be the $i$-th term of the upper central series.

For $A \in \Fr{g}$, we define the map $\map{\ad_A}{\Fr{g}}{\Fr{g}}$ to be given by $\ad_A(B) = [A,B]$. We denote the $m$-fold Lie bracket of $\set{A_i}_{i=1}^m \subseteq \Fr{g}$ as $[A_1, \cdots A_m]$ with the convention that $[A_1, \cdots, A_m] = [[A_1,\cdots,A_{m-1}], A_m]$.
                                      
\subsection{Finitely generated groups and separability}

\subsubsection{Residually finite groups}\label{word_defn}
Following \cite{BouRabee10} (see also \cite{Rivin}), we define the depth function of $\Ga$ as $\map{\Depth_\Ga}{\Ga - \set{1}}{\N \cup \set{\infty}}$ for finitely generated groups to be given by
$$
\Depth_{\Gamma} (\gamma) \bdef \text{min}\set{|Q| \: | \: \: \map{\varphi}{\Gamma}{Q}, |Q| < \infty, \text{ and } \varphi (\gamma) \neq 1}.
$$
We define
$
\map{\Farb_{\Gamma,S}}{\mathbb{N}}{\mathbb{N}}
$
by
$
\Farb_{\Gamma,S} (n) \bdef \text{max}\set{\Depth_{\Gamma} (\ga) \: | \: \|\ga\|_S \leq n \text{ and } \ga \neq 1}.$ When $\Ga$ is a residually finite group, then $\Farb_{\Ga,S}(n) < \infty$ for all $n \in \N$. For any two finite generating subsets $S_1$ and $S_2$, we have 
$\Farb_{\Gamma,S_1} (n) \approx \Farb_{\Gamma,S_2} (n)$ (see \cite[Lem 1.1]{BouRabee10}). Therefore, we will suppress the choice of finite generating subset.

\subsubsection{Conjugacy separable groups}\label{conjugacy}
Following \cite{LLM}, we define the conjugacy depth function of $\Ga$ as
$\map{\CDepth_{\Gamma}}{\Gamma \times \Gamma - \set{(\ga,\eta) \: | \: \ga \sim \eta} }{\mathbb{N} \cup \set{\infty}}$ to be given by
$$
\CDepth_{\Gamma} (\gamma, \eta) \bdef
\text{min}\set{
|Q| \: |  \: \: \map{\varphi}{\Gamma}{Q}, |Q| < \infty, \text{ and } \varphi (\gamma) \nsim \varphi (\eta)}.
$$
We define $
\map{\Conj_{\Gamma,S}(n)}{\N}{\N}
$
 as
$
\Conj_{\Gamma,S} (n) \bdef \text{max} \set{\CDepth_{\Gamma} (\gamma, \eta) \: | \: \ga \nsim \eta \text{ and } \|\ga\|_S,\|\eta\|_S \leq n }.
$ When $\Ga$ is a conjugacy separable group, then $\Conj_{\Ga,S}(n) < \infty$ for all $n \in \N$.

\begin{lemma}\label{conj_gen_set}
If $S_1,S_2$ are two finite generating subsets of $\Ga$,  then $\Conj_{\Gamma,S_1} (n) \approx \Conj_{\Gamma,S_2} (n) .$
\end{lemma} The proof is similar to \cite[Lem 1.1]{BouRabee10} (see also \cite[Lem 2.1]{LLM}). As before, we will suppress the choice of finite generating subset.

\cite[Lem 2.1]{LLM} implies that the order of minimal finite group required to separate a non-identity element $\ga \in \Ga$ from the identity is bounded above by order of the minimal finite group required to separate the conjugacy class of $\ga$ from the identity. Thus, $\Farb_{\Ga}(n) \preceq \Conj_{\Ga}(n)$ for all conjugacy separable groups. In particular, if $\Ga$ is conjugacy separable, then $\Ga$ is residually finite.
\subsection{Nilpotent groups and nilpotent Lie groups}
See \cite{Dekimpe,Hall_notes,Knapp,polycyclic} for a more thorough account of the material in this subsection. Let $\Gamma$ be a non-trivial, finitely generated group. The $i$-th term of the \emph{lower central series} is defined by $
\Ga_1 \bdef \Ga$ and for $i>1$ as $\Ga_i \bdef [\Ga_{i-1}, \Ga]$. The $i$-term of the \emph{upper central series} is defined by $\Ga^0 \bdef \set{1}$ and $Z^{i}(\Ga) \bdef  \pi_{Z^{i-1}(\Ga)}^{-1}(Z (\Ga / Z^{i-1}(\Ga)))$ for $i > 1$.

\begin{defn}
We say that $\Ga$ is a \emph{nilpotent group of step size $c$} if $c$ is the minimal natural number such that $\Ga_{c + 1} = \set{1}$, or equivalently, $Z^{c}(\Ga) = \Ga$. If the step size is unspecified, we simply say that $\Ga$ is a nilpotent group. We say that finitely generated nilpotent group is an \emph{admissible} group.
\end{defn}
 
For an admissible group $\Ga$,  the set of finite order elements of $\Ga$, denoted as $T (\Ga)$, is a finite order characteristic subgroup. Moreover, when $|\Ga| = \infty$, then $\Ga / T(\Ga)$ is torsion free.

Let $\Fr{g}$ be a non-trivial, finite dimensional $\R$-Lie algebra. The $i$-th term of the \emph{lower central series} of $\Fr{g}$ is defined by $\Fr{g}_1 \bdef \Fr{g}$ and for $i > 1$ as $\Fr{g}_i \bdef [\Fr{g}_{i-1}, \Fr{g}]$. We define the $i$-th term of the \emph{upper central series} as $Z^0(\Fr{g}) \bdef \set{0}$ and $Z^i(\Fr{g}) \bdef \pi_{Z^{i-1}(\Fr{g})}(Z(\Fr{g} / Z^{i-1}(\Fr{g})))$ for $i>1$. 

\begin{defn}
We say that $\Fr{g}$ is a \emph{nilpotent Lie algebra} of step length $c$ if $c$ is the minimal natural number satisfying $\Fr{g}^{c} = \Fr{g}$, or equivalently, $\Fr{g}_{c + 1} = \set{0}$. If the step size is unspecified, we simply say that $\Fr{g}$ is a nilpotent Lie algebra.\end{defn}

For a connected, simply connected nilpotent Lie group $G$ of step length $c$ with Lie algebra $\mathfrak{g}$, the exponential map, written as $\map{\exp}{\mathfrak{g}}{G}$, is a diffeomorphism \cite[Thm 1.127]{Knapp} whose inverse is formally denoted as $\Log$. The Baker-Campbell-Hausdorff formula \cite[(1.3)]{Dekimpe} implies that every $A,B \in \Fr{g}$ satisfies
\begin{equation}
A * B \bdef \Log (\exp A \cdot \exp B) \bdef A + B + \frac{1}{2} [  A, B ] + \sum_{m=3}^{\infty} CB_m(A,B)
\end{equation}
where $CB_m(A,B)$ is as rational linear combination of $m$-fold Lie brackets of $A$ and $B$. By assumption, $CB_m(A,B) = 0$ for $m > c$. For $\{A_i\}_{i=1}^m$ in $\Fr{g}$, we may more generally write
\begin{equation}\label{Baker}
A_1 * \cdots * A_m = \Log (\exp A_1  \cdots \exp A_m) = \sum_{i=1}^{c} CB_{i} (A_1, \cdots, A_m)
\end{equation}
where $CB_{i} (A_1, \cdots, A_m)$ is a rational linear combination of $i$-fold Lie brackets of $\{A_{j_t}\}_{t=1}^\ell \subseteq \set{A_i}_{i=1}^m$ via repeated applications of the Baker-Campbell-Hausdorff formula.

We define the adjoint representation $\map{\Ad}{G}{\Aut (\mathfrak{g})}$ of $G$ as $\Ad(g)(X) = (d\Psi_g)_1(X)$ where $\Psi_{g}(x) = g \: x \: g^{-1}$. By \cite[1.92]{Knapp}, we may write for $\ga \in \Ga$ and $A \in \Fr{g}$
\begin{equation}\label{baker_adjoint}
\Ad(\ga)(A)  = A + \frac{1}{2}[\Log(\ga), A] +  \sum_{i=3}^{c} \frac{(\ad_{\Log (\ga)})^i (A)}{ i!}.
\end{equation}

By \cite{Mal01}, a connected, simply connected nilpotent Lie group $G$ with Lie algebra $\Fr{g}$ admits a cocompact lattice $\Gamma$ if and only if $\mathfrak{g}$ admits a basis $\set{X_i}_{i=1}^{\text{dim}(G)}$ with rational structure constants (see \cite[Thm 7]{complete} for more details). We say $G$ is \emph{$\Q$-defined} if it admits a cocompact lattice. For any torsion free admissible group $\Gamma$, \cite[Thm 6]{complete} implies that there exists a $\mathbb{Q}$-defined group unique up to isomorphism in which $\Gamma$ embeds as a cocompact lattice. 
\begin{defn} We call this $\Q$-defined group the \emph{Mal'tsev completion} of $\Ga$. When given a $\Q$-defined group $G$, the tangent space at the identity with the Lie bracket of vector fields is a finite dimensional, nilpotent Lie algebra. We say that a connected, simply connected nilpotent Lie group is an \emph{admissible} Lie group.
\end{defn}

\subsection{Polycyclic groups}\label{polycylic_section_title}
See \cite{computational_group,rag,polycyclic} for the material contained in the following subsection.
\begin{defn}
A group $\Ga$ is \emph{polycyclic} if there exists an ascending chain of subgroups $\set{\De_i}_{i=1}^{m}$ such that $\De_1$ is cyclic, $\De_i \trianglelefteq \De_{i+1}$, and $\De_{i+1} / \De_{i}$ is cyclic for all $i$. We call $\set{\De_i}_{i=1}^{m}$ a \emph{cyclic series} for $\Ga$. We say $\set{\xi_i}_{i=1}^{m}$ is a \emph{compatible generating subset} with respect to $\set{\De_i}_{i=1}^m$ if $\innp{\xi_1} = \De_1$ and $\innp{\xi_{i+1}, \De_{i}} = \De_{i+1}$ for $i > 1$. We define the \emph{Hirsch length} of $\Ga$, denoted as $h(\Ga)$, as the number of indices $i$ such that $|\De_{i+1} : \De_{i}| = \infty$. 
\end{defn}

For a general polycyclic group, there may be infinitely many different cyclic series of arbitrary length (see \cite[Ex 8.2]{computational_group}). However, the Hirsch length of $\Ga$ is independent of the choice of cyclic series. With respect to the compatible generating subset $\set{\xi_i}_{i=1}^m$, \cite[Lem 8.3]{computational_group} implies that we may represent every $\gamma \in \Ga$ uniquely as $\gamma = \prod_{i=1}^{m} \xi_i^{\al_i}$ where $\al_i \in \Z$ if $|\De_{i+1} : \De_{i}| = \infty$ and $0 \leq \al_i < r_i$ if $|\De_{i+1} : \De_{i}| = r_i$. If $|\Ga| < \infty$, then the second paragraph after \cite[Defn 8.2]{computational_group} implies that $|\Ga| = \prod_{i=1}^{m} r_{i}$.

\begin{defn}\label{mal_coord}
We call the collection of such $m$-tuples a \emph{Mal'tsev basis} for $\Ga$ with respect to the compatible generating subset $\set{\xi_i}_{i=1}^{m}$. We call $(\al_i)_{i=1}^{m}$ the \emph{Mal'tsev coordinates} of $\ga$.
\end{defn}

For an admissible group $\Ga$, we may refine the upper central series of $\Gamma$ to obtain a cyclic series and compatible generating subset. First assume that $\Ga$ is abelian.  We may write $\Ga \cong \Z^m \oplus T(\Ga)$, and we let $\{\xi_i\}_{i=1}^{h(\Ga)}$ be a free basis for $\Z^m$. Since $T(\Ga)$ is a finite abelian group, there exists an isomorphism $\map{\varphi}{T(\Ga)}{\oplus_{i=1}^\ell \Z / p_i^{k_i} \Z}$. If $x_i$ generates $\Z / p_i^{k_i} \Z$ in $\oplus_{i=1}^\ell \Z / p_i^{k_i} \Z$, we then set $\xi_i = \varphi^{-1}(x_{i - h(\Ga)})$ for $h(\Ga) + 1 \leq i \leq h(\Ga) + \ell$. Thus, the groups $\{\De_i\}_{i=1}^{h(\Ga)+\ell}$ given by $\De_i = \innp{\xi_t}_{t=1}^i$ form a cyclic series for $\Ga$ with a compatible generating subset $\{\xi_i\}_{i=1}^{h(\Ga) + \ell}$.

Now assume that $\Ga$ has step length $c > 1$. There exists a generating basis $\set{z_i}_{i=1}^{h(\Ga)}$ for $Z(\Ga)$ and integers $\set{t_i}_{i=1}^{h(\Ga)}$ such that $\set{z_i^{t_i}}_{i=1}^{h(\Ga_c)}$ is a basis for $\Ga_c$ and for each $i$ there exist $x_i \in \Ga_{c-1}$ and $y_i \in \Ga$ such that $z_i^{t_i} = [x_i,y_i]$. We may choose a cyclic series $\set{H_i}_{i=1}^{h(\Ga)}$ such that $H_i = \innp{z_s}_{s=1}^i$. Induction implies that there exists cyclic series $\{\La_i\}_{i=1}^{k}$ and compatible generating subset $\{\la_i\}_{i=1}^{k}$ for $\Ga / Z(\Ga).$ For $1 \leq i \leq \ell$, we set $\De_i = H_i$, and for $\ell + 1 \leq i \leq \ell + k$, we set $\De_i = \pi_{Z(\Ga)}^{-1}(\La_{i-\ell})$. For $1 \leq i \leq \ell$, we set $\xi_i = z_i$. For $\ell + 1 \leq i \leq \ell + k$, we choose a $\xi_i$ such that $\pi_{Z(\Ga)}(\xi_i) = \la_{i - \ell}$. It then follows that $\{\De_i\}_{i=1}^{\ell + k}$ is a cyclic series for $\Ga$ with a compatible generating subset given by $\{\xi_i\}_{i=1}^{\ell+ k}$. Moreover, the given construction implies that $h (\Ga ) = \sum_{i=1}^{c} \text{rank}_{\mathbb{Z}} \: (Z^{i}(\Gamma) / Z^{i-1}(\Gamma))$. Whenever $\Ga$ is an admissible group, we choose the cyclic series and compatible generating subset this way. 

\begin{defn}
	An admissible $\Ga$ with a cyclic series $\set{\De_i}$ and a compatible generating subset $\set{\xi_i}$ is called an admissible $3$-tuple and is denoted as $\set{\Ga,\De_i,\xi_i}$.
\end{defn}

Let $\set{\Ga,\De_i,\xi_i}$ be an admissible $3$-tuple such that $\Ga$ is torsion free. \cite[Thm 6.5]{Hall_notes} implies that multiplication of $\ga,\eta \in \Ga$ can be expressed as polynomials in terms of the Mal'tsev basis associated to the cyclic series $\set{\De_i}_{i=1}^{h(\Ga)}$ and the compatible generating subset $\set{\xi_{i}}_{i=1}^{h(\Ga)}$. Specifically, we may write $\ga \: \eta = (\prod_{i=1}^{h (\Ga)} \xi_i^{a_i}) \cdot (\prod_{j=1}^{h (\Ga)}\xi_j^{b_j}) = \prod_{s=1}^{h(\Ga)}\xi_s^{d_s}$ where each $d_s$ can be expressed as a polynomial in the Mal'stev coordinates of $\ga$ and $\eta$, respectively. Similarly, we may write $\ga^\ell = (\prod_{i=1}^{h(\Ga)}\xi_{i}^{a_i})^\ell = \prod_{j=1}^{h (\Ga)}\xi^{e_j}$ where each $e_j$ can be expressed as a polynomial in the Mal'tsev coordinates of $\ga$ and the integer $\ell$.

The polynomials that define the group product and group power operation of $\Ga$ with respect to the given cyclic series and compatible generating subset have unique extensions to $\R^{h(\Ga)}$. That implies $G$ is diffeomorphic to $\R^{h(\Ga)}$ (see \cite[Thm 6.5]{Hall_notes}, \cite[Cor 1.126]{Knapp}) where $G$ is the Mal'tsev completion of $\Ga$. Consequently, the dimension and step length of $G$ are equal to the Hirsch length and step length of $\Gamma$, respectively. Thus, we may write $h(\Gamma) = \text{dim}(G)$. Furthermore, we may identify $\Ga$ with its image in $G_\Ga$ which is the set $\Z^{h(\Ga)}$. 

The following definition will be of use for the last lemma of this subsection.
\begin{defn}
	Let $\Ga$ be a torsion free admissible group, and let $\De \leq \Ga$ be a subgroup. We define the \emph{isolator of $\De$ in $\Ga$} as the subgroup given by
	$$
	\sqrt[\Ga]{\De} = \{ \ga \in \Ga \: | \: \text{ there exists } \: k \in \N \: \text{ such that } \: \ga^k \in \De \}\cup \set{1}.
	$$ 
\end{defn}
We make some observations. By the paragraph proceeding exercise 8 of \cite[Ch 8]{polycyclic} and \cite[Thm 4.5]{Hall_notes}, $\sqrt[\Ga]{\De}$ is a subgroup such that $|\sqrt[\Ga]{\De} : \De| < \infty$. If $\Ga$ is abelian, then we may write
$
\Ga = (\Ga / \sqrt[\Ga]{\De}) \oplus \sqrt[\Ga]{\De}.
$
We also have for any subgroup $\De \nsub \Ga$ that $\Ga / \sqrt[\Ga]{\De}$ is torsion free.

We finish this section with the following result. When given a torsion free admissible group $\Ga$, the following lemma  relates the word length of element $\ga$ in $\Ga$ to the Mal'tsev coordinates of $\ga$ with respect to a choice of a cyclic series and a compatible generating subset.
\begin{lemma}\label{coord_bound}
	Let $\set{\Ga,\De_i,\xi_i}$ be an admissible $3$-tuple such that $\Ga$ has step length $c$. If $\ga \in \Ga$ such that $\|\ga\|_S \leq n$, then $|\al_i| \leq C \: n^{c}$ where $(\al_i)$ are the Mal'tsev coordinates of $\ga$ for some $C \in \N$.
\end{lemma}
\begin{proof}
	If $\Ga$ is finite, then the statement is evident. Thus, we may assume that $|\Ga| = \infty$. We proceed by induction on step length, and observe that the base case of abelian groups is clear. Now suppose that $\Ga$ has step length $c > 1$ and that $\|\ga\|_S \leq n$. Since $\|\pi_{\Ga_i}(\ga)\|_{\pi_{\Ga_i}(S)} \leq n$, the inductive hypothesis implies there exists a constant $C_0 > 0$ such that $|\al_t| \leq C_0 \: n^{t}$ when $\pi_{\Ga_t}(\xi_i) \neq 1$. We let $m$ be the smallest integer such that $\xi_i \notin \Ga_c$ when $i > m$ and $\xi_i \in \Ga_c$ for $i \leq m$, and let $k$ be the length of the cyclic series $\De_i$.
	
	Suppose that $\xi_i \notin \Ga_{t+1}$, but $\xi_i \in \Ga_t$ where $t < c$. We have by assumption that $|\al_i| \leq C_0 \: n^t$. By \cite[3.B2]{asymptotic_group}, we have $\|\xi_i^{\al_i}\|_S \approx |\al_i|^{1/t}$. Thus, there exists a constant $C_1$ such that $\|\xi_i^{\al_i}\|_S \leq C_1 |\al_i|^{1/t}$ when $\text{Ord}_\Ga(\xi_i) = \infty$. Therefore, we may write $\|\xi_i^{\al_i}\|_S \leq  C_0 \: C_1^{1/t} \: n$. Thus, for $\xi_i \notin \Ga_{c}$ such that $\text{Ord}_\Ga(\xi_{i}) = \infty$, we have $\|\xi_i^{\al_i}\|_S \leq C_2 \: n$ for some $C_2 \in \N$. By increasing $C_2$, we may assume that $|C_2| > |T(\Ga)|$, and thus, we cover the cases where $\text{Ord}_\Ga(\xi_i) < \infty$.
	
	Letting $\zeta = (\prod_{i=m + 1}^k \xi_i^{\al_i})^{-1}$ where $\xi_i \in \Ga_{c}$, it is evident that $\|\zeta\|_S \leq C_3 \: n$ for some $C_3 \in \N$. Thus, if $\eta = \prod_{i=1}^{m}\xi_i^{\al_i}$, we may write $\| \eta \|_S =\|\ga \: \zeta \|_S  \leq \|\ga \|_S + \| \zeta\|_S \leq 2 \: C_0 \: n$. Hence, we may assume that $\ga \in \Ga_{c}$.
	
	We may write $\ga = \prod_{i=1}^{m} \xi_i^{\al_i}$ where $\xi_i \in \Ga_{c}$. If we let $\la_t = \prod_{i=1,i \neq t}^{m}\xi_i^{-\al_i}$, we may write $\|\xi_t^{\al_t}\|_{S} = \|\zeta \: \la_t\|_S \leq \|\zeta\|_S + \|\la_t\|_S \leq 4 \: C_0 \: n$. Thus, we need only consider when $\ga = \xi_i^{\al_i}$ for $\xi_i \in \Ga_{c}$.  Since $\xi_i \in \Ga_{c}$, \cite[3.B2]{asymptotic_group} implies that $|\al_i| \leq C_3 \: n^{c}$ for some $C_3 \in \N$.  \end{proof}

\section{Torsion-free Quotients with One Dimensional Centers}
In the following subsection, we define what a choice of an admissible quotient with respect to a central non-trivial element is, what a choice of a maximal admissible quotient is, and define the constants $\psi_{\text{RF}}(\Ga)$ and $\psi_{\text{Lower}}(\Ga)$ for an infinite admissible group $\Ga$.

\subsection{Existence of torsion free quotients with one dimensional center}
The following proposition will be useful throughout this article. 
\begin{prop}\label{one_dim_center}
Let $\Ga$ be a torsion free admissible group, and suppose that $\ga$ is a central, non-trivial element. There exists a normal subgroup $\La$ in $\Ga$ such that $\Ga / \La$ is an irreducible, torsion free, admissible group such that $\innp{\pi_{\La}(\ga)}$ is a finite index subgroup of $Z(\Ga / \La)$. If $\ga$ is primitive, then $Z(\Ga / \La) \cong \innp{\pi_{\La}(\ga)}$.
\end{prop}
\begin{proof}
We construct $\La$ by induction on Hirsch length, and since the base case is trivial, we may assume that $h  (\Gamma) > 1$. If $Z (\Gamma) \cong \Z$, then the proposition is now evident by letting $\La = \{1\}$. 

Now assume that $h(Z(\Ga)) \geq 2$. There exists a basis $\{z_i\}_{i=1}^{h(Z(\Ga))}$ for $Z(\Ga)$ such that $z_1^k = \ga$ for some $k \in \Z - \set{0}$. Letting $K = \innp{z_i}_{i=2}^{h (Z(\Ga))}$, we note that $K \nsub \Ga$ and $\pi_K (\ga) \neq 1$. Additionally, it follows that $\Ga / K$ is a torsion free, admissible group. If $h(Z(\Ga / K)) = 1$, then our proposition is evident by defining $\La \cong K$.

Now suppose that $h(Z(\Ga / K)) \geq 2$. Since $h(\Ga / K) < h(\Ga)$, the inductive hypothesis implies there exists a subgroup $\La_1$ such that $\La_1 \nsub \Ga / K$ and where $(\Ga / K)/ \La_1$ is a torsion free admissible group. Letting $\map{\rho}{\Ga / K}{(\Ga / K) / \La_1}$ be the natural projection, induction additionally implies that $Z((\Ga / K)/ \La_1) \cong \innp{\rho (\pi_K(\ga))}$. Taking $\La_2 = \pi_{K}^{-1}(\La_1)$, we note that $\La_2 / K \cong \La_1$ and that $K \leq \La_2$. Thus, the third isomorphism theorem implies that $(\Ga / K) / (\La_2 / K) \cong \Ga / \La_2$. Hence, $\Ga / \La_2$ is a torsion free admissible, and by construction, $\innp{\pi_{\La_2}(\ga)}$ is a finite index subgroup of $Z(\Ga / \La_2)$.

Letting $\La$ satisfy the hypothesis of the proposition for $\ga$, we now demonstrate that $\Ga / \La$ is irreducible. Suppose for a contradiction there exists a pair of non-trivial admissible groups $\De_1$ and $\De_2$ such that $\Ga / \La \cong \De_1 \times \De_2$. Since $\Ga / \La$ is torsion free, $\De_1$ and $\De_2$ are torsion free. By selection, it follows that $Z(\De_1)$ and $Z(\De_2)$ are torsion free, finitely generated abelian groups. Hence, $\Z^2$ is isomorphic to a subgroup of $Z(\Ga / \La)$. Subsequently, $h(Z(\Ga / \La)) \geq 2$ which is a contradiction. Thus, either $\De_1 \cong \{1\}$ or $\De_2 \cong \{1\}$, and subsequently, $\Ga / \La$ is irreducible.
\end{proof}

\begin{defn}
	Let $\ga \in \Ga$ be a central, non-trivial element, and let $\mathcal{J}$ be the set of subgroups of $\Ga$ that satisfy Proposition \ref{one_dim_center} for $\ga$. Since the set  $\{h(\Ga / \La) \: | \: \La \in \mathcal{J}\}$ is bounded below by $1$, there exists an $\Omega \in \mathcal{J}$ such that $h(\Ga / \Omega) = \text{min}\{h(\Ga / \La) \: | \: \La \in \mathcal{J}\}$. We say  $\Omega$ is a choice of a \emph{admissible quotient of $\Ga$ with respect to $\ga$ }.
\end{defn}
For a primitive $\ga \in Z(\Ga) - 1$, we let $\Ga/ \La_1$ and $\Ga / \La_2$ be two different choices of admissible quotients of $\Ga$ with respect to $\ga$. In general, $\Ga / \La_1 \ncong \Ga / \La_2$. On the other hand, we have, by definition, that $h(\Ga / \La_1) = h(\Ga / \La_2)$. Subsequently, the Hirsch length of a choice of an admissible quotient with respect to $\ga$ is a natural invariant of $\Ga$ associated to $\ga$. Such a quotient corresponds to a torsion free quotient of $\Ga$ of minimal Hirsch length such that $\ga$ has a non-trivial image that generates the center. That will be useful in finding the smallest finite quotient in which $\ga$ has a non-trivial image. \begin{defn}	
	Let $\Ga$ be a torsion free admissible group of step length $c$. For each $\ga \in Z(\Ga) - \set{1}$, we let $\Ga / \La_{\ga}$ be a choice of an admissible quotient of $\Ga$ with respect to $\ga$. Let $\mathcal{J}$ be the set of $\ga \in Z(\Ga) - \set{1}$ such that there exists a $k$ such that $\ga^k = [a,b]$ where $a \in \Ga_{c - 1}$ and $b \in \Ga$. Observe that the set $\set{h(\Ga / \La_{\ga} \: | \: \ga \in \mathcal{J})}$ is bounded above by $h(\Ga)$. Thus, there exists an $\eta \in \mathcal{J}$ such that $h(\Ga / \La_{\eta}) = \text{max}\set{h(\Ga / \La_{\ga}) \: | \: \ga \in \mathcal{J}}.$ We say that $\Ga / \La_{\eta}$ corresponds to a \emph{choice of a maximal admissible quotient of $\Ga$.}
\end{defn}
 
We now define $\psi_{\text{RF}}(\Ga)$ and $\psi_{\text{Lower}}(\Ga)$ when $\Ga$ is an infinite admissible group.
 	\begin{defn}
 		Let $\Ga$ be an infinite admissible group, and let $(\Ga / T(\Ga)) / \La$ be a choice of a maximal admissible quotient of $\Ga / T(\Ga)$. We set $\psi_{\text{RF}}(\Ga) = h((\Ga / T(\Ga)) / \La)$. When $\Ga$ is not virtually abelian, we define $\psi_{\text{Lower}}(\Ga) = \psi_{\text{RF}}(\Ga)\pr{c - 1 }$ where $c$ is the step length of $\Ga / T(\Ga)$.
 	\end{defn}
Suppose that $\Ga / \La_1$ and $\Ga / \La_2$ are two choices of a maximal admissible quotients of $\Ga$ when $\Ga$ is torsion free. In general, $\Ga / \La_1 \ncong \Ga / \La_2$. However, $h(\Ga / \La_1) = h(\Ga / \La_2) = \psi_{\text{RF}}(\Ga)$ by definition; hence, $\psi_{\text{RF}}(\Ga)$ is a well defined invariant of $\Ga$. The value $\psi_{\text{RF}}(\Ga)$ is important because it corresponds to the polynomial in logarithm degree of growth for $\Farb_{\Ga}(n)$. Similarly, we have that $\psi_{\text{Lower}}(\Ga)$ is a well defined invariant of admissible groups that are not virtually abelian. Moreover, the value  $\psi_{\text{Lower}}(\Ga)$ will correspond to the polynomial degree of growth of an asymptotic lower bound for $\Conj_{\Ga}(n)$.

A natural observation is that if $h(Z(\Ga / T(\Ga))) = 1$, then $\psi_{\text{RF}}(\Ga) = h(\Ga)$. Additionally, if $\Ga$ is an infinite, finitely generated abelian group, then $\psi_{\text{RF}}(\Ga) = 1$.

Let $\Ga$ be a torsion free admissible group with a primitive $\ga \in Z(\Ga) - \set{1}$, and let $\Ga / \La$ be a choice of an admissible quotient $\Ga / \La$ with respect to $\ga$. The next proposition demonstrates that we may choose a cyclic series and a compatible generating subset such that a subset of the compatible generating subset generates $\La$.
\begin{prop}\label{minimal_quotient_cyclic_series}
	Let $\Ga$ be a torsion free admissible group, and let $\ga$ be a primitive, central non-trivial element. Let $\Ga / \La$ be a choice of an admissible quotient with respect to $\ga$. There exists a cyclic series $\set{\De_i}_{i=1}^{h(\Ga)}$ and a compatible generating subset $\set{\xi_i}_{i=1}^{h(\Ga)}$ such that $\Ga / \La$ is a choice of an admissible quotient with respect to $\xi_1$ where $\ga = \xi_1$. Moreover, there exists a subset, possibly empty, $\set{\xi_{i_j}}_{j=1}^{h(\La)}$ of the compatible generating subset satisfying the following. The subgroups $W_{t} = \innp{\xi_{i_j}}_{j=1}^t$ form a cyclic series for $\La$ with a compatible generating subset given by $\set{\xi_{i_j}}_{j=1}^{h(\La)}$.
\end{prop} 	
\begin{proof}
	We proceed by induction on $h(\Ga)$, and note that the base case of $h(\Ga) = 1$ is evident. Thus, we may assume that $h(\Ga) > 1$. If $h(Z(\Ga)) = 1$, then $\La \cong \set{1}$; hence, we may take any cyclic series and compatible generating subset. Therefore, we may also assume that $h(Z(\Ga)) > 1$.
	
	There exists a generating basis $\set{z_i}_{i=1}^{h(Z(\Ga)}$ for $Z(\Ga)$ such that $z_1 = \ga$. Letting $K = \innp{z_{i}}_{i = 2}^{h(Z(\Ga))}$, we note that $K \leq \La$. By passing to the quotient $\Ga / K$, we note that $(\Ga /K) / (\La / K)$ is a choice of an admissible quotient with respect to $\pi_K(\ga) = 1$.  Induction implies that there exists a cyclic series $\set{\De_i / K}_{i=1}^{h(\Ga / K)}$ and a compatible generating subset $\set{\pi_{K}(\xi_i)}_{i=1}^{h(\Ga / K)}$ such that there exists a subset $\set{\pi_{K}(\xi_{i_j})}_{j=1}^{h(\Ga / K)}$ satisfying the following. The subgroups $W_t / K = \innp{\xi_{i_j}}_{j=1}^t$ form a cyclic series for $\La / K$ with a compatible generating subset $\set{\pi_{K}(\xi_{i_j})}_{j=1}^{h(\La /K)}$. 
	We let $H_i = \innp{z_s}_{s=1}^{i}$ for $ 1 \leq i \leq h(Z(\Ga))$ and for $i > h(Z(\Ga))$, we let $H_i = \innp{\set{K} \cup \set{\xi_t}_{t=1}^{i-h(K)}}$. We also take $\eta_i = z_i$ for $1 \leq i \leq h(Z(\Ga))$ and for $i > h(Z(\Ga))$, we take $\eta_i = \xi_{i - h(Z(\Ga))}$. Thus, $\set{H_i}_{i=1}^{h(\Ga)}$ is cyclic series for $\Ga$ with a compatible generating subset  $\set{\eta_i}_{i=1}^{h(\Ga)}$.
	
	Consider the subset $\set{\eta_{i_j}}_{j=1}^{h(\La)}$ where $\eta_{i_j} = z_{j +1}$ for $1 \leq j \leq h(K)$ and where $\eta_{i_j} = \xi_{i_{j - h(K)}}$ for $j > h(K)$. Thus, by selection, $\set{\eta_{i_j}}_{j=1}^{h(\La)}$ is the required subset.	\end{proof}

For the next two propositions, we establish some notation. Let $\Ga$ be a torsion free admissible nilpotent group. For each primitive element $\ga \in Z(\Ga) - \set{1}$, we let $\Ga / \La_{\ga}$ be a choice of an admissible quotient with respect to $\ga$. 

We demonstrate that we may calculate $\psi_{\text{RF}}(\Ga)$ for $\Ga$ when given a generating basis for $(\Ga / T(\Ga))_c$ where $c$ is the step length of $\Ga / T(\Ga)$.
\begin{prop}\label{one_dimension_center_compatible_generating subset}
	Let $\Ga$ be a torsion free admissible group, and let $\set{z_i}_{i=1}^{h(Z(\Ga))}$ be a basis of $Z(\Ga)$. Moreover, assume there exist integers $\set{t_i}_{i=1}^{h(\Ga_c)}$ such that $\set{z_i^{t_i}}_{i=1}^{h(\Ga_c)}$ is a basis of $\Ga_c$ and that there exist $a_i \in \Ga_{c - 1}$ and $b_i \in \Ga$ such that $z_i^{t_i} = [a_i,b_i]$. For each $\ga \in \sqrt[Z(\Ga)]{\Ga_c}$, there exists an $i_0 \in \set{1, \cdots, h(\Ga_c)}$ such that $\Ga / \La_{z_{i_0}}$ is a choice of an admissible quotient with respect to $\ga$. More generally, if $\set{z_i}_{i=1}^{h(Z(\Ga)}$ is any basis of $Z(\Ga)$, then there exists $i_0 \in \set{1,\cdots,h(Z(\Ga))}$ such that $\Ga / \La_{i_0}$ is a choice of an admissible quotient with respect to $\ga$.
\end{prop}
\begin{proof}
	Letting $M = \sqrt[Z(\Ga)]{\Ga_c}$, we may write $\ga = \prod_{i=1}^{h(M)}z_i^{\al_i}$. There exist indices $1 \leq i_1 < \cdots < i_\ell \leq h(M)$ such that $\al_{i_j} \neq 0$ for $1 \leq j \leq \ell$ and $\al_{i} = 0$, otherwise. We observe that $\Ga / \La_{z_{i_t}}$ satisfies the conditions of Proposition \ref{one_dim_center} for $\ga$ for each $1 \leq t \leq \ell$.  Therefore, $h(\Ga / \La_\ga) \leq \text{min}\{ h(\Ga / \La_{z_{i_t}})\: | \: 1 \leq t \leq \ell\}$. 
	
	Since $\pi_{\La_\ga}(\ga) \neq 1$, there exists $i_{t_0}$ such that $\pi_{\La_\ga}(z_{i_{t_0}}) \neq 1$. Thus, $\Ga / \La_\ga$ satisfies the conditions of Proposition \ref{one_dim_center} for $z_{i_{t_0}}$. Thus, $ h(\Ga / \La_{z_{i_{t_0}}}) \leq h(\Ga / \La_\ga)$. In particular, $\text{min}\{h(\Ga / \La_{z_{i_t}}) \: | \: 1 \leq t \leq \ell\} \leq h(\Ga / \La_\ga)$. Therefore, $h(\Ga / \La_{\ga}) = \text{min}\{h(\Ga / \La_{z_{i_t}}) \: | \: 1 \leq i \leq \ell \}$. The last statement follows using similar reasoning.
\end{proof}

The following proposition demonstrates that $\psi_{\text{RF}}(\Ga)$ can always be realized as the Hirsch length of a choice of an admissible quotient with respect to a central element of a fixed basis of $\Ga_c$. 
\begin{prop}\label{cyclic_series_maximal_one_dimensional_center}
	Let $\Ga$ be a torsion free admissible group of step length $c$ with a basis $\set{z_{i}}_{i=1}^{h(Z(\Ga))}$ for $Z(\Ga)$. Moreover, assume there exist integers $\set{t_i}_{i=1}^{h(\Ga_c)}$ such that $\set{z_i^{t_i}}_{i=1}^{h(\Ga_c)}$ is a basis of $\Ga_c$ and that there exist elements $a_i \in \Ga_{c -1}$ and $b_i \in \Ga$ such that  $z_i^{t_i} = [a_i,b_i]$. There exists an $i_0 \in \set{1, \cdots, h(\Ga_c)}$ such that $\psi_{\text{RF}}(\Ga) = h(\Ga / \La_{z_{i_0}})$. Hence, $\psi_{\text{RF}}(\Ga) = \text{max}\set{h(\Ga / \La_{z_i}) \: | \: 1 \leq i \leq h(\Ga_c)}.$ More generally, if $\set{z_i}$ is any basis of $Z(\Ga)$, then  $\psi_{\text{RF}}(\Ga) = \text{max}\set{h(\Ga / \La_{z_i}) \: | \: 1 \leq i \leq h(\Ga)}$
\end{prop}
\begin{proof} 
	Let $\mathcal{J}$ be the set of central non-trivial $\ga$ such that there exists a $k$ where $\ga^k$ is a $c$-fold commutator bracket. Given that the set $\{h(\Ga / \La_\ga) \: | \ga \in \mathcal{J}\}$ is bounded above by $h(\Ga)$, there exists a non-trivial element $\eta \in \mathcal{J}$ such that $h(\Ga / \La_\eta) = \text{max}\{h(\Ga / \La_\ga) \: | \ga \in \mathcal{J} \}.$ Proposition \ref{one_dimension_center_compatible_generating subset} implies there exists an $i_0 \in \set{1, \cdots, h(Z(\Ga))}$ such that $h(\Ga / \La_\eta) = h(\Ga / \La_{z_{i_0}})$. By the definition of $\psi_{\text{RF}}(\Ga)$, it follows  $\psi_{\text{RF}}(\Ga) =\text{max}\{h(\Ga / \La_{z_{i_0}}) \: | \: 1 \leq i \leq h(\Ga_c)\}$. The last statement follows using similar reasoning.
\end{proof}

\subsection{Properties of admissible quotients of $\Ga$}
We demonstrate conditions for a choice of an admissible quotient of $\Ga$ with respect to some primitive, central, non-trivial element to have the same step length as $\Ga$. For the next proposition, if $\Ga$ is an admissible group,then we fix its step length as $c(\Ga)$.

\begin{prop}\label{step_length_maximal_one_dim_center}
	Let $\Ga$ be a torsion free admissible group. If we let $\ga \in \sqrt[Z(\Ga)]{\Ga_{c(\Ga)}} - \set{1}$ be a primitive element with a choice of admissible quotient $\Ga / \La$ with respect to $\ga$, then $c(\Ga / \La) = c(\Ga)$. In particular, if $\Ga / \La$ is a choice of a maximal admissible quotient of $\Ga$, then $c(\Ga / \La) = c(\Ga)$. If $c(\Ga) > 1$, then $h(\Ga / \La) \geq 3$.
\end{prop}
\begin{proof}
	By definition, there exists $k \in \Z - \{0\}$ such that $\ga^k \in \Ga_{c(\Ga)}$ Suppose for a contradiction that $c(\Ga / \La) < c(\Ga)$. That implies $\Ga_{c(\Ga)} \leq \ker (\pi_{\La})$, and hence, $\pi_{\La}(\ga^k) = 1$. Since $\Ga / \La$ is torsion free, it follows that $\pi_{\La}(\ga) = 1$. That contradicts the construction of $\Ga / \La$, and thus, $c(\Ga / \La) = c(\Ga)$. 
	
	Since every irreducible, torsion free admissible group $\Ga$ where $c(\Ga) \geq 2$ contains a subgroup isomorphic to the $3$-dimensional integral Heisenberg group, we have  $h(\Ga / \La) \geq 3$.
\end{proof}

The following proposition relates  the value $\psi_{\text{RF}}(\Ga)$ to the value $\psi_{\text{RF}}(\La)$ when $\La$ is a torsion-free quotient of $\Ga$ of lower step length.
\begin{prop}\label{induction_prec_rf}
	Let $\Ga$ be a torsion free admissible group of step length $c > 1$. If $M = \sqrt[Z(\Ga)]{\Ga_{c}}$, then $\psi_{\text{RF}}(\Ga) \geq \psi_{\text{RF}}(\Ga / M)$.
\end{prop}
\begin{proof} There exist elements $\set{z_i}_{i=1}^{h(Z(\Ga/M))}$ and integers $\set{t_i}_{i=1}^{h(N)}$ satisfying the following. The set $\{\pi_M(z_i)\}_{i=1}^{h(Z(\Ga / M))}$ generates $Z(\Ga / M)$ and there exist $a_i \in (\Ga / M)_{c-2}$ and $b_i \in \Ga / M$ such that $\pi([a_i,b_i]) = \pi_{M}(z_i^{t_i})$. Finally, the set $\innp{\set{\pi_M(z_i)^{t_i}}_{i=1}^{h(N)}}$ generates $\sqrt[Z(\Ga / M)]{(\Ga / M)_{c - 1}}$. 
	
	There also exist $\ga_i \in \Ga$ such that the $\set{[z_i,\ga_i]}_{i=1}^{h(\Ga_{c})}$ generate $\Ga_{c}$. Finally, there exist elements $\set{y_i}_{i=1}^{h(M)}$ in $Z(\Ga)$ and integers $\set{s_i}_{i=1}^{h(M)}$ such that $y_i^{s_i} = [z_i,\ga_i]$. For each $i \in \set{1, \cdots, h(M)}$, we let $\Ga / \La_i$ be a choice of an admissible quotient with respect to $y_i$.
	
Let $(\Ga / M) / \Omega_i$ be a choice of an admissible with respect to $\pi_M(z_i)$. It is evident that $\La \leq \pi_M^{-1}(\Omega_i)$. Thus, it follows that $h(\Ga / \La_i) \geq h(\Ga / \pi_M^{-1}(\Omega_i)) = h((\Ga / M) / \Omega_i)$.  Proposition \ref{cyclic_series_maximal_one_dimensional_center} implies $\psi_{\text{RF}}(\Ga) \geq h((\Ga / M) / \Omega_i)$. Applying Proposition \ref{cyclic_series_maximal_one_dimensional_center} again, we have that $\psi_{\text{RF}}(\Ga) \geq \psi_{\text{RF}}(\Ga / M)$. \end{proof}

This last proposition demonstrates that the definition of $\psi_{\text{RF}}(\Ga)$ is the maximum value over all possible Hirsch lengths of choice of admissible quotients with respect to primitive, central non-trivial elements of $\Ga$.

\begin{prop}
Let $\Ga$ be a torsion free admissible group of step length $c$, and for each primitive $\ga \in Z(\Ga) - \set{1}$, let $\Ga / \La_{\ga}$ be a choice of an admissible quotient with respect to $\ga$. Then $\psi_{\text{RF}}(\Ga) = \text{max} \set{h(\Ga / \La_{\ga}) \: | \: \ga \in Z(\Ga) - \set{1}}.$
\end{prop}
\begin{proof}
	The statement is evident for torsion free, finitely generated, abelian groups since $\sqrt[Z(\Ga)]{\Ga_{c}} \cong \Ga$. Thus, we may assume that $c > 1$.
	
	Let $M = \sqrt[Z(\Ga)]{\Ga_{c}}$, and let $\ga \in Z(\Ga) - \set{1}$. There exists a basis $\set{z_i}^{h(Z(\Ga))}$ for $Z(\Ga)$ and integers $\set{t_i}_{i=1}^{h(\Ga_c)}$  such that $\set{z_i^{t_i}}_{i=1}^{h(\Ga)}$ is a basis for $\Ga_c$. Moreover,  there exist $a_i \in \Ga_{c -1}$ and $b_i \in \Ga$ such that $z_i^{t_i} =[a_i,b_i]$. If $\ga \in M$, then by definition of $\psi_{\text{RF}}(\Ga)$ and Proposition \ref{cyclic_series_maximal_one_dimensional_center}, we have $h(\Ga / \La_\ga) \leq \psi_{\text{RF}}(\Ga)$. Thus, we may assume that $\ga \notin M$.
	
	Since $\ga \notin M$, $\pi_M(\ga) \neq 1$. Hence, it is evident that $(\Ga / M) / \pi_{M}(\La_\ga)$ satisfies Proposition \ref{one_dim_center} for $\pi_{M}(\ga)$. Thus, if $(\Ga / M)/ \Omega$ is a choice of an admissible quotient with respect to $\pi_M(\ga)$, we note that $\Ga / \pi_K^{-1}(\Omega)$ satisfies Proposition \ref{one_dim_center} for $\ga$. Thus, by definition, $$h(\Ga / \La_\ga) \leq h(\Ga / \pi_K^{-1}(\Omega)) \leq h((\Ga / M)/ \Omega) \leq \psi_{\text{RF}}(\Ga / M).$$ Proposition \ref{induction_prec_rf} implies that $\psi_{\text{RF}}(\Ga / M) \leq \psi_{\text{RF}}(\Ga)$. Thus, $h(\Ga / \La_{\ga}) \leq \psi_{\text{RF}}(\Ga)$ giving the result.
\end{proof}

\begin{defn}
	A choice of a torsion free admissible group $\Ga$ with a choice of a maximal admissible quotient $\Ga / \La$, and cyclic series $\set{\De_i}_{i=1}^{h(\Ga)}$ a compatible generating subset $\set{\xi_i}_{i=1}^{m}$ that satisfy Proposition \ref{minimal_quotient_cyclic_series} is called an admissible $4$-tuple and is denoted as $\set{\Ga, \La, \De_i, \xi_i}$. 
\end{defn} 

Whenever we are given an admissible $4$-tuple $\{\Ga, \La, \De_i,\xi_i\}$, we take the Mal'tsev completion to be constructed as defined in \S \ref{polycylic_section_title}. We observe that the vectors $v_i = \Log (\xi_i)$ span $\Fr{g}$ the Lie algebra of the Mal'tsev completion. We call the subset $\set{v_i}_{i=1}^{h (\Gamma)}$ an \emph{induced basis} for $\Fr{g}$.

\begin{defn}\label{important_bound_defn}
An admissible $4$-tuple $\set{\Ga, \La, \De_i,\xi_i}$,  Mal'tsev completion $G$ with Lie algebra $\Fr{g}$, and induced basis $\left\{ v_i \right\}_{i=1}^{h (\Gamma)}$ is called an \emph{admissible $7$-tuple} and is denoted as $\set{\Ga, \La, \De_i, \xi_i,  G, \Fr{g}, \nu_i}$. We always take $S =\set{\xi_i}_{i=1}^{h(\Ga)}$ as a generating subset of $\Ga$ and $X = \set{\nu_i}_{i=1}^{h(\Ga)}$ as a basis of $\Fr{g}.$ 
\end{defn}
If we are given a admissible $4$-tuple $\set{\Ga, \La, \De_i,\xi_i}$, then we have a natural choice of an admissible $3$-tuple given by $\set{\Ga, \De_i,\xi_i}$. Whenever we are given an admissible $7$-tuple $\set{\Ga, \La, \De_i,\xi_i, G,\Fr{g},\nu_i}$, then $\set{\Ga,\La,\De_i,\xi_i}$ is an admissible $4$-tuple and $\set{\Ga,\De_i,\xi_i}$ is an admissible $3$-tuple.

\section{Commutator geometry and lower bounds for residual finiteness}\label{lower_bound_geometry}

The following definitions and propositions will be important in the construction of the lower bounds found in the proof of Theorem \ref{word_precise}. \subsection{Finite Index Subgroups and Cyclic Series}
The following proposition tells us how to view finite index subgroups in light of a choice of a cyclic series and compatible generating subset.
\begin{prop}\label{cyclic_series_finite_index_subgroup}
	Let $\{\Ga,\De_i,\xi_i\}$ be an admissible $3$-tuple. Suppose $\Ga$ is torsion free, and let $K \leq \Ga$ be a finite index subgroup. Then there exist natural numbers $\{t_i\}_{i=1}^{h(\Ga)}$ satisfying the following. The subgroups $\{H_i\}_{i=1}^{h(\Ga)}$ given by $H_i = \innp{\xi_s^{t_s}}_{s=1}^i$ form a cyclic series for $K$ with a compatible generating subset given by $\{\xi_i^{t_i}\}_{i=1}^{h(\Ga)}$.
\end{prop}
\begin{proof}
	We proceed by induction on Hirsch length. For the base case, assume that $h(\Ga) = 1$. In this case, we have that $\Ga \cong \Z$ and that $K \cong t \Z$ for some $t \geq 1$. Now the statement of the proposition is evident by choosing $H_1 = K$ and the compatible generating subset is given by $t$.
	
	Thus, we may assume $h(\Ga) > 1$. Observing that  $\De_{h(\Ga)- 1} \cap K$ is a finite index subgroup of $\De_{h(\Ga) - 1}$ and that $h(\De_{h(\Ga) - 1}) = h(\Ga) - 1$, the inductive hypothesis implies there exist natural numbers $\set{t_i}_{i=1}^{h(\Ga)}$ satisfying the following. The groups $\{H_i\}_{i=1}^{h(\Ga) - 1}$ given by $H_i = \innp{\xi_s^{t_s}}_{s=1}^i$ form a cyclic series for $\De_{h(\Ga)-1} \cap K$ with a compatible generating subset given by $\{ \xi_i^{t_i}\}_{i=1}^{h(\Ga) - 1}$. We also have that $\pi_{\De_{h(\Ga)-1}}(K)$ is a finite index subgroup of $\Ga / \De_{h(\Ga) - 1}$. Thus, there exists a natural number $t_{h(\Ga)}$ such that $K / \De_{h(\Ga) - 1} \cong \innp{\pi_{\De_{h(\Ga) - 1}}(\xi_{h(\Ga)}^{t_{h(\Ga)}})}$. If we set $H_{h(\Ga)} \cong \innp{H_{h(\Ga) - 1}, \xi_{h(\Ga)}^{t_{h(\Ga)}}}$, then the groups $\{H_i\}_{i=1}^{h(\Ga)}$ form a cyclic series for $K$ with a compatible generating subset given by $\{\xi_i^{t_i}\}_{i=1}^{h(\Ga)}$.
\end{proof}

We now apply Proposition \ref{cyclic_series_finite_index_subgroup} to give a description of the subgroups of $\Ga$ of the form $\Ga^m$ for $m \in \N$.
\begin{cor}\label{p_th_finite_index}
	Let $\{\Ga,\De_i,\xi_i\}$ be an admissible $3$-tuple such that $\Ga$ is torsion free, and let $m \in \N$. The subgroups $H_i = \innp{\xi_s^{m}}_{s=1}^i$ form a cyclic series for $\Ga^{m}$ with a compatible generating subset given by $\set{\xi_i^{m}}_{i=1}^{h(\Ga)}$. In particular, $|\Ga / \Ga^m| = m^{h(\Ga)}$.
\end{cor}
\begin{proof}
	Proposition \ref{cyclic_series_finite_index_subgroup} implies there exist natural numbers $\set{t_i}_{i=1}^{h(\Ga)}$ such that the subgroups $\set{H_i}_{i=1}^{h(\Ga)}$ given by $H_i = \innp{\xi_{s}^{t_s}}_{s=1}^i$ form a cyclic series for $\Ga^m$ with a compatible generating subset given by $\set{\xi_i^{t_i}}_{i=1}^{h(\Ga)}$. We observe that $(\Ga / \De_1)^m \cong (\Ga^m / \De_1)$. It is also evident that the series $\set{H_i / \De_1}_{i=2}^{h(\Ga)}$ is a cyclic series for $(\Ga / \De_1)^m$ with a compatible generating subset given by $\set{\pi_{\De_1}(\xi_i^{t_i})}_{i=2}^{h(\Ga)}$. Thus, the inductive hypothesis implies that $t_i = m$ for all $2 \leq i \leq h(\Ga)$. To finish, we observe that $\Ga^m \cap \De_1 \cong \De_1^m$.  Thus, $t_1 = m$ as desired. \end{proof}

Let $\Ga$ be a torsion free admissible group, and let $K \leq \Ga$ be a finite index subgroup. The following proposition allows us to understand how $K$ intersects a fixed choice of an admissible quotient of $\Ga$ with respect to a primitive central element.

\begin{prop}\label{subset_cyclic_series_intersection_finite_index}
	Let $\{\Ga, \La, \De_i, \xi_i \}$ be an admissible $4$-tuple. Let $K$ be a finite index subgroup of $\Ga$. There exist indices $1 \leq i_1 < i_2 < \cdots  i_\ell \leq h(\Ga)$ with natural numbers $\{t_s\}_{s=1}^\ell$ such that the subgroups $H_s = \innp{\xi_{i_j}^{t_j}}_{j=1}^s$ form a cyclic series for $K \cap \La$ with a compatible generating subset  $\{\xi_{i_s}^{t_s}\}_{s=1}^\ell$.
\end{prop}
\begin{proof}
	 By assumption, the cyclic series $\set{\De_i}_{i=1}^{h(\Ga)}$ and compatible generating subset $\set{\xi_i}_{i=1}^{h(\Ga)}$ satisfy the  conditions of Proposition \ref{minimal_quotient_cyclic_series} for $\La$. Thus, there exists indices $1 \leq i_1 < i_2 < \cdots  i_\ell$ such that the subgroups $M_j = \innp{\xi_{i_s}}_{s=1}^j$ form a cyclic series for $\La$ with a compatible generating subset given by $\set{\xi_{i_s}}_{s=1}^{h(\La)}$. Applying Proposition \ref{cyclic_series_finite_index_subgroup} to the admissible $3$-tuple $\set{\Ga,\De_i,\xi_i}$, we have that there exist natural numbers $\{t_i\}_{i=1}^{h(\Ga)}$ such that the subgroups given by $W_i = \innp{\xi_s^{t_s}}_{s=1}^i$ form a cyclic series for $K$ with a compatible generating subset given by $\set{\xi_i^{t_i}}_{i=1}^{h(\Ga)}$. Since $K \cap \La$ is a finite index subgroup of $\La$, the groups given by $H_{s} = \innp{\xi_{i_j}^{t_{i_j}}}_{j=1}^s$ form the desired cyclic series for $K \cap \La$ with a compatible generating subset given by $\{\xi_{i_s}^{t_{i_s}}\}_{s=1}^\ell$. Therefore, $\{t_{i_s}\}_{s=1}^{\ell}$ are the desired integers.
\end{proof}

\subsection{Reduction of Complexity for Residual Finiteness}
We first demonstrate that we may assume that $\Ga$ is torsion free when calculating $\Farb_{\Ga}(n)$.
\begin{prop}\label{torsion_reduction}
	Let $\Ga$ be an infinite admissible group. Then $\Farb_{\Ga}(n) \approx \Farb_{\Ga / T(\Ga)}(n)$.
\end{prop}
\begin{proof}
	We proceed by induction on $|T(\Ga)|$, and observe that the base case is evident. Thus, we may assume that $|T(\Ga)| > 1$. Note that $\map{\pi_{Z(T(\Ga))}}{\Ga}{\Ga / Z(T(\Ga))}$ is surjective and that $\ker (\pi_{Z(T(\Ga))}) = Z(T(\Ga))$ is a finite central subgroup. Since admissible groups are linear, \cite[Lem 2.4]{BK} implies that $\Farb_{\Ga}(n) \approx \Farb_{\Ga / T(Z(\Ga))}(n)$. Since $(\Ga / Z(T(\Ga))) / T(\Ga / Z(T(\Ga))) \cong \Ga / T(\Ga)$, the inductive hypothesis implies that $\Farb_{\Ga}(n) \approx \Farb_{\Ga / T(\Ga)}(n)$.
\end{proof} 

The following proposition implies that we may pass to a choice of a maximal admissible quotient of $\Ga$ when computing the lower bounds of $\Farb_{\Ga}(n)$ when $\Ga$ is a torsion free admissible group.

\begin{prop}\label{reduction_one_dim_word}
	Let $\set{\Ga, \La, \De_i, \xi_i}$ be an admissible $4$-tuple. If $\map{\varphi}{\Ga}{Q}$ is a surjective homomorphism to a finite group, then $\varphi(\xi_1^{m}) \neq 1$ if and only if $\pi_{\varphi(\La)}(\varphi(\xi_1^m)) \neq 1$ where $m \in \N$.
\end{prop} 
\begin{proof}
	If $\La \cong \set{1}$, then there is nothing to prove. Thus, we may assume that $\La \ncong \set{1}$. Proposition \ref{minimal_quotient_cyclic_series} implies that $\xi_1 \notin \La$ and that there exists a collection of elements of the Mal'tsev basis $\{\xi_{i_s} \}_{s=1}^\ell$ such that $\La \cong \innp{\xi_{i_s}}_{s=1}^{h(\La)}$. Moreover, we have that $H_s = \innp{\xi_{i_t}}_{t=1}^s$ is cyclic series for $\La$ with a compatible generating subset given by $\{\xi_{i_s}\}_{s=1}^\ell$. 	Proposition \ref{subset_cyclic_series_intersection_finite_index} implies there exist $\{t_{s}\}_{s=1}^\ell \subseteq \N$ such that the series of subgroups $\{W_s\}_{s=1}^{h(\La)}$ given by $W_s = \innp{\xi_{i_{j}}^{t_j}}_{j=1}^s$ forms a cyclic series for $\ker(\varphi) \cap \La$ with a compatible generating subset given by $\{\xi_{i_s}^{t_{s}}\}_{s=1}^{h(\La)}$
	
	Since the backwards direction is clear, we proceed with forward direction. To be more specific, we demonstrate that if $\varphi(\xi_1^m) \neq 1$, then $\pi_{\varphi(\La)}(\xi_1^m) \neq 1$. We proceed by induction on $|\varphi(\La)|$, and observe that the base case is clear. Thus, we may assume that $|\varphi(\La)| > 1$. In order to apply the inductive hypothesis, we find a non-trivial normal subgroup $M \leq Z(Q)$ such that $\varphi(\xi_1^m) \notin M$. 
	
	We first observe that if $\varphi(\xi_{i_0}) \neq 1$ for some $i_0 \in \set{2, \cdots, h(Z(\Ga))}$,  we may set $M = \innp{\varphi(\xi_{i_0})}$. It is straightforward to see that $M \neq \set{1}$ and that $\varphi(z) \notin M$. Thus, we may assume that $\xi_i \in \ker(\varphi)$ for $i \in \set{2, \cdots, h(Z(\Ga))}$.
	
	In this next paragraph, we prove that there exists an element of the compatible generating subset, say $\xi_{i_{0}}$, such that $\xi_{i_0} \in \La$, $\xi_{i_0} \notin \ker(\varphi)$, and $\varphi(\xi_{i_0}) \in Z(Q)$. To that end, we note that if $|t_{i_s}| = 1$, then $\xi_{i_s} \in \ker(\varphi)$. Since $|\varphi(\La)| > 1$, the set  $E = \{\xi_{i_s} \: | \: |t_{i_s}| \neq 1 \}$ is non empty. Given that $E$ is a finite set, there exists $\xi_{i_{s_0}} \in E$ such that $\Height(\xi_{i_{s_0}}) = \text{min}\{\Height(\xi_{i_s}) \: | \: \xi_{i_s} \in E \}$. We claim that $\varphi(\xi_{i_{s_0}})$ is central in $Q$, and since we are assuming that $\varphi(\xi_i) = 1$ for $i \in \set{2, \cdots, h(Z(\Ga))}$, we may assume that $\Height(\xi_{i_{s_0}}) > 1$. Since $\Height([\xi_{i_{s_0}},\xi_t]) < \Height(\xi_{i_{s_0}})$ for any $\xi_t$ and that $\varphi(\La) \nsub Q$, it follows $[\varphi(\xi_{i_{s_0}}), \varphi(\xi_t)] \in \varphi(\La)$. Thus, $[\varphi(\xi_{i_{s_0}}), \varphi(\xi_t)]$ is a product of $\varphi(\xi_{i_{s_j}})$ where $\Height(\xi_{i_{s_j}}) < \Height(\xi_{i_{s_0}})$. Since $\xi_{i_{s_j}} \in \varphi(\La)$ and  $\Height(\xi_{i_{s_j}}) < \Height(\xi_{i_{s_0}})$, the definition of $E$ and the choice of $\xi_{i_0}$ implies that $t_{i_{s_j}} = 1$. Thus, $\xi_{i_{s_j}} \in \ker(\varphi)$, and subsequently, $\varphi(\xi_{i_{s_j}}) = 1$. Hence, $[\varphi(\xi_{i_{s_0}}), \varphi(\xi_t)] =1$, and thus, $\varphi(\xi_{i_0}) \in Z(Q) - \set{1}$.
	
	Since $\varphi(\xi_{i_{s_0}})$ is central in $Q$, the group $M = \innp{\varphi(\xi_{i_0})}$ is a normal subgroup of $Q$. By selection, $\varphi(\xi_1^m) \notin M$, and since $|\pi_{M}( \varphi(\La))| < |\varphi(\La)|$, we may apply the inductive hypothesis to the surjective homomorphism $\map{\pi_M \circ \varphi}{\Ga}{Q / M}$. Letting $N = \pi_M \circ \varphi(\La)$, we have that $\pi_N(\pi_M( \varphi(\xi_1^m))) \neq 1$. Thus, $\pi_{\varphi(\La)}(\xi_1^m) \neq 1$. \end{proof}

\subsection{Rank and Step Estimates}
\begin{defn} Let $\{\Ga, \De_i, \xi_i \}$ be an admissible $3$-tuple such that $\Ga$ is an infinite admissible group of step size $c$, and let $\ell \leq h(\Ga)$. Let $\vec{a} = (a_i)_{i=1}^{\ell}$ where $1 \leq a_i \leq h(\Ga)$ for all $i$. We write $[\xi_{\vec{a}}] = [\xi_{a_1}, \hdots, \xi_{a_\ell}].$ We call $[\xi_{\vec{a}}]$ a \emph{simple commutator of weight $\ell$ with respect to $\vec{a}$}. Let $W_k(\Ga)$ be the set of non-trivial simple commutators of weight $k$. Since $\Ga$ is a nilpotent group of step size $c$, $W_k(\Ga)$ is an empty set for $k \geq c + 1$.  Thus, the set of non-trivial simple commutators of any weight, denoted as $W(\Ga)$, is finite.  \end{defn}

When considering a surjective homomorphism to a finite group $\map{\varphi}{\Ga}{Q}$, we need to ensure that the step length of $Q$ is equal to the step length of $\Ga$. We do that by assuming that $\varphi([\xi_{\vec{a}}]) \neq 1$ for all $[\xi_{\vec{a}}] \in W(\Ga) \cap Z(\Ga)$.
\begin{prop}\label{step_preservation}
	Let $\{\Ga,\De_i,\xi_i\}$ be an admissible $3$-tuple, and assume that $\Ga$ is torsion free. Let $\map{\varphi}{\Ga}{Q}$ be a surjective homomorphism to a finite group such that if $[\xi_{\vec{a}}] \in W(\Ga) \cap Z(\Ga)$, then $\varphi([\xi_{\vec{a}}]) \neq 1$. Then $\varphi([\xi_{\vec{a}}]) \neq 1$ for all $[\xi_{\vec{a}}] \in W(\Ga)$. If $c_1$ and $c_2$ are the step sizes of $\Ga$ and $Q$, respectively, then $c_1 = c_2$.
\end{prop}
\begin{proof}
We first demonstrate that $\varphi([\xi_{\vec{a}}]) \neq 1$ for all $[\xi_{\vec{a}}] \in W(\Ga)$ by induction on $\Height([\xi_{\vec{a}}])$. Observe that if $[\xi_{\vec{a}}] \in W_k(\Ga)$, then $\Height([\xi_{\vec{a}}]) = c(\Ga) - k + 1$. Thus, if $[\xi_{\vec{a}}] \in W_{c(\Ga)}(\Ga)$, then $\Height([\xi_{\vec{a}}]) = 1$. Hence, the base case follows from assumption.

Now consider $[\xi_{\vec{a}}] \in W(\Ga)$ where $\Height([\xi_{\vec{a}}]) = \ell > 1$. If $[\xi_{\vec{a}}] \in Z(\Ga)$, then the assumptions of the proposition imply $\varphi([\xi_{\vec{a}}]) \neq 1$. Thus, we may assume there exists an element $\xi_{i_0}$ of the Mal'tsev basis such that $[[\xi_{\vec{a}}],\xi_{i_0}] \neq 1$. The induction hypothesis implies that  $\varphi([[\xi_{\vec{a}}],\xi_{i_0}]) \neq 1$ since $[[\xi_{\vec{a}}],\xi_{i_0}]$ is a simple commutator of $\Height([[\xi_{\vec{a}}],\xi_{i_0}]) \leq \ell - 1$. Thus, $\varphi([\xi_{\vec{a}}]) \neq 1$. Therefore, for each $[\xi_{\vec{a}}] \in W(\Ga)$, it follows that $\varphi([\xi_{\vec{a}}]) \neq 1$.

If $c_1 < c_2$, then $\varphi$ factors through $\Ga / \Ga_{c_1}$, and subsequently $W_{c_1} \subseteq \ker (\varphi)$. Since $W_{c_1} \subseteq W(\Ga) \cap Z(\Ga)$, we have a contradiction, and subsequently, $c_1 = c_2$.
\end{proof}

For a surjective homomorphism to a finite $p$-group $\map{\varphi}{\Ga}{Q}$, the following proposition gives conditions for $|Q| \geq p^{h(\Ga)}$. To be more specific, if $\varphi$ is an injective map when restricted to the set of central simple commutators and is an injective map when restricted to a central elements of a fixed compatible generating subset, then $\varphi$ is an injection when restricted to that same compatible generating subset.

\begin{prop}\label{rank_preserving_bound}
Let $\set{\Ga,\De_i,\xi_i}$ be an admissible $3$-tuple such that $\Ga$ is torsion free. Let $\map{\varphi}{\Ga}{Q}$ be a surjective homomorphism to a finite $p$-group. Suppose that  $\varphi([\xi_{\vec{a}}]) \neq 1$ for all $[\xi_{\vec{a}}] \in W(\Ga) \cap Z(\Ga)$. Also, suppose that $\varphi (\xi_i) \neq 1$ for $\xi_i \in Z(\Ga)$ and $\varphi (\xi_i) \neq \varphi(\xi_j)$ for $\xi_i,\xi_j \in Z(\Ga)$ where $i \neq j$. 
Then $\varphi(\xi_i) \neq 1$ for $1 \leq i \leq h(\Ga)$ and $\varphi(\xi_i) \neq \varphi(\xi_j)$ for $1 \leq j_1 < j_2 \leq h(\Ga)$.
\end{prop}
\begin{proof}
Let $\xi_{j_1} \notin Z(\Ga)$. By selection, there exists $\xi_{j_2}$ such that $[\xi_{j_1},\xi_{j_2}] \neq 1$. Since $[\xi_{j_1},\xi_{j_2}]$ is a simple commutator of weight $2$, Proposition \ref{step_preservation} implies that $\varphi([\xi_{j_1}, \xi_{j_2}]) \neq 1$. Thus, $\varphi(\xi_{j_1}) \neq 1$.

We now demonstrate that $\varphi(\xi_i) \neq \varphi(\xi_j)$ for all $i < j$. Assume for a contradiction that $\varphi(\xi_i) = \varphi(\xi_j)$. That implies $\xi_i \: \xi_j^{-1} \in \ker(\varphi)$. Since $\ker(\varphi)$ is a normal finite index subgroup, Proposition \ref{cyclic_series_finite_index_subgroup} implies that there exist natural numbers $\set{t_i}_{i=1}^{h(\Ga)}$ such that the subgroups $\set{H_i}_{i=1}^{h(\Ga)}$ given by $H_i \cong \innp{\xi_s^{t_s}}_{s=1}^i$ form a cyclic series for $\ker(\varphi)$ with a compatible generating subset given by $\set{\xi_i^{t_i}}_{i=1}^{h(\Ga)}$. Since $Q$ is a $p$-group and $\xi_i \notin \ker(\varphi)$ for all $i$, each $t_i$ must be a non-zero power of $p$. We may write $\xi_i \: \xi_j^{-1} = \prod_{\ell = 1}^{h(\Ga)} \xi_\ell^{s_\ell \: t_\ell}$ for $s_\ell \in \Z$. We observe that $s_i \: t_i$ is divisible by $p$ for all $i$, and since each element can be represented uniquely in terms of its Mal'tsev coordinates, it follows that $s_\ell = 0$ for all $\ell \neq i,j$. Thus, we may write $\xi_i \: \xi_j^{-1} = \xi_i^{s_i \: t_i} \: \xi_j^{s_j \: t_j}$. That implies $1 = s_i \: t_i$ and $-1 = s_j \: t_j$. Since $p$ divides $s_i \: t_i$, we have a contradiction. Thus, $\varphi(\xi_i) \neq \varphi(\xi_j)$.

Thus, $\set{\varphi(\xi_i)}_{i=1}^{h(\Ga)}$ is a generating subset of $Q$ where $\text{Ord}_Q(\varphi(\xi_i)) \geq p$ for all $i$. \cite[Thm 1.10]{Hall_notes} implies that $|Q|$ divides some power of $p^{h(\Ga)}$. Hence, $|Q| \geq p^{h(\Ga)}$.\end{proof}

The following definition will be important in the proofs of Theorem \ref{word_precise} and Theorem \ref{lower}.
\begin{defn}
	Let $\set{\Ga, \De_i, \xi_i}$ be an admissible $3$-tuple such that $\Ga$ is torsion free and $h(Z(\Ga)) = 1$. For $[\xi_{\vec{a}}] \in W(\Ga) \cap Z(\Ga)$, we let $k\pr{[\xi_{\vec{a}}]}$ satisfy $\xi_1^{k\pr{[\xi_{\vec{a}}]}} = [\xi_{\vec{a}}]$. 
	Let $\Bas(\Ga) = \lcm\{|k\pr{[\xi_{\vec{a}}]}| \: | \: [\xi_{\vec{a}}] \in W(\Ga) \cap Z(\Ga) \}.$
\end{defn}

\begin{prop}\label{p_group_lower_bounds}
	Let $\{\Ga,\De_i,\xi_i\}$ be a admissible $3$-tuple where $\Ga$ is torsion free and $h(Z(\Ga)) = 1$.  Suppose $\map{\varphi}{\Ga}{Q}$ is a surjective homomorphism to a finite $p$-group such that $p > \Bas(\Ga)$, and suppose that $\varphi(\xi_1) \neq 1$. Then $c_1 = c_2$, $Z(Q) = \innp{\varphi(\xi_1)}$ and $|Q| \geq p^{h(\Ga)}$ where $c_1$ and $c_2$ are the step lengths of $Q$ and $\Ga$, respectively.
\end{prop}
\begin{proof}
	Since $\varphi(\xi_1) \neq 1$ and $Q$ is $p$-group, we have $\text{Ord}_Q(\varphi(\xi_1)) \geq p$. We claim that if $[\xi_{\vec{a}}] \in W(\Ga) \cap Z(\Ga)$, then $\varphi([\xi_{\vec{a}}]) \neq 1$. Suppose for a contradiction that $\varphi([\xi_{\vec{a}}]) = 1$ for some $[\xi_{\vec{a}}] \in W(\Ga) \cap Z(\Ga)$. Since $\varphi(\xi_1^{\Bas(\Ga)})$ is a power of $\varphi([\xi_{\vec{a}}])$ by definition, we have $\varphi(\xi_1^{\Bas(\Ga)}) = 1$. Thus, $\varphi(\xi_1)$ has order strictly less than $p$ which is a contradiction. 
	
	Since $\varphi([\xi_{\vec{a}}]) \neq 1$ for $[\xi_{\vec{a}}] \in W(\Ga) \cap Z(\Ga)$, Proposition \ref{step_preservation} implies that $c_1 = c_2$. On the other hand, Proposition \ref{rank_preserving_bound} implies that $\varphi(\xi_i) \neq 1$ for all $1 \leq i \leq h(\Ga)$ and $\varphi(\xi_{j_1}) \neq \varphi(\xi_{j_2})$ for all $1 \leq j_1 < j_2 \leq h(\Ga)$. Thus, $\{\varphi(\De_i)\}_{i=1}^{h(\Ga)}$ is a cyclic series for $Q$ and $\{\varphi(\xi_i)\}_{i=1}^{h(\Ga)}$ is a compatible generating subset for $Q$. 	Since $Q$ is a $p$-group, $|\varphi(\De_{i}) : \varphi(\De_{i-1})| \geq p$ for each $1 \leq i \leq h(\Ga)$ with the convention that $\De_0 = \{1\}$. Hence, the  second paragraph after \cite[Defn 8.2]{computational_group} implies $|Q| = \prod_{i=1}^{h(\Ga)}|\De_i : \De_{i-1}| \geq p^{h(\Ga)}.$
	
	We finish by demonstrating $Z(Q) = \innp{\varphi(\xi_1)}$. Since $\{\varphi(\De_i) \}_{i=1}^{h(\Ga)}$ is an ascending central series that is a refinement of the upper central series, there exists $i_0$ such that $\varphi(\De_{i_0}) = Z(Q)$. For $t>1$, there exists $j \neq t$ such that $[\xi_t,\xi_j] \neq 1$. Since $[\xi_t,\xi_j]$ is a simple commutator of weight $2$, Proposition \ref{step_preservation} implies that $\varphi([\xi_t,\xi_j]) \neq 1$. Given that $\varphi([\xi_t,\xi_j]) = [\varphi(\xi_t),\varphi(\xi_j)]$, it follows $\varphi(\xi_t) \notin Z(Q)$. That implies $\varphi(\De_t) \gneq Z(Q)$ for all $t>1$. Hence, $Z(Q) = \innp{\varphi(\xi_1)}$.
\end{proof}

If $\set{\Ga,\La,\De_i,\xi_i}$ is an admissible $4$-tuple with a surjective homomorphism to a finite group $\map{\varphi}{\Ga}{Q}$ and $m \in \Z - \set{0}$, then the following proposition gives conditions such that $Q$ has no proper quotients in which $\varphi(\xi_1^m) \neq 1$.
\begin{prop}\label{alternative}
	Let $\set{\Ga,\La, \De_i,\xi_i}$ be an admissible $4$-tuple. Suppose that $\map{\varphi}{\Ga}{Q}$ is a surjective homomorphism to a finite $p$-group where $\varphi(\La) \cong \set{1}$, $p > B(\Ga / \La)$, and $|Q| \leq p^{\psi_{\text{RF}}(\Ga)}$. If $\varphi(\xi_1^m) \neq 1$ for some $m \in \Z$, then $|Q| = p^{\psi_{\text{RF}}(\Ga)}$. Additionally, if $N$ is a proper quotient of $Q$, then $\rho(\varphi(\xi_1^m)) = 1$ where $\map{\rho}{Q}{N}$ is the natural projection. Finally, $Z(Q) \cong \Z / p \Z$.
\end{prop}
\begin{proof} 
	Let us first demonstrate that $|Q| = p^{\psi_{\text{RF}}(\Ga)}$. Since $\La \leq \ker(\varphi)$, we have an induced homomorphism $\map{\widetilde{\varphi}}{\Ga / \La}{Q}$ such that $\widetilde{\varphi}\circ \pi_{\La} = \varphi$. Since $\map{\widetilde{\varphi}}{\Ga / \La}{Q}$ is a surjective homomorphism to a finite $p$-group where $p > \Bas(\Ga / \La)$, $h(Z(\Ga / \La)) = 1$, and $\varphi(\xi_1) \neq 1$, Proposition \ref{p_group_lower_bounds} implies that $\varphi(Z(\Ga / \La)) \cong Z(Q)$ and $|Q| \geq p^{\psi_{\text{RF}}(\Ga)}$. Therefore, $|Q| = p^{\psi_{\text{RF}}(\Ga)}$.
	
	We now demonstrate that $Z(Q) \cong \Z / p\Z$. Since $\varphi(\De_i / \La)$ is a cyclic series for $Q$ with a compatible generating given by $\set{\varphi(\xi_i) \: | \: \xi_i \notin \La}$, it follows that $|Q| = \prod_{\xi_i \notin \La}\text{Ord}_Q(\varphi(\xi_i))$ (see the second paragraph after \cite[Defn 8.2]{computational_group}). Thus, we must have that $\text{Ord}_Q(\varphi(\xi_i)) \leq p$. Since $\text{Ord}_Q(\varphi(\xi_1)) \geq p$. we have $\text{Ord}_Q(\varphi(\xi_1)) = p$. Since $Z(Q) \cong \innp{\varphi(\xi_1)}$, it follows that $Z(Q) \cong \Z / p \Z$.
	
	Since $Z(Q) \cong \Z / p\Z$, there are no proper, non-trivial subgroups of $Z(Q)$. Given that $\ker(\rho) \nsub Q$, we have $Z(Q) \cap \ker(\rho) = Z(Q)$; hence, $\rho(\varphi(\xi_i^m)) = 1$ because $\varphi(\xi_1^m) \in Z(Q) \leq \ker(\rho)$.
\end{proof} \section{Some Examples of Precise Residual Finiteness}
To demonstrate the techniques used in the proof of Theorem \ref{word_precise}, we make a precise calculation of $\Farb_{\Heis_{2m+1}(\Z)}(n)$ where $\Heis_{2m+1}(\Z)$ is the $(2m+1)$-dimensional integral Heisenberg group.
\subsection{Integral Heisenberg Group Basics}\label{integral_heisen_basics}
We start by introducing basic facts about the $(2m+1)$-dimensional integral Heisenberg group which will be useful in the calculation of $\Farb_{\Heis_{2m+1}(\Z)}(n)$ and in Section \ref{Heisenberg_conjugacy_section}. We may write
\[
\Heis_{2m + 1}(\mathbb{Z}) = \set{ \left. \begin{pmatrix}
	1 & \vec{x} & z \\
	\vec{0} & \textbf{I}_m & \vec{y} \\
	0 & \vec{0} & 1
	\end{pmatrix}
	\right| z \in \mathbb{Z},  \: \: \vec{x},\vec{y}^T \in \mathbb{Z}^m }
\]
where $\textbf{I}_m$ is the $m \times m$ identity matrix. If $\gamma \in \Heis_{2k+1}(\mathbb{Z})$, we write
\[
\gamma = 
\begin{pmatrix}
1 & \vec{x}_{\gamma} & z_{\gamma} \\
\vec{0} & \textbf{I}_m & \vec{y}_{\gamma} \\
0 & \vec{0} & 1
\end{pmatrix}
\]
where $\vec{x}_{\gamma} = [x_{\gamma,1}, \hdots, x_{\gamma,m}]$ and $\vec{y}_{\gamma}^{T} = [y_{\gamma,1}, \hdots, y_{\gamma,m}]$. 

We let $E = \set{\vec{e}_i}_{i=1}^m$ be the standard basis of $\mathbb{Z}^m$ and then choose a generating subset for $\Heis_{2m+1}(\Z)$ given by $S = \set{\alpha_1,\dots,\alpha_m,\beta_1,\dots,\beta_m,\lambda}$ where
\[
\alpha_i  =
\begin{pmatrix}
1 & \vec{e}_i & 0 \\
\vec{0} & \mathbf{I}_m & \vec{0} \\
0 & \vec{0} & 1
\end{pmatrix}, \:
\beta_i =
\begin{pmatrix}
1 & \vec{0} & 0 \\
\vec{0} & \textbf{I}_m & \vec{e}_i^T \\
0 & \vec{0} & 1
\end{pmatrix},
\text{ and }
\lambda =
\begin{pmatrix}
1 & \vec{0} & 1 \\
\vec{0} & \textbf{I}_m & \vec{0} \\
0 & \vec{0} & 1
\end{pmatrix}.
\]
Thus, if $\gamma \in B_{\Heis_{2m+1}(\mathbb{Z}),S} (n)$, then $\vec{x}_{\gamma}, \vec{x}_{\eta}, \vec{y}_{\gamma}^{T}, \vec{y}_{\eta}^{T} \in B_{\mathbb{Z}^m, E} (C_0 \: n)$ and $|z_{\gamma}| \leq C_0 \: n^{2}$ for some $C_0 \in \N$ \cite[3.B2]{asymptotic_group}.  Lastly, we obtain a finite presentation for $\Heis_{2m+1}(\Z)$ written as 
\begin{equation}\label{presentation_heisen} \Heis_{2m+1}(\Z) = \innp{\kappa, \mu_{i},\nu_{j} \text{ for } 1 \leq i,j \leq m \: | \: [\mu_t,\nu_t]= \kappa \text{ for } 1 \leq t \leq m} 
\end{equation} with all other commutators being trivial. 

\subsection{Residual Finiteness of $\Heis_{2m+1}(\Z)$}\label{Heisen_cyclic_series}
Since the upper and lower asymptotic bounds for $\Farb_{\Heis_{2m+1}(\Z)}(n)$ require different strategies, we approach them separately. We start with the upper bound as it is more straight forward. 

Before we begin with the upper bound, we collect some basic facts. We take presentation given in Equation \ref{presentation_heisen} with  $S = \{\mu_i,\nu_j,\kappa \: | \: 1 \leq i,j\leq m\}$. Let $\De_1 = \innp{\kappa}$, $\De_{i} = \innp{\set{\kappa} \cup \set{\mu_s}}_{s=1}^{t-1}$ for $2 \leq i \leq m+1$, and $\De_i =\innp{\set{\De_{m+1}}, \set{\nu_t}_{t=1}^{i-m-1}}$ for $m+2 \leq i \leq 2m+1$. One can see that $\{\De_i\}_{i=1}^{2m+1}$ is a cyclic series for $\Heis_{2m+1}(\Z)$ and that $S$ is a compatible generating subset.
\begin{prop}\label{heisenberg_word_upper}
	$\Farb_{\Heis_{2m+1}(\Z)}(n) \preceq (\log(n))^{2m+1}$.
\end{prop}
\begin{proof}
	Let $\|\ga\|_S \leq n$. We seek to construct a surjective homomorphism $\map{\varphi}{\Heis_{2m+1}(\Z)}{Q}$ to a finite group such that $\varphi(\ga) \neq 1$. Moreover, we want to construct this finite group such that $|Q| \leq C_0 (\log(C_0 \: n))^{2m+1}$ for some $C_0 > 0$.
	
	Via the Mal'tsev basis, we may write $\ga = \kappa^{\al} \pr{\prod_{i=1}^{m} \mu_i^{\beta_i}} \pr{\prod_{j=1}^m \nu_j^{\la_j}}.$ We proceed based on whether $\ga$ has a trivial image in the abelianization or not.
	
	Suppose that $\pi_{\text{ab}}(\ga) \neq 1$.	Since $\ga \neq 1$, either $\be_i \neq 0$ for some $i$, or $\la_j \neq 0$ for some $j$. Without loss of generality, we may assume that there exists some $i_0$ such that $\be_{i_0} \neq 0$. The Prime Number Theorem \cite[1.2]{tenenbaum} implies there exists a prime $p$ such that $p \nmid |\be_{i_0}|$ and where $p \leq C_2 \log(C_2 \: |\be_{i_0}|) \leq C_2 \log(C_1 \: C_2 \: n^2)$. Consider the map $\map{\rho}{\Heis_{2m+1}(\Z)}{\Z / p \Z}$ given by
	
	$$\kappa^{\al} \pr{\prod_{i=1}^{m} \mu_i^{\beta_i}} \pr{\prod_{j=1}^m \nu_j^{\la_j}} \longrightarrow (\beta_1,\cdots,\beta_m,\la_1,\cdots,\la_m) \longrightarrow \be_{i_0} \longrightarrow \be_{i_0} \:  ( \text{ mod } \: p \:).$$ 
	
	Here, the first arrow is the abelianization map, the second arrow is the projection to the $\beta_{i_0}$ coordinate, and the last arrow is the natural projection from $\Z$ to $\Z / p \Z$. By construction, $\rho(\ga) \neq 1$ and $|\Z / p \Z| \leq C_1 \: C_2 \log(C_1 \: C_2 \: n^2)$. Thus, $\Depth_{\Heis_{2m+1}(\Z)}(\ga) \leq C_3 \log(C_3 \: n)$ for some $C_3 > 0$.
	
	Now suppose that $\pi_{\text{ab}}(\ga) = 1$. That implies $\be_i,\la_j = 0$ for all $i,j$. As before, the Prime Number Theorem \cite[1.2]{tenenbaum} implies there exists a prime $p$ such that $p \nmid |\al|$ and $p \leq C_4 \: \log(C_4 \: n)$ for some  $C_4 \in \N$. We have that $\sigma_p(\ga) = \sigma_p(\kappa^{\al_1}) \neq 1$. The second paragraph after \cite[Defn 8.2]{computational_group} implies that $|\Heis_{2m+1}(\Z) / (\Heis_{2m+1}(\Z))^p| = p^{2m+1}$ since $|\De_{i+1} : \De_i| = p$ for $1 \leq i \leq 2m+1$. Hence, $|\Heis_{2m+1}(\Z) / (\Heis_{2m+1}(\Z))^p| \leq (C_4)^{2m+1} \: (\log(C_4 \: n))^{2m+1}.$ Thus, $\Depth_{\Heis_{2m+1}(\Z)}(\ga) \leq C_4 (\log(C_4 \: n))^{2m+1}$, and therefore, $\Farb_{\Heis_{2m+1}(\Z)}(n) \preceq \pr{\log(n)}^{2m+1}.$
\end{proof}

\begin{prop}
	$(\log(n))^{2m+1} \preceq \Farb_{\Heis_{2m+1}(\Z)}(n)$.
\end{prop}
\begin{proof}
	In order to demonstrate that $(\log(n))^{2m+1} \preceq \Farb_{\Heis_{2m+1}(\Z)}(n)$, we construct a sequence of elements $\{\ga_i\}$ such that $C_1(\log(C_1 \: \|\ga_i\|_S))^{2m+1} \leq \Depth_{\Heis_{2m+1}(\Z)}(\ga_i)$ for some $C_1 > 0$. The proof of Proposition \ref{heisenberg_word_upper} implies that when $\ga \notin Z(\Heis_{2m+1}(\Z))$ that $\Depth_{\Heis_{2m+1}(\Z)}(\ga) \leq  C_1 \: \log(C_1 \: \|\ga\|_S)$ for some $C_1 \in \N$. Thus, the elements we are looking for will be central elements.
	
	Let $\{p_i\}$ be an enumeration of the primes, and let $\al_i = (\lcm\{1,2,\cdots,p_i - 1\})^{2m+2}$. We claim for all $i$ that $\D_{\Heis_{2m+1}(\Z)}(\kappa^\al_i) \approx \log(\|\kappa^{\al_i}\|_S))^{2m+1}$. It is clear that $\bar{\kappa}^{\al_i} \neq 1$ in $ \Heis_{2m+1}(\Z)/(\Heis_{2m+1}(\Z))^{p_i}$. \cite[3.B2]{asymptotic_group} implies that $\|\kappa^{\al_i}\|_S \approx \sqrt{|\al_i|}$, and the Prime Number Theorem \cite[1.2]{tenenbaum} implies that $\log(|\al_i|) \approx p_i$. Subsequently, $\log(\|\kappa^{\al_i}\|_S) \approx p_i$, and thus, $(\log(\|\kappa^{\al_i}\|_S))^{2m+1} \approx p_i^{2m+1}$. Given that $|\Heis_{2m + 1}(\Z) / \pr{\Heis_{2m + 1}(\Z)}^{p_i}| = p_i^{2m+1}$, we establish that $(\log(\|\ga_i\|_S))^{2m+1} \approx \D_{\Heis_{2m+1}(\Z)}(\ga_i)$ by demonstrating for all surjective homomorphisms $\map{\varphi}{\Heis_{2m+1}(\Z)}{Q}$ to finite groups  satisfying $|Q| < p_i^{2m+1}$ that $\varphi(\kappa)^{\al_i} = 1$.
	
	\cite[Thm 2.7]{Hall_notes} implies that we may assume that $|Q| = q^\beta$ where $q$ is a prime.
	Since $\varphi(\kappa^\al_i) = 1$ when $\varphi(\kappa) =  1$, we may assume that $\varphi(\kappa) \neq 1$. Since $[\mu_i, \nu_i] = \kappa$ for all $i$, it follows that $\varphi(\nu_i),\varphi(\mu_j) \neq 1$ for all $i,j$ and that $|Q| \geq q^{2m+1}$ (see the second paragraph after \cite[Defn 8.2]{computational_group}). 
	
	Suppose $Q$ is a $p_i$-group. If $\varphi(\ga_i) \neq 1$, then Proposition \ref{alternative} implies that $|Q| = p_i^{2m+1}$ and that there are no proper quotients of $Q$ where the image of $\varphi(\ga_i)$ does not vanish. In particular, there are no proper quotients of $\Heis_{2m + 1}(\Z) / (\Heis_{2m + 1}(\Z))^{p_i}$ where $\sigma_{p_i}(\ga_{i})$ does not vanish. Thus, we may assume that $q \neq p_i$.

 If $q > p_i$, then we have $\text{Ord}_Q(\varphi(\nu_i)),\text{Ord}_Q(\varphi(\mu_j)) \geq p_i$ for all $i,j$. That implies $|\De_i : \De_{i-1}| > p_i$. Thus, the second paragraph after \cite[Defn 8.2]{computational_group} implies that $|Q| > p_i^{2m+1}$; hence, we may disregard this possibility. We now assume that $Q$ is a $q$-group where $q < p_i$. If $q^\beta < p$, then $|Q| \mid \al_i$. Since the order of an element of a finite group divides the order of the group, we have $\la \mid \al_i$ where $\la = \text{Ord}_{Q}(\varphi(\kappa))$. Thus, $\varphi(\ga_i) = 1$. 
	
	Hence, we may assume that $Q$ is a $q$-group where $q < p_i$ and $p_i < q^\beta < p_i^{2m+1}$. There exists $v$ such that $q^{(2m+1)v} < p_i^{2m+1} < p^{(2m+1)(v+1)}$. Thus, we may write $\beta = vt + r$ where $t \leq 2m+1$ and $0 \leq r < t$. By construction, $q^{(2m+1)t + r} \leq \al_i$, and since $q < p_i$, it follows that $q^{(2m+1)t + r} \mid \al_i$. Subsequently, $\la \mid \al_i$ and $\varphi(\kappa^{\al_i}) = 1$ as desired.
\end{proof}
\begin{cor}
	Let $\Heis_{2m+1}(\Z)$ be the integral Heisenberg group. Then $\Farb_{\Heis_{2m+1}(\Z)}(n) \approx (\log(n))^{2m+1}$.
\end{cor}

\section{Proof of Theorem \ref{word_precise}} 
Our goal for Theorem \ref{word_precise} is to demonstrate that $\Farb_{\Ga}(n) \approx (\log (n))^{\psi_{\text{RF}}(\Ga)}$. Proposition \ref{torsion_reduction} implies that we may assume that $\Ga$ is torsion free. We proceed with the proofs of the upper and lower asymptotic bounds for $\Farb_{\Ga}(n)$ separately since they require different strategies. We start with the upper bound as its proof is simpler. 

For the upper bound, our task is to prove for any non-identity element $\ga \in \Ga$ there exists a surjective homomorphism to a finite group $\map{\varphi}{\Ga}{Q}$ such that $\varphi (\ga) \neq 1$ and $|Q| \leq C_0 \: (\log(C_0 \: \|\ga\|_S))^{\psi_{\text{RF}}(\Ga)}$ for some $C_0 \in \N$. When $\ga \notin \sqrt[Z(\Ga)]{\Ga_{c(\Ga)}}$, we pass to the quotient given by $\Ga / \sqrt[Z(\Ga)]{\Ga_{c(\Ga)}}$ and then appeal to induction on step length. Otherwise, for $\ga \in \sqrt[Z(\Ga)]{\Ga_{c(\Ga)}}$, we find a choice of an admissible quotient of $\Ga$ with respect to some primitive central element in which $\ga$ has a non-trivial image.

\begin{prop}\label{word_tf_upper_bound}
	Let $\Ga$ be a torsion free admissible group. Then $\Farb_{\Ga}(n) \preceq (\log(n))^{\psi_{\text{RF}}(\Ga)}$.
\end{prop}
\begin{proof}
Let $\set{\De_i}_{i=1}^{h(\Ga)}$ be a cyclic series with a compatible generating subset $\set{\xi_i}_{i=1}^{h(\Ga)}$. By assumption, there exist integers $\set{t_i}_{i=1}^{h(\Ga_c)}$ such that$\set{\xi_i^{t_i}}_{i=1}^{h(\Ga_c)}$ is a basis for $\Ga_c$ and there exist $a_i \in \Ga_{c-1}$ and $b_i \in \Ga$ such that $\xi_i^{t_i} =[a_i,b_i]$ for $1 \leq i \leq h(\Ga_c)$.

Suppose $\ga \in \Ga$ such that $\|\ga\|_S \leq n$. Using the Mal'tsev coordinates of $\ga$, we may write $\ga = \prod_{i=1}^{h(\Ga)}\xi_{i}^{\al_i}$, and Lemma \ref{coord_bound} implies that $|\al_i| \leq C_1 \: n^{c(\Ga)}$ for some $C_1 \in \N$ for all $i$. We construct a surjective homomorphism $\map{\varphi}{\Ga}{Q}$ to a finite group where $\varphi(\ga) \neq 1$ and $|Q| \leq C_2 \: (\|\ga\|_S)^{\psi_{\text{RF}}(\Ga)}$ for some $C_2 > 0$.

Letting $M = \sqrt[Z(\Ga)]{\Ga_{c(\Ga)}}$, suppose that $\pi_{M}(\ga) \neq 1$. Passing to the group $\Ga / M$, the inductive hypothesis implies there exists a surjective homomorphism $\map{\varphi}{\Ga /M}{Q}$ such that $\varphi(\bar{\ga}) \neq 1$, and  $\Depth_\Ga(\ga) \leq C_3 \: (\log(C_3 \:  n))^{\psi_{\text{RF}}(\Ga / M)}$ for some $C_3 \in \N$. Proposition \ref{induction_prec_rf} implies that $\psi_{\text{RF}}(\Ga / M) \leq \psi_{\text{RF}}(\Ga)$, and thus, $\D_\Ga(\ga) \leq C_3 \pr{\log(C_3 \: n)}^{\psi_{\text{RF}}(\Ga)}$.

Otherwise, we may assume that $\ga \in M$. Therefore, we may write $\ga = \prod_{i=1}^{h(\Ga_c)}\xi_i^{\al_i}$, and since $\ga \neq 1$, there exists $1 \leq j \leq h(\Ga_c)$ such that $\al_j \neq 0$.  The Prime Number Theorem \cite[1.2]{tenenbaum} implies that there exists a prime $p$ such that $p \nmid |\al_j|$ and $p \leq C_4 \: \log (C_4 \: |\al_j|)$ for some $C_4 \in \N$. If $\Ga / \La_j$ is a choice of an admissible quotient with respect to $\xi_j$, then $\bar{\ga} \neq 1$ in $\Ga / \La_j \cdot \Ga^p$. Corollary \ref{p_th_finite_index} implies $|\Ga / \La_j \cdot \Ga^p| \leq C_4^{h(\Ga / \La_j)} \cdot \pr{\log (C_4 \: |\al_j|)}^{h(\Ga / \La_j)}$. Proposition \ref{cyclic_series_maximal_one_dimensional_center} implies that $h(\Ga / \La_j) \leq \psi_{\text{RF}}(\Ga)$. Thus, we have $\D_\Ga(\ga) \leq C_5 \:  (\log(C_5 \: n))^{\psi_{\text{RF}}( \Ga)}$ for some $C_5 \in \N$. Hence, $\Farb_{\Ga}(n) \preceq \pr{\log(n)}^{\psi_{\text{RF}}( \Ga)}$.
\end{proof}

In order to demonstrate that $(\log (n) )^{\psi_{\text{RF}}(\Ga)} \preceq \Farb_{\Ga}(n)$, we require a sequence of elements $\{\ga_j\} \subseteq \Ga$ such that $C_1 \: (\log(C_1 \: \|\ga_j\|_S))^{\psi_{\text{RF}}(\Ga)} \leq \Depth_{\Ga}(\ga_j)$ for some $C_1 \in \N$ independent of $j$.  That entails finding elements that are of high complexity with respect to residual finiteness, i.e. non-identity elements that have relatively short word length in comparison to the order of the minimal finite group required to separate them from the identity. 

\begin{prop}\label{word_tf_lower_bound}
	Let $\Ga$ be torsion free admissible group. Then $\pr{\log(n)}^{\psi_{\text{RF}}(\Ga)} \preceq \Farb_{\Ga}(n)$.\end{prop}
\begin{proof}
Let $\Ga / \La$ be a choice of a maximal admissible quotient of $\Ga$. There exists a $g \in Z(\Ga) - \set{1}$ such that $\Ga / \La$ is an admissible quotient with respect to $\ga$. Moreover, there exists a $k \in \Z- \set{0}$, $a \in \Ga_{c -1}$, and $b \in \Ga$ such that $g^k = [a,b]$. If $g$ is not trivial, then there exists a primitive $x_\La \in Z(\Ga)$ such that $x^s = g$ for some $s$. In particular, $x_\La$ is a primitive, central non-trivial element such that $x_\La^{s \: k} = [a,b]$.

Let $\set{\De_{i,\La}}_{t=1}^{h(\Ga)}$ be a choice of cyclic series with a compatible generating subset given by $\set{\xi_{i,\La}}_{i=1}^{h(\Ga)}$ that satisfy Proposition \ref{minimal_quotient_cyclic_series} for $\La$ such that $\xi_{1,\La} = x_\La$. Proposition \ref{minimal_quotient_cyclic_series} implies that $\Ga / \La$ is a choice of an admissible quotient with respect to $\xi_{1,\La}$. Let $\al_{j,\La} = (\lcm\{1, 2, \cdots, p_{j,\La} - 1\})^{\psi_{\text{RF}}(\Ga) + 1}$ where  $\{p_{j,\La}\}$ is an enumeration of primes greater than $\Bas(\Ga / \La)$. Letting $\ga_{j,\La} = \xi_1^{\al_{j,\La}}$, we claim that $\{\ga_{j,\La}\}$ is our desired sequence.

Before continuing, we make some remarks. The value $\Bas(\Ga / \La)$ depends on the choice of a maximal admissible quotient of $\Ga$. To be more specific, if $\Ga / \Omega$ is another choice of a maximal admissible quotient of $\Ga$, then, in general, $\Ga / \La \ncong \Ga / \Omega$, and subsequently, $\Bas(\Ga / \La) \neq \Bas(\Ga / \Omega)$. As a natural consequence, the sequence of elements $\set{\ga_{i,\La}}$ depends on the choice of a maximal admissible quotient of $\Ga$. However, we will demonstrate that the given construction will work for any choice of a maximal admissible quotient we take.

We claim for all $j$ that $\Depth_{\Ga}(\ga_{j,\La}) \approx (\log(p_{j,\La}))^{\psi_{\text{RF}}(\Ga)}$. It is evident that $\overline{\ga_{j,\La}} \neq 1$ in $\Ga / \La \cdot \Ga^{p_{j,\La}}$, and Proposition \ref{p_th_finite_index} implies that $|\Ga / \La \cdot \Ga^{p_{j,\La}}| = p_{j,\La}^{\psi_{\text{RF}}(\Ga)}$. To proceed, we demonstrate for all surjective homomorphisms $\map{\varphi}{\Ga}{Q}$  to finite groups where $|Q| < p_{j,\La}^{\psi_{\text{RF}}(\Ga)}$ that $\varphi(\ga) = 1$.

\cite[Thm 2.7]{Hall_notes} implies that we may assume that $|Q| = q^\beta$ where $q$ is a prime. If $\xi_{1,\La}\in \ker(\varphi)$, then $\varphi(\ga_{j,\La}) = 1$. Hence, we may assume that $\varphi(\xi_{1,\La}) \neq 1$. Proposition \ref{reduction_one_dim_word} implies that $\varphi(\ga_{j,\La}) \neq 1$ if and only if $\pi_{\varphi(\La)}(\varphi(\ga_{j,\La})) \neq 1$. Thus, we may restrict our attention to surjective homomorphisms that factor through $\Ga / \La$, i.e. homomorphisms $\map{\varphi}{\Ga}{Q}$ where $\varphi(\La) \cong \set{1}$.

Suppose that $q = p_{j,\La}$. If $\varphi(\ga_{j,\La}) = 1$, then there is nothing to prove. So we may assume that $\varphi(\ga_{j,\La}) \neq 1$. Since $|Q| \leq p_{j,\La}^{\psi_{\text{RF}}(\Ga)}$, Proposition \ref{alternative} implies that $|Q| = p_{j,\La}^{\psi_{\text{RF}}(\Ga)}$ and that if $N$ is a proper quotient of $Q$ with natural projection given by $\map{\rho}{Q}{N}$, then $\rho(\varphi(\ga_{j,\La})) = 1$. We have two natural consequences. There are no proper quotients of $\Ga / \La \cdot \Ga^{p_{j,\La}}$ where $\varphi(\ga_{j,\La})$ has non-trivial image. Additionally, If $\map{\varphi}{\Ga}{Q}$ is a surjective homomorphism to a finite $p_{j,\La}$-group where $|Q| < p_{j,\La}^{\psi_{\text{RF}}(\Ga)}$ then $\varphi(\ga_{j,\La}) = 1$. Thus, we may assume that $q \neq p_{j,\La}$.

Suppose that $q > p_{j,\La}$. Since $\map{\widetilde{\varphi}}{\Ga / \La}{Q}$ is a surjective homomorphism to a finite $q$-group where $q > \Bas(\Ga / \La)$, Proposition \ref{p_group_lower_bounds} implies that $|Q| > p_{j,\La}^{\psi_{\text{RF}}(\Ga)}$. Hence, we may assume that $q < p_{j,\La}$.

Now suppose that $Q$ is a $q$-group where $|Q| < p_{j,\La}$. By selection, it follows that $|Q|$ divides $\al_{j,\La}$. Since the order of an element divides the order of the group, we have that $\text{Ord}_Q(\varphi(\xi_{1,\La}))$ divides $\al_{j,\La}$. That implies $\varphi(\ga_{j,\La}) = 1$. 

Now suppose that $Q$ is a $q$-group where $q < p_{j,\La}$ and $q^\beta > p_{j,\La}$. Thus, there exists $\nu \in \N$ such that $q^{\nu \: \psi_{\text{RF}}(\Ga)} < p_{j,\La}^{ \psi_{\text{RF}}(\Ga)} < q^{(\nu + 1) \: \: \psi_{\text{RF}}(\Ga)}$. Subsequently, we may write $\beta = \nu t + r$ where $t \leq \psi_{\text{RF}}(\Ga)$ and $0 \leq r < \nu$. By construction, $q^{vt + r} \leq \al_{j,\La}$, and since $q < p_{j,\La}$, it follows that $q^\beta = q^{vt + r} \mid \al_j$. Given that the order of any element in a finite group divides the order of the group, it follows that $\text{Ord}_Q(\varphi(\xi_{1,\La}))$ divides $\al_{j,\La}$. Thus, $\varphi (\ga_{\La,j}) = 1$, and therefore, $\Depth_\Ga (\ga_{j,\La}) = p_{j,\La}^{\psi_{\text{RF}}(\Ga)}$.

Since $\ga_{j,\La} \in \Ga_c$ where $c$ is the step length fo $\Ga$, \cite[3.B2]{asymptotic_group} implies $(\|\ga_{j,\La}\|_S) \approx (|\al_{j,\La}|)^{1/c}$, and the Prime Number Theorem \cite[1.2]{tenenbaum} implies $\log(|\al_{j,\La}|) \approx p_{j,\La}$. Hence, $\pr{\log(\|\ga_{j,\La}\|_S)}^{\psi_{\text{RF}}(\Ga)} \approx p_{j,\La}^{\psi_{\text{RF}}(\Ga)}.$ Thus, $\D_{\Ga}(\ga_{j,\La}) \approx \pr{\log(\|\ga_{j,\La}\|_S)}^{\psi_{\text{RF}}(\Ga)}$, and subsequently, $\pr{\log(n)}^{\psi_{\text{RF}}(\Ga)} \preceq \Farb_{\Ga}(n)$.
\end{proof}

\begin{proof}\textit{Theorem \ref{word_precise}}\\
Let $\Ga$ be an infinite admissible group. Proposition \ref{torsion_reduction} implies that $\Farb_{\Ga}(n) \approx \Farb_{\Ga/T(\Ga)}(n)$. Proposition \ref{word_tf_upper_bound} implies that $\Farb_{\Ga/T(\Ga)}(n) \preceq \pr{\log(n)}^{\psi_{\text{RF}}(\Ga)}$, and Proposition \ref{word_tf_lower_bound} implies $\pr{\log(n)}^{\psi_{\text{RF}}(\Ga)} \preceq \Farb_{\Ga/T(\Ga)}(n)$. Thus, $\Farb_{\Ga}(n) \approx \pr{\log(n)}^{\psi_{\text{RF}}(\Ga)}$.\end{proof}
\section{Cyclic series, lattices in nilpotent Lie groups, and Theorem \ref{maltsev_invariant}}
Let $\Ga$ be a torsion free admissible group, and let $G$ be its Mal'tsev completion. Let $\Ga / \La$ be a choice of a maximal admissible quotient. The main task of this section is the demonstration that the value $h(\Ga / \La)$ is a well-defined invariant of the Mal'tsev completion of $\Ga$. Thus, we need to establish some properties of cocompact lattices in admissible Lie groups. We start with the following lemma that relates the Hirsch lengths of centers of cocompact lattices within the same admissible Lie group.

\begin{lemma}\label{cocompact_center}
	Let $G$ be an admissible Lie group with two cocompact lattices $\Ga_1$ and $\Ga_2$. Then $\text{dim}(G) = h(Z(\Ga_1)) = h(Z(\Ga_2))$.
\end{lemma}
\begin{proof}
	This proof is a straightforward application of \cite[Lem 1.2.5]{Dekimpe}.
\end{proof}

We now introduce the notion of one parameter families of group elements of a Lie group.
\begin{defn}
	Let $G$ be a connected, simply connected Lie group. We call a map $\map{f}{\R}{G}$ a \emph{one parameter family of group elements of $G$} if $f$ is an injective group homomorphism from the real line with addition.
\end{defn}

Let $\set{\Ga, \De_i, \xi_i}$ be an admissible $3$-tuple such that $\Ga$ is torsion free. Via the exponential map and \cite[Lem 1.2.5]{Dekimpe}, the maps given by $f_{\Ga,i}(t) = \exp(t \: v_i)$ are one parameter families of group elements. The discussion below \cite[Thm 1.2.4 Pg 9]{Dekimpe} implies that we may uniquely write each $g \in G$ as $g = \prod_{i=1}^{h(\Ga)}f_{\Ga,i}(t_i)$ where $G$ is the Mal'tsev completion of $\Ga$. \begin{defn} We say that the one parameter families $f_{\Ga,i}$ are \emph{associated} to $\set{\Ga, \De_i, \xi_i}$. \end{defn} 
We characterize when a discrete subgroup of an admissible Lie group is a cocompact lattice based on how it intersects a collection of one parameter families of group elements.

\begin{prop}\label{lattice_characterization}
	Let $G$ be a $\Q$-defined admissible Lie group, and suppose that $\Ga$ is a discrete subgroup of $G$. Suppose there exists a collection of one parameter families of group elements of $G$, written as $\map{f_i}{\R}{G}$ for $1 \leq i \leq \text{dim}(G)$, such that $G$ is homeomorphic to $\prod_{i=1}^{\text{dim}(G)}f_i(\R)$. Then $\Ga$ is a cocompact lattice in $G$ if and only if $\Ga \cap f_i(\R) \cong \Z$ for all $i$.
\end{prop}
\begin{proof}
	Let $\map{\rho}{G}{G / \Ga}$ be the natural projection onto the space of cosets. Suppose that there exists $i_0$ such that $f_{i_0}(\R) \cap \Ga \ncong \Z$. Since $\Ga$ is discrete in $\Ga$, $\Ga \cap f_{i_0}(\R)$ is a discrete subset of $f_{i_0}(\R)$. Given that $\Ga \cap f_{i_0}(\R)$ is discrete and not infinite cyclic, we have $\Ga \cap f_{i_0}(\R) \cong \set{1}$. Hence, each element of the sequence $\{ f_{i_0}(t) \}_{t \in \N}$ projects to a unique element of $G / \Ga$. Thus, $\{ \rho(f_{i_0}(t))\}_{t\in \N}$ is an infinite sequence in $G / \Ga$ with no convergent subsequence. Hence, $\Ga$ is not a cocompact lattice of $G$
	
	Now suppose that $f_{i}(\R) \cap \Ga \cong \Z$ for all $i$. That implies for each $i \in \set{1, \cdots,  h(\Ga)}$ there exists $t_i > 0$ such that $\Ga \cap f_i(\R) \cong \{f_i(n \: t_i) \:| \: n \in \Z  \}$. Letting $E = \prod_{i=1}^{\text{dim}(G)}f_i([0,t_i])$, we claim that $E$ is compact and that $\rho(E) \cong G / \Ga$. 
	
	Let $\map{f}{\R^{\text{dim}(G)}}{G}$ be the continuous map given by $f(\pr{a_i,\cdots,a_{\text{dim}(G)}}) = \prod_{i=1}^{\text{dim}(G)}f_i(a_i)$. Since $\prod_{i=1}^{\text{dim}(G)}[0,t_i]$ is a closed and bounded subset of $\R^{\text{dim}(G)}$, the Heine-Borel theorem implies that $\prod_{i=1}^{\text{dim}(G)}[0,t_i]$ is compact. Since $f$ is continuous, $E$ is compact.
	
	We now claim that each coset of $\Ga$ in $G$ has a representative in $E$. Let $g = \prod_{i=1}^{\text{dim}(G)}f_i(\ell_i)$. For each $i$, there exists $s_i \in \Z$ such that $s_i \: t_i \leq \ell_i \leq (s_i + 1) \: t_i$. Let $k_i = \ell_i - s_i \: t_i$ and write $h \in E$ to be given by $h = \prod_{i=1}^{\text{dim}(G)}f(k_i)$. By construction, $\rho(h) = \rho(g)$, and subsequently, $\rho(g) \in \rho(E)$. Thus, $\rho(E) =  \rho(G)$. Since $G/ \Ga$ is the image of a compact set under a continuous map, $G/ \Ga$ is compact. \cite[Thm 2.1]{rag} implies that $\Ga$ is a cocompact lattice of $G$.
\end{proof}

These next two propositions give some structural information needed about the Mal'tsev completion of a torsion free admissible group and some structural information of choices of an admissible quotients with respect to some primitive, central non-trivial element.
\begin{prop}\label{one_dimensiona_center_completion_closed}
	Let $\Ga$ be a torsion free admissible group. Let $\ga \in Z(\Ga) -\set{1}$ be a primitive element, and let $\Ga / \La$ be a choice of an admissible quotient with respect to $\ga$. Suppose that $G$ is the Mal'tsev completion of $\Ga$, and let $H$ be the Mal'tsev completion of $\La$. Then $H$ is isomorphic to a closed, connected, normal subgroup of $G$.
\end{prop}
\begin{proof}
	Proposition \ref{minimal_quotient_cyclic_series} there exists a cyclic series $\set{\De_i}_{i=1}^{h(\Ga)}$ and compatible generating subset $\set{\xi_i}_{i=1}^{h(\Ga)}$ satisfying the following. There exists a subset $\{\xi_{i_s}\}_{s=1}^{h(\La)} \subseteq \set{\xi_{i}}_{i=1}^{h(\Ga)}$ such that if $W_s = \innp{\xi_{i_t}}_{t=1}^s$, then $\{W_{s}\}_{s=1}^{h(\La)}$ is a cyclic series for $\La$ with a compatible generating subset given by $\{\xi_{i_s}\}_{s=1}^{h(\La)}$ where $\xi_1 = \ga$. Let $\set{f_{\Ga,i}}_{i=1}^{h(\Ga)}$ be the one parameter families of group elements of $G$ associated to the admissible $3$-tuple  $\set{\Ga,\De_i,\xi_i}$.
	
	\cite[Thm 1.2.3]{Dekimpe} implies that we may view $H$ as a connected subgroup of $G$. We  proceed by induction on $h(\Ga)$ to demonstrate that $H$ is a closed and normal subgroup of $G$. If $h(\Ga) = 1$, then $\Ga = \Z$. It then follows that $G$ is Lie isomorphic to $\R$ and that $H \cong \{1\}$. Now our claim is evident. 
	
	Now suppose that $h(\Ga) > 1$. If $h(Z(\Ga)) = 1$, then we may take $\La = \set{1}$ which implies $H \cong \{1\}$. Thus, our claims are evident. Now suppose that $h(Z(\Ga)) > 1$. Let $\Omega = \innp{\xi_i}_{i=1}^{h(Z(\Ga))}$, and let $K$ be the Mal'tsev completion of $\Omega$. By \cite[Lem 1.2.5]{Dekimpe}, it follows that $K \leq Z(G)$. Thus, $K$ is a closed, connected, normal subgroup of $G$. 
	
	We claim that $H / K$ is Mal'tsev completion of $\pi_K(\La)$. We may write $H = \prod_{s=1}^{h(\La)} f_{\Ga,i_s}(\R).$ Since $\La$ is a cocompact lattice of $H$, Proposition \ref{lattice_characterization} implies that $\La \cap f_{\Ga,i_s}(\R) \cong \Z$ for all $1 \leq s \leq \ell$. By Proposition \ref{lattice_characterization} again, we have that $K \cap \La$ is a cocompact lattice of $K$. \cite[Prop 5.1.4]{greenleaf_corwin} implies that $\pi_K(\La)$ is a cocompact lattice in $H / K$.
	
	Observe that $\pi_K(\La) \cong \La / \Omega$. We have that $\La / \Omega$ satisfies Proposition \ref{one_dim_center} for $\pi_K(\xi_1)$. Thus, the inductive hypothesis implies that $H / K$ is a closed normal subgroup of $G / K$. Since $H$ isomorphic to the pullback of the closed normal subgroup of $G / K$, $H$ is a closed normal subgroup of $G$.
\end{proof}

The next proposition that $G$ is the Mal'tsev completion of $\Ga$ with a choice of an admissible quotient $\Ga / \La$ with respect to a primitive, central non-trivial element of $\Ga$, then the Mal'tsev completion of $H$ intersects any cocompact as a cocompact lattice.
\begin{prop}\label{completion_intersect_cocompact}
	Let $\Ga$ be a torsion free admissible group. Let $\ga \in Z(\Ga) - \set{1}$ be a primitive element, and let $\Ga / \La$ be a choice of an admissible quotient with respect to $\ga$, $G$ be the Mal'tsev completion of $\Ga$, and let $H \leq G$ be the Mal'tsev completion of $\La$. If $\Omega \leq G$ is another cocompact lattice of $G$, then $\Omega \cap H$ is a cocompact lattice of $H$.
\end{prop}
\begin{proof}
	Proposition \ref{minimal_quotient_cyclic_series} implies that there exists a cyclic series $\set{\De_i}_{i=1}^{h(\Ga)}$ and a compatible generating subset $\set{\xi_i}_{i=1}^{h(\Ga)}$ satisfying the following. There exists a subset $\set{\xi_{i_j}}_{j=1}^{h(\La)}$ such that the groups $\set{W_i}$ where $W_i \cong \innp{\xi_{i_j}}_{j=1}^i$ form a cyclic series for $\La$ with a compatible generating subset given by $\set{\xi_{i_j}}_{j=1}^{h(\La)}$. Let $\set{f_{i}}_{i=1}^{h(\Ga)}$ be the associated one parameter families of group elements of the Mal'tsev completion $G$ of $\Ga$. We may write $G \cong \prod_{i=1}^{h(\Ga)}f_{i}(\R)$. By construction, $H \cong \prod_{j=1}^{h(\La)}f_{i_j}(\R).$ Proposition \ref{lattice_characterization} implies that $\Omega \cap f_{i,\Ga} \cong \Z$ for all $i$.  In particular, $\Omega \cap f_{i_j,\Ga}(\R) \cong \Z$ for all $j$. Proposition \ref{lattice_characterization} implies that $\Omega \cap H$ is a cocompact lattice in $H$ as desired.
	\end{proof}

The following lemma demonstrates that you can select a cyclic series and compatible generating subset for a cocompact lattice in an admissible Lie group by intersecting the lattice with a collection of one parameter families of group elements.
\begin{lemma}\label{cyclic_series_intersect_completion}
	Let $G$ be a $\Q$-defined admissible Lie group with a cocompact lattice $\Ga$. Let $f_i$ be a collection of one parameter families of elements of $G$ such that $G$ is homeomorphic to $\prod_{i=1}^{\text{dim}(G)}f_i(\R)$. Let $\innp{\xi_i} \cong f_i(\R) \cap \Ga$. Then the groups given by $\De_i  = \innp{\xi_t}_{t=1}^i$ form a cyclic series for $\Ga$ with a compatible generating subset given by $\set{\xi_i}_{i=1}^{h(\Ga)}$.
\end{lemma}
\begin{proof}
We proceed by induction on the dimension of $G$. If $\text{dim}(G) = 1$, then our statement is evident. Now suppose that $\text{dim}(G) > 1$. If we let $H \cong \prod_{i=1}^{\text{dim}(G) - 1}f_i(\R)$, then Proposition \ref{lattice_characterization} implies that $\Ga \cap H$ is a cocompact lattice in $H$. The inductive hypothesis implies that the elements $\xi_i$ given by $\innp{\xi_i} \cong f_i(\R) \cap \Ga$ satisfy the following. The groups given by $\De_t = \innp{\xi_i}_{i=1}^t$ for $1 \leq t \leq \text{dim}(G) - 1$ form a cyclic series for $\Ga \cap H$ with a compatible generating subset given by $\set{\xi_i}_{i=1}^{\text{dim}(G) - 1}$. Since $\Ga$ is a cocompact lattice in $G$, Proposition \ref{lattice_characterization} implies that $f_{\text{dim}(G)}(\R) \cap \Ga \cong \Z$. Letting $\De_{\text{dim}(G)} = \innp{\De_{\text{dim}(G) - 1}, \xi_{\text{dim}(G)}}$, we have that the groups given by $\set{\De_i}_{i=1}^{\text{dim}(G)}$ form a cyclic series for $\Ga$ with a compatible generating subset given by $\set{\xi_i}_{i=1}^{\text{dim}(G)}$. \end{proof}

 Let $\Ga$ be a torsion free admissible group. We now demonstrate that the value $\psi_{\text{RF}}(\Ga)$ is a well-defined invariant of the Mal'tsev completion of $\Ga$.
\begin{prop}\label{same_growth_maltsev_completion}
	Let $G$ be a $\Q$-defined admissible Lie group, and suppose that $\Ga_1$ and $\Ga_2$ are two cocompact lattices of $G$. Then $\psi_{\text{RF}}(\Ga_1) = \psi_{\text{RF}}(\Ga_2)$.
\end{prop} 
\begin{proof}
	If $h(Z(\Ga_1)) = 1$, then Proposition \ref{cocompact_center} implies that $h(Z(\Ga_2)) = 1$. It then follows from the definition of $\psi_{\text{RF}}(\Ga_1)$ and $\psi_{\text{RF}}(\Ga_2)$ that $\psi_{\text{RF}}(\Ga_1) = h(\Ga) = \psi_{\text{RF}}(\Ga_2)$.
	
	Therefore, we may assume that $h(Z(\Ga_1)), h(Z(\Ga_2)) \geq 2$. In this case, we demonstrate the equality by showing that $\psi_{\text{RF}}(\Ga_1) \leq \psi_{\text{RF}}(\Ga_2)$ and $\psi_{\text{RF}}(\Ga_2) \leq \psi_{\text{RF}}(\Ga_1)$. 
	
	Let $G$ be the Mal'tsev completion of $G$. Let $\set{\De_i}_{i=1}^{h(\Ga)}$ be a cyclic series for $\Ga$ with a compatible generating subset $\set{\xi_i}_{i=1}^{h(\Ga)}$, and let $f_i$ be the associated one parameter families of group elements. We may write $G \cong \prod_{i=1}^{h(\Ga)}f_i(\R)$.
	
	Let $\{\eta_i\}_{i=1}^{h(\Ga_2)} \subseteq \Ga_2$ such that $\innp{\eta_i} \cong \Ga_2 \cap f_i(\R)$. If we let $W_i \cong \innp{\eta_j}_{j=1}^{i}$, then Proposition \ref{cyclic_series_intersect_completion} implies that $\set{W_i}_{i=1}^{h(\Ga_2)}$ is a cyclic series for $\Ga_2$ with a compatible generating subset given by $\set{\eta_i}_{i=1}^{h(\Ga_2)}$.  Let $\xi_i$ be a central element of the compatible generating subset of $\Ga$, and let $\Ga_1 / \La$ be a choice of an admissible quotient with respect to $\xi_i$. Let $H$ be the Mal'tsev completion of $\La$. Since $\pi_\La(\xi_i) \cong Z(\Ga / H)$, it is evident that $\innp{\pi_{H}(f_i(\R))} \cong Z(\Ga / H)$. In particular, $\pi_H(\eta_i) \neq 1$. Proposition \ref{completion_intersect_cocompact} implies that $H \cap \Omega$ is a compact lattice of $H$ and Proposition \ref{one_dimensiona_center_completion_closed} implies that $H$ is a closed connected normal subgroup of $G$. Thus, \cite[Prop 5.1.4]{greenleaf_corwin}  implies that $\pi_H(\Omega)$ is a compact lattice in $G / H$. Proposition \ref{cocompact_center} implies that $h(\Ga_2 / \La) \cong h(\pi_H(\Ga_2))$, it follows that $\pi_H(\Ga_2)$ satisfies the conditions of Proposition \ref{one_dim_center}. If we let $\Ga_2 / \Omega$ be a choice of an admissible quotient with respect to $\eta_i$, it follows that $h(\Ga / \Omega) \leq \pi_H(\Ga_2) \leq h(\Ga / \La)$. By Proposition \ref{cyclic_series_maximal_one_dimensional_center}, $h(\Ga_2 / \Omega) \leq \psi_{\text{RF}}(\Ga_1)$. 	\cite[Lem 1.2.5]{Dekimpe} implies that $\eta_i \in Z(\Ga_2)$, and thus, the above inequality holds for each element of the compatible generating subset of $\Ga_2$ in $Z(\Ga_2)$. Therefore, Proposition \ref{cyclic_series_maximal_one_dimensional_center} implies that $\psi_{\text{RF}}(\Ga_2) \leq \psi_{\text{RF}}(\Ga_1)$. By interchanging $\Ga_1$ and $\Ga_2$, we have $\psi_{\text{RF}}(\Ga_1) \leq \psi_{\text{RF}}(\Ga_2)$.\end{proof}

We now come to the main result of this section.
\begin{proof}\textit{Theorem \ref{maltsev_invariant}}

	Suppose that $\Ga_1$ and $\Ga_2$ are two infinite admissible groups such that $\Ga_1 / T(\Ga_1)$ and $\Ga_2 / T(\Ga_2)$ have isomorphic Mal'tsev completions. Proposition \ref{torsion_reduction} implies that $\Farb_{\Ga_1}(n) \approx \Farb_{\Ga_1/T(\Ga_1)}(n)$ and  $\Farb_{\Ga_2}(n) \approx \Farb_{\Ga_2 / T(\Ga_2)}(n)$. Theorem \ref{word_precise} implies $\Farb_{\Ga_1}(n) \approx \pr{\log(n)}^{\psi_{\text{RF}}(\Ga_1)}$ and $\Farb_{\Ga_2/T(\Ga_2)}(n) \approx \pr{\log(n)}^{\psi_{\text{RF}}(\Ga_2)}$. Proposition \ref{same_growth_maltsev_completion} implies $\psi_{\text{RF}}(\Ga_1 / T(\Ga_1)) = \psi_{\text{RF}}(\Ga_2 / T(\Ga_2))$. Thus, $\Farb_{\Ga_1}(n) \approx \Farb_{\Ga_2}(n)$.
\end{proof}

\section{Some Examples and the Proof of Theorem \ref{applications 2}}

\subsection{Free nilpotent groups and Theorem \ref{applications 2}(i)}
\begin{defn}
Let $\text{F}(X)$ be the free group of rank $m$ generated by $X$. We define $\Free(X,c,m) = \text{F}(X) / (\text{F}(X))^{c + 1}$ as \emph{the free nilpotent group of step size $c$ and rank $m$ on the set $X$}. 
\end{defn}

Following \cite[Sec 2.7]{NGA}, we construct a cyclic series for $\Free(X,c,m)$ and a compatible generating subset using iterated commutators in the set $X$. 
\begin{defn}\label{compatible_free}
We call elements of $X$ \emph{basic commutators of weight $1$ of $\Free(X,c,m)$}, and we choose an arbitrary linear order for weight $1$ basic commutators. If $\ga_1$ and $\ga_2$ are basic commutators of weight $i_1$ and $i_2$, respectively, then $[\ga_{1},\ga_2]$ is a \emph{basic commutator of weight $i_1 + i_2$ of $\Free(X,c,m)$} if $\ga_1 > \ga_2$ and if $\ga_1 = [\ga_{1,1},\ga_{1,2}]$ where $\ga_{1,1}$ and $\ga_{1,2}$ are basic commutators such that $\ga_{1,2} \leq \ga_2$.

Basic commutators of higher weight are greater with respect to the linear order than basic commutators of lower weight. Moreover, we choose an arbitrary linear order for commutators of the same weight.

For $x_{i_0} \in X$, we say that a $1$-fold commutator $\ga$  contains $x_{i_0}$ if $\ga = x_{i_0}$. Inductively, we say that a $j$-fold  commutator $[\ga_1,\ga_2]$ contains $x_{i_0}$ if either $\ga_1$ contains $x_{i_0}$ or $\ga_2$ contains $x_{i_0}$. 

Note that any basic commutator of weight greater or equal to $2$ must contain two distinct commutators of weight $1$ but not necessarily more than $2$. Additionally, if $\ga$ is a basic commutator of weight $k$, then $\ga$ can contain at most $k$ distinct basic commutators of weight $1$.
\end{defn}

It is well known that the number of basic commutators of $\Free(X,c,m)$ is equal to the Hirsch length of $\Free(X,c,m)$. Letting $\mu$ be the M\"{o}bius function, we may write $$h(\Free(X,c,m)) = \sum_{r=1}^c \pr{\frac{1}{r} \sum_{d | r} \mu(d) m^{\frac{r}{d}}}.$$
We label the basic commutators as $\{\xi_i\}_{i=1}^{h(\Free(X,c,m))}$ with respect to the given linear order. 

\begin{defn} One can see that the subgroup series $\{\De_i\}_{i=1}^{h(\Free(X,c,m))}$ where $\De_i = \innp{\xi_t}_{t=1}^i$ is a cyclic series for $\Free(X,c,m)$, and \cite[Cor 2.7.3]{NGA} implies that $\{ \xi_i \}_{i=1}^{h(\Free(X,c,m))}$ is a compatible generating subset. We call $\{\De_i\}_{i=1}^{h(\Free(X,c,m))}$ the \emph{cyclic series of basic commutators for $\Free(X,c,m)$} and $\{\xi_i\}_{i=1}^{h(\Free(X,c,m))}$ \emph{the compatible generating subset of basic commutators for $\Free(X,c,m)$.}
\end{defn}

\begin{prop}\label{extra_gen}
	Let $\Free(X,c,m)$ be the free nilpotent group of step size $c$ and rank $m$ on the set $X = \set{x_i}_{i=1}^{m}$.  Let $\ga$ be a basic commutator of weight $c$ in the set $X$ that contains only $Y \subsetneq X$. There exists a normal subgroup $\Omega$ such that $\Free(X,c,m)/ \Omega$ is torsion free where  $\innp{\pi_\Omega(\ga)} \cong Z(\Free(X,c,m)/ \Omega)$. Additionally, if $\eta$ is a $j$-fold commutator that contains elements of $X / Y$, then $\pi_\Omega(\eta) = 1$.
\end{prop}
\begin{proof}
	Let $\set{\De_i}_{i=1}^{h(\Free(X,c,m))}$ be the cyclic series of basic commutators, and let $\set{\xi_i}_{i=1}^{h(\Free(X,c,m))}$ be the compatible generating subset of basis commutators. By assumption, there exists an $i_0 \in \set{1, \cdots, h(Z(\Free(X,c,m)))}$ such that $\xi_{i_0} = \ga$. Without loss of generality, we may assume that $\xi_1 = \ga$. 
	
	We will construct a normal descending series $\set{K_t}_{t=1}^c$ such that $\Free(X,c,m) / K_t$ is torsion free for each $t$, $\pi_{K_t}(\xi_1) \neq 1$ for each $t$, and if $\eta$ is a $i$-fold commutator that contains only elements of $X / Y$ where $i \geq t$, then $\pi_{K_t}(\eta) = 1$. We will also have that $K_t$ is generated by basis commutators of weight greater than or equal to $t$, and finally, we will have $\innp{\pi_{K_1}(\xi_1)} \cong Z(\Free(X,c,m)/K_1)$. We proceed by induction on $t$. 
	
	Consider the subgroup given by $K_c = \innp{\xi_i}_{i=2}^{h(Z(\Free(X,c,m)))}$. Observe that if $\eta$ is a $c$-fold commutator such that $\eta$ contains only elements of $X/Y$, then it follows by construction that $\pi_{K_1}(\eta) = 1$. Thus, we have the base case.
	
	Now suppose that subgroup $K_t$ has been constructed for $t < c$, and let $\eta$ be a $(t-1)$-fold commutator bracket that contains elements of $X/Y$. It then follows that $[\eta,x_i]$ contains elements of $X/Y$. Thus, $\pi_{K_t}([\eta,x_i])  = 1$ by assumption. Since that is true for all $1 \leq i \leq m$, we have that $\pi_{K_t}(\eta) \in Z(\Free(X,c,m) / K_t)$. Let $W$ be the set of basic commutator brackets $\xi_i$ such that $\xi_i \neq \xi_1$ and $\pi_{K_t}(\xi_i)$ is central. By construction, $\pi_{K_t}(\xi_1) \notin \innp{\pi_{K_t}(W)}$ and if $\eta$ is a $\ell$-fold commutator bracket that contains elements of $X / Y$ where $\ell \geq t-1$, then $\pi_{K_t}(\eta) \in \innp{\pi_{K_t}(W)}$. We set $K_{t-1} \cong \innp{K_t,W}$, and suppose that $\eta$ is a $\ell$-fold commutator that contains elements of $X/Y$ and where $\ell \geq t-1$.  By construction, we have that $\pi_{K_{t-1}}(\eta) = 1$. Since $K_{t-1} \cong \pi_{K_t}^{-1}(\innp{\pi_{K_t}(W)})$, we have that $K_{t-1}$ is normal in $\Free(X,c,m)$ and $K_t \leq K_{t-1}$. Finally, it is evident that $\Free(X,c,m) / K_t$ is torsion free. Hence, induction gives the construction of $K_t$ for all $i$. Additionally, the construction of $K_1$ implies that $Z(\Free(X,c,m) / K_1) \cong \innp{\pi_{K_1}(\xi_1)}$. Thus, by taking $\Omega \cong K_1$, we have our proposition. \end{proof}

\begin{proof}\textit{Theorem \ref{applications 2}(i)}\\ 
Let $c \geq 1$ and $\ell \geq 2$, and let $X_\ell = \{x_i\}_{i=1}^\ell$. Let $\Free(X_\ell,c,\ell)$ to be the free nilpotent group of step size $c$ and rank $\ell$ on the set $X_\ell$. Theorem \ref{word_precise} implies that there exists a natural number $\psi_{\text{RF}}(\Free\pr{X_\ell,c,\ell})$ such that $\Farb_{\Free(X_\ell,c,\ell)}(n) \approx (\log(n))^{\psi_{\text{RF}}(\Free\pr{X_\ell,c,\ell})}$.  We will demonstrate that $(\log(n))^{\psi_{\text{RF}}(\Free\pr{X_\ell,c,\ell})} \preceq (\log(n))^{\psi_{\text{RF}}(\Free(X_c,c,c))}$ for each $\ell > c$, and since $\Free\pr{X_\ell,c,\ell}$ is a nilpotent group of step size $c$ and Hirsch length greater than $\ell$, we will have our desired result.

Let $\{\De_i\}_{i=1}^{h(\Free(X_\ell,c,\ell))}$ be the cyclic series of basic commutators and $\{ \xi_i \}_{i=1}^{h(\Free(X_\ell,c,\ell))}$ be the compatible generating subset of basic commutators for $\Free(X_\ell,c,\ell)$. For each $\xi_i \in Z(\Free(X_\ell,c,\ell))$, let $\Free(X_\ell,c,\ell) / \La_i$ be a choice of an admissible quotient with respect to $\xi_i$. Proposition \ref{cyclic_series_maximal_one_dimensional_center} implies there exists an $i_0 \in \set{1, \cdots, h(Z(\Free(X_\ell,c,\ell)))}$ such that $h(\Ga / \La_{i_0}) = \psi_{\text{RF}}(\Free(X_\ell,c,\ell))$.

For each $\xi_i \in Z(\Free(X_\ell,c,\ell))$, there exists a subset $Y_i \subseteq X$ such that $\xi_i$ is a basic commutator of weight $c$ that contains only elements of $Y_i$.  Proposition \ref{extra_gen} implies that there exists a subgroup $\Omega_i$ such that $\Free(X_\ell,c,\ell) / \Omega_i$ satisfies Proposition \ref{one_dim_center} with respect to $\xi_i$. Moreover, elements of $X - Y_i$ are contained in $\Omega_i$.

There is a natural surjective homomorphism $\map{\rho_i}{\Free(X_\ell,c,\ell)}{\Free(Y_i,c,|Y_i|)}$ given by sending elements of $X - Y_i$ to the identity. Therefore, we have an induced map $\map{\varphi}{\Free(Y_i,c,|Y_i|)}{\Free(X_\ell,c,\ell) / \Omega_i}$ such that $\pi_{\Omega_i} =  \varphi \circ \rho_i$. In particular, $\Free(X_\ell,c,\ell) / \Omega_i \cong \Free(Y_i,c,|Y_i|) / \rho_i(\Omega_i)$. Thus, $\Free(X_\ell,c,\ell) / \Omega_i$ satisfies the conditions of Proposition \ref{one_dim_center} for $\rho_i(\xi_i)$. Proposition \ref{cyclic_series_maximal_one_dimensional_center} implies that $h(\Free(X_\ell,c,\ell) / \La_i) \leq \psi_{\text{RF}}(\Free(Y_i,c,|Y_i|))$.
Since $\Free(X_\ell,c,\ell)$ has step size $c$, we have that $|Y_i| \leq c$ for any $\xi_i \in Z(\Free(X_\ell,c,\ell))$. Additionally, we have that $\Free(Y_i,c,|Y_i|) \cong \Free(X_j,c,j)$ when $|Y_i| = j$. In particular, $\psi_{\text{RF}}(\Free(Y_i,c,|Y_i|)) = \psi_{\text{RF}}(\Free(X_j,c,j))$. By setting $m(c) = \text{max}\{\psi_{\text{RF}}(\Free(X_j,c,j)) \: | \: 1 \leq j \leq c \}$, Proposition \ref{cyclic_series_maximal_one_dimensional_center}
implies $\Farb_{\Free(X_\ell,c,\ell)}(n) \preceq (\log(n))^{m(c)}$. \end{proof}

\subsection{Central products and applications}
The examples we contruct for Theorem \ref{applications 2}(ii), (iii) and (iv) arise as iterated central products of torsion free admissible groups whose centers have Hirsch length $1$. In the given context, Corollary \ref{cong_rf} allows us to compute the precise residually finiteness function in terms of the Hirsch length of the torsion free admissible groups of whom we take the central product.
\begin{defn}
Let $\Ga$ and $\De$ be finitely generated groups, and let $\map{\theta}{Z(\Ga)}{Z(\De)}$ be an isomorphism. We define the \emph{central product of $\Ga$ and $\De$ with respect to $\theta$} as $\Ga \circ_\theta \De = (\Ga \times \De) / K$ where $K = \{(z, \theta(z)^{-1}) \: | \: z \in Z(\Ga) \}$. We define the central product of the groups $\{\Ga_i\}_{i=1}^{\ell}$ with respect to the automorphism $\map{\theta_i}{Z(G_i)}{Z(G_{i+1})}$ for $1 \leq i \leq \ell -1$ inductively. Assuming that $(\Ga_i \circ_{\theta_i})_{i=1}^{\ell}$ has already been defined, we define $(\Ga_i \circ_{\theta_i})_{i=1}^{\ell}$ as the central product of $(\Ga_i \circ_{\theta_i})_{i=1}^{\ell -1}$ and $\Ga_\ell$ with respect to the induced isomorphisms $\map{\bar{\theta}_{\ell-1}}{Z((\Ga_i \circ_{\theta_i})_{i=1}^{\ell - 1})}{Z(\Ga_\ell)}$. When $\Ga = \Ga_i$ and $\theta = \theta_i$ for all $i$, we write the central product as $(\Ga \circ_\theta)_{i=1}^{\ell}$.
\end{defn}
Suppose that $\Ga \circ_{\theta} \De$ is a central product of two nilpotent groups. Since products and quotients of nilpotent groups are nilpotent, it follows that $\Ga \circ_\theta \De$ is a nilpotent group. However, the isomorphism type of $\Ga \circ_\theta \De$ is dependent on $\theta$.

\begin{prop}\label{central_product_growth}
Let $\{\Ga_i\}_{i=1}^{\ell}$ be a collection of torsion free admissible groups where $h(Z(\Ga_i)) = 1$ for all $i$. Let $Z(\Ga_i) = \innp{z_i}$, and let $\map{\theta_i}{Z(\Ga_i)}{Z(\Ga_{i+1})}$ be the isomorphism given by $\theta(z_i) = z_{i+1}$ for $1 \leq i \leq \ell -1$. Then $h((\Ga_i \circ_{\theta_i})_{i=1}^\ell) = \sum_{i=1}^{\ell }h(\Ga_i) - \ell + 1$ and $h(Z(\Ga_i \circ_{\theta_{i}})_{i-1}^\ell)) = 1$
\end{prop}
\begin{proof}
We may assume that $\ell =2$.  First note that if $\Ga$ is a torsion free admissible group with a normal subgroup $\De \nsub \Ga$ such that $\Ga / \De$ is torsion free, then $h(\Ga) = h(\De) + h(\Ga / \De)$. Observe that $\Ga_1 \circ_{\theta} \Ga_2 / Z(\Ga_1  \circ_{\theta} \Ga_2) \cong \Ga_1 / Z(\Ga_1) \times \Ga_2 / Z(\Ga_2)$. Since $h(Z(\Ga_1 \circ_{\theta} \Ga_2)) = 1$, we may write $h(\Ga_1 / Z(\Ga_1)) + h(\Ga_2 / Z(\Ga_2)) + 1 = h(\Ga_1 \circ_{\theta} \Ga_2)$. Thus, $h(\Ga_1 \circ_{\theta} \Ga_2) = h(\Ga_1) - 1 + h(\Ga_2) - 1 + 1 = h(\Ga_1) + h(\Ga_2) - 1$.
\end{proof}

\begin{defn}\label{filiform}
For $\ell \geq 3$, we define $\La_\ell$ to be the torsion free admissible group generated by the set $S_\ell =\{ x_i \}_{i=1}^{\ell}$ with relations consisting of commutator brackets of the form $[x_1,x_i] 
= x_{i+1}$ for $2 \leq i \leq \ell - 1$ and all other commutators being trivial.
\end{defn}

$\La_\ell$ is an example of a Filiform nilpotent group. It has Hirsch length $\ell$ and has step length $\ell-1$. Defining $\De_i = \innp{x_s}_{s=m-i+1}^\ell$, it follows that $\{ \De_i \}_{i=1}^\ell$ is a cyclic series for $\La_{\ell}$ and $\{\xi_i \}_{i=1}^{\ell}$ is a compatible generating subset where $\xi_{i} = x_{\ell - i +1}$. Additionally, $h(Z(\La_\ell)) = 1$.

\begin{proof}\textit{Theorem \ref{applications 2}(ii), (iii) and (iv)}\\
Assume that $\ell \geq 3$. By construction, $\La_\ell$ is a torsion free admissible group of Hirsch length $\ell$ such that $h(Z(\Ga_{\ell})) = 1$ . Corollary \ref{cong_rf} implies that $\Farb_{\La_m}(n) \approx (\log(n))^\ell$ which gives Theorem \ref{applications 2}(ii).

For $2 \leq c_1 < c_2$ and $\ell \geq 1$, there exist natural numbers $j_{\ell}$ and $\iota_{\ell}$ satisfying $(j_{\ell} -1 ) \: (c_1 + 1)  = \ell \: \lcm(c_1 + 1,c_2 + 1)$ and $(\iota_{\ell} -1 ) \: (c_2 + 1)  = \ell \: \lcm(c_1 + 1,c_2 + 1)$, respectively. Let $\Ga_{\ell} = (\La_i \circ_{\theta_\Ga})_{i=1}^{j_\ell}$ and $\De_{\ell} = (\La_i \circ_{\theta_\De})_{i=1}^{\iota_\ell}$ where $\map{\theta_\Ga}{Z(\La_{c_1+1})}{Z(\La_{c_1+1})}$ and $\map{\theta_\De}{Z(\La_{c_2+1})}{Z(\La_{c_2+1})}$ are the identity isomorphisms, respectively. Proposition \ref{central_product_growth} implies that $h(\Ga_\ell) = h(\De_\ell)$ and $h(Z(\Ga_{\ell})) = h(Z(\De_{\ell})) = 1$, and thus, Corollary \ref{cong_rf} implies that $\Farb_{\Ga_{\ell}}(n) \approx \Farb_{\De_{\ell}}(n)$.

Lastly, let $c > 1$ and $m \geq 1$, and consider the group $\Ga_{c \:m} = (\La_{c+1} \circ_{\theta})_{i=1}^{c \:m}$ with finite generating subset $S_{cm}$. Proposition \ref{central_product_growth} implies that $h(\Ga_{cm}) = c \: m^2 + c \:m - 1$, and since $c \: m^2 + c \: m - 1 \geq m$, Corollary \ref{cong_rf} implies that $(\log (n))^m \preceq \Farb_{\Ga_{c \: m}}(n)$ as desired.
\end{proof}

\section{A review of Blackburn and a proof of Theorem \ref{precise_heisenberg_calc}}\label{Heisenberg_conjugacy_section}
We start with a review of Blackburn's proof of conjugacy separability for infinite admissible groups. This section provides motivation for estimates in the following sections and how one obtains an upper bound for $\Conj_{\Gamma}(n)$.

Let $\set{\Ga,\De_i,\xi_i}$ be an admissible $3$-tuple such that $\Ga$ is torsion free, and let $\ga, \eta \in \Ga$ such that $\ga \nsim \eta$.  In order to construct a surjective homomorphism to a finite group that separates the conjugacy classes of $\ga$ and $\eta$, we proceed by induction on $h (\Ga )$. Since the base case is evident, we may assume that $h (\Gamma) > 1$. When $\pi_{\De_1}(\ga) \nsim \pi_{\De_1}(\eta)$, induction implies there exists a surjective homomorphism to a finite group $\map{\varphi}{\Ga}{Q}$ such that $\varphi(\ga) \nsim \varphi(\eta)$. Otherwise, we may assume that $\eta = \gamma \: \xi_1^{t}$ for some $t \in \Z - \{0\}$. The following integer is of particular importance.

\begin{defn} \label{important_map}
Let $\Gamma$ be a torsion free admissible group with a cyclic series $\set{\Delta_i}_{i=1}^{h (\Gamma)}$ and a compatible generating subset $\set{\xi_i}_{i=1}^{h(\Ga)}$. Let $\gamma \in \Gamma$.  If we let $\map{\varphi}{\pi_{\Delta_1}^{-1} (C_{\Gamma / \Delta_1}(\bar{\gamma}))}{\Delta_1}$ be given by $\varphi (\eta) = [  \gamma,\eta  ]$, we then define $\tau(\Ga, \De_i, \xi_i, \ga) = \tau(\ga) \in \N$ such that $\innp{ \xi_1^{\tau(\ga)}} \cong \text{Im}  (\varphi)$.
\end{defn}

We choose a prime power $p^\al$ such that $p^\al \mid \tau(\Ga, \De_i, \xi_i, \ga)$ and $p^\al \nmid t$. We then find a $w \in \N$ such that if $\alpha \geq w$, then for each $\gamma \in \Ga^{p^\alpha}$ there exists $\eta \in \Gamma$ satisfying $\eta^{p^{\alpha - w}} = \gamma$ (see \cite[Lem 2]{Blackburn}).  Consider the following definition (see \cite[Lem 3]{Blackburn}).

\begin{defn}\label{e_value}
Let $\set{\Ga, \De_i,\xi_i}$ be an admissible $3$-tuple such that $\Ga$ is torsion free, and let $\gamma \in \Gamma$. We define $e (\Gamma, \De_i,\xi_i,\gamma) = e(\ga) \in \N$ to be the smallest natural number such that if $\alpha \geq e(\ga)$, then $C_{\Ga / \Ga^{p^\alpha}} (\bar{\ga}) \subseteq \sigma_{p^\al}(C_\Ga (\ga) \cdot \Ga^{p^{\al - e(\ga)}}).$
\end{defn}

Letting $e'(\bar{\ga}) = e(\Gamma / \Delta_1,\De_i/\De_1, \bar{\xi_i}, \bar{\gamma})$, we set $\omega = \al + w + e'(\bar{\gamma})$. Blackburn then proves  that $\sigma_{p^\be}(\gamma) \nsim \sigma_{p^\be}(\eta)$ (see \S \ref{proof of upper} and \cite{Blackburn}). However, as a consequence of the choice of a cyclic series and a compatible generating subset, it becomes evident that the integer $w$ is unnecessary. When $\Ga$ has finite order elements, Blackburn inducts on $|T (\Ga)|$. Thus, it suffices to bound $p^{e(\Ga,\De_i,\xi_i,\ga)}$ and $\tau(\Ga,\De_i,\xi_i,\ga)$ in terms of $\|\ga\|_S$ and $\|\eta\|_S$. Following Blackburn's method, we calculate the asymptotic upper bound for $\Conj_{\Heis_{2m+1}(\Z)}(n)$. We then demonstrate that the  upper bound is sharp. 

Before starting, we make the following observations for $\Heis_{2m+1}(\Z)$. Using the cyclic series and compatible generating subset given in the second paragraph of Subsection \ref{Heisen_cyclic_series}, we have $\tau(\ga) = \tau(\Heis_{2m+1}(\Z),\De_i,\xi_i,\ga) = \GCD \{x_{\gamma,i}, y_{\gamma,j} | 1 \leq i,j\leq m \}.$ Thus, $\tau(\ga) \leq C_0 \|\ga\|_S$ for some $C_0 \in \N$. Moreover, via Subsection \ref{integral_heisen_basics} we may write the conjugacy class of $\ga$ as
\begin{equation}\label{heisenberg_conj} 
 \set{ \left. \begin{pmatrix} 1 & \vec{x}_{\gamma} & \tau(\ga) \: \beta + z_{\gamma} \\ \vec{0}  & \textbf{I}_m & \vec{y}_{\gamma} \\ 0 & \vec{0} & 1  \end{pmatrix} \right|  \beta \in \mathbb{Z}}.
\end{equation}
\begin{prop} \label{heisen_upper_1}
$\Conj_{\Heis_{2m+1} (\mathbb{Z})} (n) \preceq n^{2m + 1}$.
\end{prop}
\begin{proof}
Let $\ga, \eta \in \Ga$ such that $\|\ga\|_S,\|\eta\|_S \leq n$ and $\ga \nsim \eta$. We need to construct a surjective homomorphism $\map{\varphi}{\Heis_{\text{2m+1}}(\Z)}{Q}$ to a finite group such that $\varphi (\ga) \nsim \varphi(\eta)$ and $|Q| \leq C \: n^{2m+1}$ for some $C \in \N$. We proceed based on whether $\ga$ and $\eta$ have equal images in $(\Heis_{2m + 1}(\Z))_{\text{ab}}$.  Corollary \ref{abelian_rf} (see also \cite[Cor 2.3]{BouRabee10}) implies that there exists a surjective homomorphism $\map{\varphi}{\Z^{2m}}{Q}$ such that $\varphi(\pi_{\text{ab}}(\ga \: \eta^{-1})) \neq 1$ and $|Q| \leq C_1 \: \log (C_1 \: n)$ for some $C_1 \in \N$.  Since $\ga$ and $\eta$ are non-equal central elements in $Q$, it follows that $\varphi (\pi_{\text{ab}} (\ga)) \nsim \varphi( \pi_{\text{ab}} (\eta))$, and thus, $\CDepth_{\Heis_{2m+1}(\Z)}(\ga,\eta) \leq C_1 \log (C_1 \: n)$.

Thus, we may assume that $\pi_{\text{ab}}(\ga) = \pi_{\text{ab}}(\eta)$. In particular, we may write $\eta = \ga \: \la^t$ where $|t| \leq C_0 \: n^2$. Let $p^\omega$ be a prime power that divides $\tau(\ga)$ but not $t$. We claim that $\sigma_{p^\omega}(\ga) \nsim \sigma_{p^\omega}(\ga \: \la^t)$, and for a contradiction, suppose otherwise. That implies there exists $x \in \Heis_{2m+1}(\Z)$ such that $\sigma_{p^\omega}([\ga, x]) = \sigma_{p^\omega}(\la^t)$. Equation (\ref{heisenberg_conj}) implies that $\modulo{z_\eta \in \set{\ell_\ga \: \beta + z_\ga : \beta \in \Z}}{p^\omega}.$ Therefore, there exist $a,b \in \Z$ such that $t = a \: \tau(\ga) + b \: p^\omega$. Thus, $p^\omega \mid t$, a contradiction. Hence, $\sigma_{p^\omega}(\ga) \nsim \sigma_{p^\omega}(\eta)$.

When $\tau(\ga) \neq 0$, we have that $p^\omega \leq \tau(\ga) \leq C_0 \: n$. Hence, $\CDepth_{\Heis_{2m+1}(\Z)}(\ga,\eta) \leq C_0^{2m+1}\: n^{2m+1}$. When $\tau(\ga) = 0$, the Prime Number Theorem \cite[1.2]{tenenbaum} implies that there exists a prime $p$ such that $p \nmid t$ where $p \leq C_2 \: \log(C_2 \: |t|)$ for some $C_2 \in \N$. Hence, $p  \leq  C_3 \: \log(C_3 \:  n)$ for some $C_3 \in \N$, and thus, $\CDepth_{\Heis_{2m+1}(\Z)}(\ga,\eta) \leq C_3 \: (\log (C_3 \: n))^{2m+1}$. Hence, $\Conj_{\Heis_{2m+1}(\Z)}(n) \preceq n^{2m+1}$.
\end{proof}

The following proposition finishes the proof of Theorem \ref{precise_heisenberg_calc}.
\begin{prop}\label{heisen_lower}
$n^{2m+1} \preceq  \Conj_{\Heis_{2m+1} (\mathbb{Z})} (n)$.
\end{prop}
\begin{proof}
We will construct a sequence of non-conjugate pairs $\{\ga_i,\eta_i\}$ such that $\CDepth_{\Heis_{2m + 1}(\Z)}(\ga_i,\eta_i) = n_i^{2m+1}$ where $\| \ga_i \|, \| \eta_i \| \approx n_i$. Let $\set{p_i}$ be an enumeration of the primes. Writing $p_i \cdot e_1$ as the scalar product, we consider the following pair of elements:
\begin{equation*}
\gamma_i = \begin{pmatrix}
1 & p_i \cdot \vec{e}_1 & 1 \\
\vec{0} & \mathbf{I}_m & \vec{0} \\
0 & \vec{0} & 1
\end{pmatrix}
\quad
\text{ and }
\quad
\eta_i = \begin{pmatrix}
1 & p_i \cdot \vec{e}_1 & 2 \\
\vec{0} & \mathbf{I}_m & \vec{0} \\
0 & \vec{0} & 1
\end{pmatrix}.
\end{equation*} Equation (\ref{heisenberg_conj}) implies that we may write the conjugacy class of $\ga_i$ as
\begin{equation}\label{heisen_conjugacy_class}
\set{\left. \begin{pmatrix} 1 & p_i \cdot \vec{e}_1 & tp_i + 1 \\ \vec{0} &  \mathbf{I}_m & \vec{0} \\ 0 & \vec{0} & 1 \end{pmatrix} \right| t \in \Z}.
\end{equation} Since $\sigma_{p_i}(\gamma_i)$ and $\sigma_{p_i}(\eta_i)$ are non-equal central elements of $\Heis_{2m+1} (\mathbb{Z} ) / (\Heis_{2m+1} (\mathbb{Z}))^{p_i}$, it follows that $\gamma_{i} \nsim \eta_{i}$ for all $i$. Moreover, we have $\|\ga_i\|_S, \|\eta_i\|_S \approx p_i$. Given that $|\Heis_{2m+1}(\Z) / (\Heis_{2m+1}(\Z))^{p_i}| = p_i^{2m+1}$, we claim $\CDepth_{\Heis_{2m + 1}(\Z)} ( \ga_i,\eta_i) = p_i^{2m+ 1}$. In order to demonstrate our claim, we show for all surjective homomorphisms to a finite group $\map{\varphi}{\Heis_{2m+1}(\Z)}{Q}$ where $|Q| < p_i^{2m+1}$ that $\varphi(\ga_i) \sim \varphi(\eta_i)$. \cite[Thm 2.7]{Hall_notes} implies that we may assume that $|Q| = q^\mu$. Since $\varphi(\ga_i) = \varphi(\eta_i)$ when $\varphi(\la) = 1$, we may assume that $\varphi(\la) \neq 1$.

Suppose first that $q = p_i$. We demonstrate that if $Q$ is a group where $\varphi(\ga_i) \nsim \varphi(\eta_i)$, then there exists no proper quotient of $Q$ such that the images of $\varphi(\ga_i)$ and $\varphi(\eta_i)$ are non-conjugate. Since $\Bas(\Heis_{2m + 1}(\Z)) = 1$, Proposition \ref{p_group_lower_bounds} implies that $|Q| = p_i^{2m+1}$. Since every choice of an admissible quotient with respect to any primitive, central, non-trivial element is isomorphic to the trivial subgroup, Proposition \ref{alternative} implies that there exist no proper quotients of $Q$ such that the image of $\varphi(\la^2)$ is non-trivial Thus, if $N$ is a proper quotient of $Q$ with natural projection $\map{\rho}{Q}{N}$, then $\ker(\rho) \cap Z(Q) \cong Z(Q)$ since $Z(Q) \cong \Z / p_j \Z$ by Proposition \ref{alternative}. Thus, $\rho(\varphi(\ga_i)) = \rho(\varphi(\eta_i))$; hence, $\rho(\varphi(\ga_i)) \sim \rho(\varphi(\eta_i))$. In particular, if $Q$ is a $p$-group where $|Q| < p_i^{2m+1}$, then $\varphi(\ga_i) \sim \varphi(\eta_i)$.  Thus, we may assume that $q \neq p_i$.

If $q > p_i$, then Proposition \ref{p_group_lower_bounds} implies that $p^{2m + 1} > q^\mu$. Thus, we may assume that $q < p_i$. Since Proposition \ref{p_group_lower_bounds} implies that $\Z / q^\nu \Z \cong Z(Q)$, Equation \ref{heisen_conjugacy_class} implies that if $1 \equiv p \: t \: ( \text{ mod }  q^\nu \: \mathbb{Z} )$ for some $t \in \mathbb{Z}$ , then $\varphi (\gamma_p) \sim \varphi(\eta_p)$.  The smallest $q^\nu$ where this fails is $q^\nu = p_i$ since $\bar{p_i}$ is a unit in $\mathbb{Z} / q^\nu \mathbb{Z}$ if and only if $\GCD(p_i,q^\nu) = 1$. Therefore, $\varphi(\ga_i) \sim \varphi(\eta_i)$ when $q^\mu < p_i$. Hence, $n^{2m+1} \preceq \Conj_{\Heis_{2m+1}(\Z)}(n)$.
\end{proof}

The following corollary will be useful for the proof of Theorem \ref{lower}.
\begin{cor}\label{difficult_lower}
Let $\Heis_3 (\Z)$ be the $3$-dimensional Heisenberg group with the presentation given by $\innp{\kappa, \mu,\nu: [\mu,\nu] = \kappa, \kappa \text{ central }}$, and let $p$ be a prime. Suppose $\map{\varphi}{\Heis_3(\Z)}{Q}$ is a surjective homomorphism such that $Q$ is a $q$-group where $q$ is a distinct from $p$ and where $\varphi(\kappa) \neq 1$. Then $\varphi(\mu^p \: \kappa) \sim \varphi(\mu^p \: \kappa^2)$.
\end{cor}
\begin{proof}
	We may write the conjugacy class of $\mu^p \: \kappa$ as $\{ \mu^p \: \kappa^{t \: p + 1} \: | \: t \in \Z \}$. Proposition \ref{p_group_lower_bounds} implies that $Z(Q) \cong \innp{\varphi(\kappa)}$. Hence, $Z(Q) \cong \Z / m \Z$ where $m = \text{Ord}_Q(\varphi(\kappa))$. Since $Q$ is a $q$-group, it follows that $m= q^\beta$. Given that $\GCD(p,q^\beta) = 1$, there exists integers $r,s$ such that $r \: p + s \: q^\beta = 1$. We have that $\mu^p \: \kappa^{r \: p + 1} \sim \mu^p \: \kappa$. We may write
	$
	\varphi(\mu^p \: \kappa^{r \: p + 1}) = \varphi(\mu^p \: \kappa^{1 - s \: q^\beta + 1}) = \varphi(\mu^p \: \kappa^{2}).
	$
	Therefore, $\varphi(\mu^p \: \kappa) \sim \varphi(\mu^p \: \kappa^2)$ as desired.
\end{proof}

\section{Relating complexity in groups and Lie algebras}
Let $\Ga$ be a torsion free admissible group with finite generating subset $S$, Mal'tsev completion $G$, and Lie algebra $\Fr{g}$ of $G$. The overall goal of this section is to provide a bound of $\|\Log(\ga)\|_{\Log (S)}$ in terms of $\|\ga\|_S$ where $\Log(S)$ gives a norm for the additive structure of $\Fr{g}$.

\begin{prop}\label{log_distortion}
Let $\set{\Gamma, \Delta_i, \xi_i, G, \Fr{g}, \nu_i}$ be an admissible 6-tuple, and suppose that $\Ga$ has step length $c$. Let $\gamma \in \Gamma$. Then there exists a $C \in \N$ such that $\|\Log (\gamma)\|_X \leq C \: (\|\gamma\|_{S})^{c^2}$.
\end{prop}
\begin{proof}
Using the Mal'tsev coordinates of $\ga$, we may write $\ga = \prod_{i=1}^{h(\Ga)} \xi_i^{\al_i}$.  Lemma \ref{coord_bound} implies that there exists a $C_1 \in \N$ such that $|\al_i| \leq C_1 (\|\ga\|_S)^{c}$ for all $i$. A straightforward application of the Baker-Campbell-Hausdorff formula (\ref{Baker}) implies that $\Log (\xi_i^{\al_i}) = \al_i \: \nu_i$. Writing $A_i = \al_i \: \nu_i$, it follows that $\| A_i \|_X \leq C_1 (\| \ga \|_S)^{c }$. Equation (\ref{Baker}) implies that we may write
$$
\| \Log (\ga) \|_X \leq \sum_{i=1}^{c} \|CB_i (A_1,\cdots, A_{h (\Ga)})\|_X
$$
where $CB_{i} (A_1, \cdots, A_{h (\Ga)})$ is a rational linear combination of $i$-fold Lie brackets in $\{ A_{j_s} \}_{s=1}^t \subseteq \set{A_i}_{i=1}^{h (\Ga)}$. Let $\{A_{j_s}\}_{s=1}^t \subset \set{A_i}_{i=1}^{h (\Ga)}$ where $[A_{j_1},\cdots A_{j_t}] \neq 0$. Via induction on the length of the iterated Lie bracket, one can see that there exists $C_t \in \N$ such that $[A_{j_1},\cdots, A_{j_t}] \leq C_t \: \prod_{s=1}^{t} \|A_{j_s} \|_X \leq C_t \: C_1 \: (\| \ga \|_S)^{t \: c}.$ By maximizing over all possible $t$-fold Lie brackets of elements of $\{A_i\}_{i=1}^{h(\Ga)}$, there exists a $D_i \in \N$ such that $\| CB_i (A_1,\cdots,A_{h(\Gamma)}) \|_X \leq D_i \: (\|\ga\|_S)^{t \: c}$. Hence, $\| \Log (\gamma) \|_X \leq C \: (\|\ga\|_S)^{c^{2}}$ for some $C \in \N$.
\end{proof}
An immediate application of Proposition \ref{log_distortion} is that the adjoint representation of $\Ga$ has matrix coefficients bounded by a polynomial in terms of word length.
\begin{prop}\label{claim_1}
Let $\set{\Gamma, \La, \Delta_i, \xi_i, G, \Fr{g},\nu_i}$ be an admissible $7$-tuple. Let $\gamma \in \Gamma$, and let $(\mu_{i,j})$ be the matrix representative of $\Ad(\ga)$ with respect to $X$. Then $|\mu_{i,j}| \leq C \: (\|\gamma\|_S)^{c}$ for some $C \in \N$ where $c$ is the step length of $\Ga$.
\end{prop}
\begin{proof}
Proposition \ref{log_distortion} implies there exists a $C_1 \in \N$ such that $\|\Log (\ga)\|_{X} \leq C_1 \: (\|\gamma\|_S)^{c^2}$. Via induction on the length of the Lie bracket and Equation (\ref{baker_adjoint}), we have  $\|\Ad (\ga) (v_i)\|_{X} \leq C_2  (\| \ga \|_S)^{c^3}$ for some $C_2 \in \N$.
\end{proof}
\section{Preliminary estimates for Theorem \ref{upper}}
Let $\set{\Ga,\De_i,\xi_i}$ be an admissible $3$-tuple such that $\Ga$ is torsion free. Let $\ga$ be a non-trivial element of $\Ga$, and let $p $ be some prime. In the following section, we demonstrate the construction of the integer $e(\ga) = e(\Ga,\De_i,\xi_i,\ga)$ and give an asymptotic bound for $p^{e(\ga)}$ in terms of  $\|\ga\|_S$ independent of the prime $p$. We first provide a bound for $\tau(\Ga,\De_i,\xi_i,\ga)$ in terms of $\|\ga\|_S$.

\begin{prop}\label{important_estimate}
Let $\set{\Gamma, \La, \Delta_i, \xi_i, G, \Fr{g}, \nu_i}$ be an admissible $7$-tuple, and let $\gamma \in \Gamma$.  There exists $k,C \in \N$ such that $|\tau(\Ga,\De_i,\xi_i,\ga)| \leq C \: (\|\gamma\|_{S})^k$. 
\end{prop}
\begin{proof} Before we start, we make some simplifying notation by letting $\tau(\ga) = \tau\pr{\Ga,\De_i,\xi_i,\ga}$. Consider the smooth map $\map{\Phi}{G}{G}$ given by $\Phi(g) = [\gamma,g]$. Suppose $\eta \in \Gamma$ satisfies $\Phi(\eta) = \xi_1^{\tau(\Ga,\ga)}$. The commutative diagram $(1.2)$ on \cite[Pg 7]{Dekimpe} implies that we may write 
$(I - \Ad (\gamma^{-1})) (\Log (\eta)) = \Log (\xi_1^{\tau(\Ga,\ga)})$ where $(d\Phi_\ga)_1 = I - \Ad (\ga^{-1})$. Proposition \ref{claim_1} implies that $I - \Ad (\gamma^{-1})$ is a strictly upper triangular matrix whose coefficients are bounded by $C \: (\| \ga \|_S)^{(c(\Ga))^3}$ for some $C \in \N$. Since it is evident that $\Log (\xi_1^{\tau(\ga)}) = \tau(\ga) \nu_1$, backwards substitution gives our result. \end{proof} 

The first statement of the following proposition is originally found in \cite[Lem 3]{Blackburn}. We reproduce its proof so that we may provide estimates for the value $e(\Ga,\De_i,\xi_i,\ga)$ in terms of $\|\ga\|_S$.
\begin{prop}\label{centralizer_prop}
Let $\set{\Gamma, \La, \Delta_i, \xi_i, G, \Fr{g},\nu_i}$ be an admissible $7$-tuple. Let $p$ be prime and $\gamma \in \Gamma$. Then there exists $e(\Gamma,\De_i,\xi_i,\gamma) = e(\ga) \in \N$ such that if $\alpha \geq e(\ga)$, then $C_{\Ga / \Ga^{ p^\alpha 
}}(\bar{\ga}) \subseteq \sigma_{p^\al}(C_\Ga(\ga)\cdot \Ga^{p^{\al - e(\ga)}}).$ Moreover, $p^{e(\gamma)} \leq  C  (\|\gamma\|_{S})^{k}$ for some $C \in \N$ and $k \in \N$.
\end{prop}
\begin{proof} We proceed by induction on Hirsch length, and given that the statement is clear for $\Z$ by setting $e(\ga) = 0$ for all $\ga$, we may assume that $h (\Ga) > 1$.
	
We construct $e(\ga)$ based on the value of $\tau(\ga) = \tau(\Ga,\De_i,\xi_i,\ga)$ (see Definition \ref{important_map}). By induction, we may assume that we have already constructed $e'(\bar{\ga}) = e(\Ga / \De_1, \De_i/\De_1,\bar{\xi_i}, \bar{\ga})$. When $\tau(\gamma) = 0$, we set $e(\ga) = e'(\bar{\gamma}).$ Suppose $\al \geq e(\ga)$ and that  $\bar{\eta} \in C_{\Ga / \Ga^{p^\al}}(\bar{\ga})$ for some $\eta \in \Ga$. By selection, $\bar{\eta} \in  C_{\Ga / \Ga^{p^\al} \cdot \De_1}(\bar{\ga})$. Thus, we may write $\eta \in \pi_{\De_1}^{-1}(C_{\Ga / \De_1}(\bar{\ga})) \cdot \Ga^{p^{\al - e(\ga)}}$. Since $\pi_{\Delta_1}^{-1} (C_{\Gamma / \Delta_1}  (\bar{\gamma})) = C_{\Gamma} (\ga)$, it follows that $\bar{\eta} \in \sigma_{p^\al}(C_\Ga (\ga)\cdot \Ga^{p^{\al - e(\ga)}})$. Thus, $C_{\Ga / \Ga^{ p^\alpha 
}}(\bar{\ga}) \subseteq \sigma_{p^\al}(C_\Ga(\ga)\cdot \Ga^{p^{\al - e(\ga)}}).$

When $\tau(\ga) \neq 0$, we let $\beta$ be the largest power of $p$ such that $p^\beta \mid \tau(\gamma)$, and set $e(\Ga) =  e'(\bar{\gamma}) + \beta.$ Let $\al \geq e(\ga)$, and let $\eta \in \Ga$ satisfy $\bar{\eta} \in C_{\Gamma /  \Ga^{p^\alpha}} ( \bar{\gamma})$. Thus, $\bar{\eta} \in C_{\Gamma / \Gamma^{p^{\alpha}} \cdot \Delta_1}  ( \bar{\gamma})$, and subsequently, induction implies $\bar{\eta} \in \pi_{\Gamma^{p^{\alpha}} \cdot \Delta_1}( C_{\Gamma / \Delta_1} (\bar{\gamma}) \cdot \Gamma^{p^{\alpha - e(\ga) + \beta}}).
$
Thus, we may write $\eta = \mu \:\epsilon^a \: \la$ where $ \mu \in C_{\Gamma} (\gamma)$, $\la \in \Ga^{p^{\alpha - e(\ga) + \beta}}$, and $\varphi_\ga(\epsilon) = \xi_1^{\tau(\ga)}$. Hence, we have $[\ga,\eta] = [\ga, \eps^a] \in \Ga^{p^{\al - e(\ga) + \beta}}$. Since  $[\ga, \epsilon^a ] \in \Ga^{p^{\alpha - e'(\bar{\gamma})}}$ and $[\ga, \epsilon^a] \in \De_1$, we have that  $p^{\al - e(\ga) + \beta} \mid a \: \tau(\ga)$. By definition of $p^\beta$, it follows that $p^{\al - e(\ga)} \mid a$, and thus, $\bar{\eta} \in \sigma_{p^\al}(C_\Ga(\ga) \cdot \Ga^{p^{\alpha - e(\ga)}})$. Hence, $C_{\Ga / \Ga^{ p^\alpha 
}}(\bar{\ga}) \subseteq \sigma_{p^\al}(C_\Ga(\ga)\cdot \Ga^{p^{\al - e(\ga)}})$.

We proceed by induction on Hirsch length to demonstrate the asymptotic upper bound, and since the base case is clear, we assume that $h (\Ga) > 1$. Let $\gamma \in \Gamma$, and suppose that $\tau(\gamma) = 0$. By construction, $e(\gamma) = e'(\bar{\gamma})$, and thus, induction implies that there exist  $C_1,k_1 \in \N$ such that $p^{e'(\bar{\gamma})} \leq C_1 \: (\|\bar{\gamma}\|_{\bar{S}})^{k_1}$. When $\tau(\ga) \neq 0$, it follows that $e(\ga) = e'(\bar{\ga}) + \beta$ where $\beta$ is the largest power of $p$ that divides $\tau(\gamma)$.
Proposition \ref{important_estimate} implies there exist $k_2,C_2 \in \N$ such that $p^{\beta} \leq C_2 \: (\|\gamma\|_S)^{k_2}$. Consequently, $
p^{e(\gamma)} \leq C_1 \: C_2 \: (\|\gamma\|_S)^{k_1 + k_2}.$
\end{proof}

\section{Proof of Theorem \ref{upper}}\label{proof of upper} Let $\Ga$ be an infinite admissible group. In order to demonstrate that there exists $k_1 \in \N$ such that $\Conj_{\Ga}(n) \preceq n^{k_1}$, we need to show for any $\ga,\eta \in \Ga$ where $\ga \nsim \eta$ and $\|\ga\|_S, \|\eta\|_S \leq n$ there exists a prime power $p^{\omega} \leq C \: n^{k_2}$ such that $\sigma_{p^\omega}(\ga) \nsim \sigma_{p^\omega}(\eta)$ for some $C,k_2 \in \N$. It then follows that $\CDepth_{\Ga}(\ga,\eta) \leq C^{h(\Ga)} n^{h(\Ga) \: k_2}$. We first specialize to torsion free admissible groups.

\begin{prop}\label{torsion_free_case}
Let $\set{\Ga, \La, \De_i,\xi_i,\Fr{g},\nu_i}$ be an admissible $7$-tuple. Then there exists $k \in \N$ such that $\Conj_{\Ga}(n) \preceq n^{k}$.
\end{prop}
\begin{proof}
Let $\ga, \eta \in \Ga$ such that $\|\ga\|_S,\|\eta\|_S \leq n$ and where $\ga \nsim \eta$. We demonstrate that there exists a $k_0 \in \N$ such that $\CDepth_{\Ga}(\ga,\eta) \leq C_0 \: n^{k_0}$ for some $C_0 \in \N$ by induction on $h(\Ga)$, and since the base case is clear, we may assume that $h (\Gamma) >1$. If $\pi_{\De_1}(\ga) \nsim \pi_{\De_1}(\eta)$, then the inductive hypothesis implies that there exists a surjective homomorphism to a finite group $\map{\varphi}{\Ga / \De_1}{Q}$ such that $\varphi(\ga) \nsim \varphi(\eta)$ and where $|Q| \leq C_1 \: n^{k_1}$ for some $C_1, k_1 \in \N$. Thus,  $\CDepth_\Ga ( \ga,\eta) \leq C_1 \: n^{k_1}$. Otherwise, we may assume that $\eta = \gamma \:  \xi_1^{t}$, and Lemma \ref{coord_bound} implies that $|t| \leq C_2 \: n^{c(\Ga)}$ for a constant $C_2 \in \N$.

For notational simplicity, let $\tau(\ga) = \tau(\Ga,\De_i,\xi_i, \ga)$ and $e'(\ga) = e(\Ga,\De_i / \De_1,\bar{\xi_i},\bar{\ga})$. Since $\ga \nsim \ga \: \xi_1^t$, there exists a prime power $p^{\alpha}$  such that $ p^\al \mid \tau(\gamma)$ but $p^\al \nmid t$. We set $\omega = \al + e'(\bar{\ga})$, and suppose for a contradiction there exists $x \in \Ga$ such that $\sigma_{p^\omega}(x^{-1} \: \ga \: x)  = \sigma_{p^\omega} ( \ga \: \xi_1)^t$. That implies $\bar{x} \in C_{\Ga/ \Ga^{p^\omega} \cdot \De_1}(\bar{\ga})$, and thus, $\bar{x} \in \pi_{\Ga^{p^\omega} \cdot \De_1} (C_{\Gamma / \Delta_1}  (\gamma) \cdot \Gamma^{p^{\alpha }})$ by Proposition \ref{centralizer_prop}. Subsequently, $x = g \: \mu$ for some $g \in \pi_{\De_1}^{-1} (C_{\Gamma / \Delta_1}  (\bar{\gamma}))$ and $\mu \in \Ga^{p^\al}$. Hence, $\sigma_{p^\omega}([\ga,g]) = \sigma_{p^\omega}(\xi_1)^{t}$, and since $[\gamma,g] = \xi_1^{q \: \tau (\Ga,\ga)}$ for some $q \in \Z$, it follows that $\xi_1^{t - q \: \tau(\Gamma,\gamma)} \in \Ga^{p^{\al +e (\Gamma / \Delta_1, \bar{\gamma}) }}$. That implies $p^\al \mid t$, which is a contradiction. Hence, $\sigma_{p^\omega}(\gamma) \nsim \sigma_{p^\omega}(\eta)$. 

Proposition \ref{centralizer_prop} implies that $p^{e'(\bar{\ga})} \leq C_3 \: n^{k_2}$ for $C_3,k_2 \in \N$. When $\tau(\ga) = 0$, the Prime Number Theorem \cite[1.2]{tenenbaum} implies that we may choose $p$ such that $|p| \leq C_4 \: \log(C_4 \: n)$ for some $C_4 \in \N$. Hence, $\CDepth_\Ga(\ga,\eta) \leq C_5 \: (\log(C_5 \: n))^{h(\Ga) \: k_2}$ for some $C_5 \in \N$. When $\tau(\ga) \neq 0$, Proposition \ref{important_estimate} implies that $\tau(\ga) \leq C_6 \: n^{k_3}$ for some $C_6, k_3 
\in \N$. Thus, $p^{\omega} \leq C_3 \: C_6 \: n^{k_2 + k_3}$. Therefore, $\CDepth_\Ga (\ga,\eta) \leq (C_3 \: C_6)^{h(\Ga)} \: n^{h(\Ga)(k_2 + k_3)}$. Hence, $\Conj_{\Ga}(n) \preceq n^{k_3}$ where $k_3 = \text{max}\{k_1, h(\Ga)(k_2 + k_3) \}$.
\end{proof}

\begin{proof}\emph{Theorem \ref{upper}}\\
Let $\Ga$ be an infinite admissible group $\Ga$ with a choice of a cyclic series $\{\De_i\}_{i=1}^{m}$ and a compatible generating subset $\{ \xi_i \}_{i=1}^m$. Let $k_1$ be the natural number from Proposition \ref{torsion_free_case} and $k_2$ be the natural number from Proposition \ref{centralizer_prop}, both for $\Ga / T(\Ga)$. Letting $k_3 = h(\Ga) \cdot \text{max}\{k_1,k_2\}$, we claim that $\Conj_{\Ga}(n) \preceq n^{k_3}$. Let $\ga,\eta \in \Ga$ satisfy $\ga \nsim \eta$ and $\|\ga\|_S,\|\eta\|_S \leq n$. In order to show that $\CDepth_{\Ga}(\ga,\eta) \leq C_0 \: n^{k_3}$ where $C_0 \in \N$, we construct a surjective homomorphism to a finite group that distinguishes the conjugacy classes of $\ga$ and $\eta$ via induction on $|T(\Ga)|$. To simplify the following arguments, we let $e(\bar{\ga}) =e(\Ga/T(\Ga),\De_i / T(\Ga),\bar{\xi_i},\bar{\ga})$.

Proposition \ref{torsion_free_case} implies that we may assume that there exists a subgroup $P \subseteq  Z(\Ga)$ of prime order $p$. If $\pi_{P}(\gamma) \nsim \pi_{P}(\eta)$, then induction implies that there exists a surjective homomorphism to a finite group $\map{\varphi}{\Ga / P}{N}$ such that $\varphi(\ga) \nsim \varphi(\eta)$ and where $|N| \leq C_1 \: n^{k_3}$ for some $C_1 \in \N$. Thus, $\CDepth_{\Ga}(\ga,\eta) \leq C_1 \: n^{k_3}$. Otherwise, we may assume that $\eta = \gamma \: \mu$ where $\innp{\mu} = P$.

Suppose there exists $Q \subseteq Z(\Ga)$ such that $|Q| = q$ where $q$ is a prime distinct from $p$. Suppose for a contradiction that there exists $x \in \Gamma$ such that $x^{-1} \: \ga \: x = \ga \:\mu \: \la$ where $Q = \innp{\la}$. Since $[\ga,x] \in Z(\Ga)$ and $\text{Ord}_\Ga(\la) = q$, basic commutator properties imply that $[\ga, x^{q}] =  \mu^{q}$. Given that $p \: s +  q \: r = 1$ for some $r,s \in \Z$, it follows that $[\gamma, x^{ q \: r}] = \gamma \: \mu^{1- p \: s} = \gamma \: \mu$ which is a contradiction. Hence, induction implies there exists a surjective homomorphism to a finite group $\map{\theta}{\Ga / Q}{M}$ such that $\theta(\ga) \nsim \theta(\ga \: \mu)$ and where $|M| \leq C_2 \: n^{k_3}$ for some $C_2 \in \N$. Thus, $\CDepth_{\Ga}(\ga,\eta) \leq C_2 \: n^{k_3}$.

Suppose that $T(\Ga)$ is a $p$-group with exponent $p^m$. We set $\omega =  m + e'(\bar{\ga})$, and suppose for a contradiction there exists $x \in \Ga$ such that $\sigma_{p^\omega}(x^{-1} \: \ga \: x) \sim \sigma_{p^\omega}(\ga \: \mu)$. Thus, $\bar{x} \in  C_{\Ga/ T(\Ga) \cdot \Ga^{p^\omega}}(\bar{\ga})$, and subsequently, Proposition \ref{centralizer_prop} implies that  $\bar{x} \in  \pi_{  T (\Ga) \cdot \Gamma^{p^\omega}} (C_{\Gamma / T(\Ga)} (\bar{\gamma}) \cdot \Gamma^{p^{m}}).$ Therefore, we may write $x= g \: \la$ where $\la \in \Gamma^{p^{m}}$ and $g \in \pi_{T(\Ga)}^{-1} (C_{\Ga / T(\Ga)} (\bar{\gamma}))$. Subsequently, $[\ga,g] \: \mu^{-1} \in \Ga^{p^m}$. Moreover, since $[\ga,g] \in T(\Ga)$ and $T(\Ga) \cap \Ga^{p^m} =1$, it follows that $[\ga,x] = \mu$ which is a contradiction. Proposition \ref{centralizer_prop} imply that $p^{e'(\bar{\ga})} \leq C_3 \: n^{k_2}$ for some $C_3 \in \N$. Thus, $\CDepth_\Ga (\gamma,\eta)  \leq C_3^{h(\Ga)} \: |T (\Ga)| \: n^{h(\Ga) \: k_2},$ and subsequently, $\Conj_{\Ga}(n) \preceq n^{k_3}$.
\end{proof}

\section{Proofs of Theorem \ref{lower} and Theorem \ref{conj_malsev_invariant}}
Let $\Ga$ be an infinite admissible group with a choice of a cyclic series $\set{\De_i}_{i=1}^{m}$ and a compatible generating subset $\set{\xi_i}_{i=1}^{m}$. Since the proofs of Theorem \ref{lower}\textit{(i)} and Theorem \ref{lower}\textit{(ii)} require different strategies, we approach them separately. We start with Theorem \ref{lower}\textit{(i)} since it only requires elementary methods. 

We assume that $\Ga$ contains an infinite, finitely generated abelian group $K$ of index $\ell$. We want to demonstrate that $\log (n) \preceq \Conj_{\Ga}(n) \preceq (\log (n))^{\ell}$. Since $\Farb_{\Ga,S}(n) \preceq \Conj_{\Ga}(n)$, Corollary \ref{abelian_rf} (see also \cite[Cor 2.3]{BouRabee10}) implies that $\log(n) \preceq \Conj_{\Ga}(n)$. Thus, we need only to demonstrate that $\Conj_{\Ga}(n) \preceq (\log (n))^\ell$. For any two non-conjugate elements $\ga,\eta \in \Ga$ such that $\|\ga\|_S,\|\eta\|_S \leq n$ we want to construct a surjective homomorphism $\map{\varphi}{\Ga}{Q}$ such that $\varphi(\ga) \nsim \varphi(\eta)$ and $|Q| \leq C \: (\log(C \: n))^{\ell}$ for some $C \in \N$.

\begin{proof}\textit{Theorem \ref{lower}\textit{(i)}}\\
Let $S_1$ be a finite generating subset for $K$, and let $\set{\upsilon_i}_{i=1}^{\ell}$ be a set of coset representatives of $K$ in $\Gamma$. We take $S = S_1 \cup \set{\upsilon_i}_{i=1}^{\ell}$ as the generating subset for $\Gamma$. If $\|\ga\|_S \leq n$, we may write $\ga = g_\ga \: \upsilon_{\ga}$ where $\|g_\ga\|_{S_1} \leq C_1 \: n$ for some $C_1 \in \N$ and $\upsilon_\ga \in \set{\upsilon_i}_{i=1}^{\ell}$. Conjugation in $\Ga$ induces a map $\map{\varphi}{\Ga / K}{\Aut(K)}$ given by $\varphi(\pi_K(\upsilon_i)) = \varphi_i$. Thus, we may write the conjugacy class of $\ga$ as $\set{\varphi_i(g_\ga) \: (\upsilon_i^{-1} \: \upsilon_\ga \: \upsilon_i)}_{i=1}^{\ell}$. Finally, there exists $C_2 \in \N$ such that if $\|\ga\|_{S_1} \leq n$, then $\|\varphi_i(\ga)\|_S \leq C_2 \: n$ for all $i$.

Suppose $\ga,\eta \in \Ga$ are two non-conjugate elements such that $\|\ga\|_S, \|\eta\|_S \leq n$. If $\pi_{K}(\ga) \nsim \pi_{K}(\eta)$, then by taking the map $\map{\pi_K}{\Ga}{\Ga / K}$, it follows that $\CDepth_\Ga(\ga,\eta) \leq \ell$. Otherwise, we may assume that $\eta = g_\eta \: \upsilon_\ga$. By Corollary \ref{abelian_rf} (see also \cite[Cor 2.3]{BouRabee10}), there exists a surjective homomorphism $\map{f_i}{\Ga}{Q_i}$  such that $f_i(g_\ga^{-1} \upsilon_\ga^{-1} \: \varphi_i(g_\eta) \: (\upsilon_i^{-1} \: \upsilon_\eta \: \upsilon_i)) \neq 1$ and $|Q_i| \leq C_3 \: \log(2 \: C_2 \: C_3 \: n )$ for some $C_3 \in \N$. By letting $H = \cap_{i=1}^{\ell} \ker(f_i)$, it follows that $\pi_H(\ga) \nsim \pi_H(\eta)$ and $|\Ga / H | \leq C_3^{\ell} \: (\log(2 \: C_2 \:  C_3 \: n ))^\ell.$ Therefore, $\Conj_\Ga (n) \preceq (\log (n))^\ell$, and subsequently, $\log (n) \preceq 	\Conj_{\Gamma} (n) \preceq (\log (n))^\ell$. \end{proof}

For Theorem \ref{lower}\textit{(ii)}, suppose that $\Ga$ does not contain an abelian group of finite index. In order to demonstrate that $n^{\psi_{\text{RF}}(\Ga)(c(\Ga / T(\Ga)) - 1)} \preceq \Conj_{\Ga}(n)$, we desire a sequence of non-conjugate pairs $\{\ga_i,\eta_i\}$ such that $\CDepth_{\Ga}(\ga_i,\eta_i) = n_i^{\psi_{\text{RF}}(\Ga)(c(\Ga/T(\Ga)) - 1)}$ where $\|\ga_i\|_S,\|\eta_i\|_S \approx n_i$. In particular, we must find non-conjugate elements whose conjugacy classes are difficult to separate i.e. non-cojugate elements that have relatively short word length in comparison to the order of the minimal finite group required to separate their conjugacy classes. 

We first reduce to the calculation of the lower asymptotic bounds for $\Conj_{\Ga}(n)$ to torsion free admissible groups by appealing to the conjugacy separability of two elements within a finite index subgroup.
\begin{prop}\label{finite_index_conjugacy}
	Let $\Ga$ be an infinite admissible group, and let $\De$ be a subgroup. Suppose there exist $\ga,\eta \in \De$ such that $\ga \nsim \eta$ in $\Ga$. Then $\CDepth_{\De}(\ga,\eta) \leq \CDepth_{\Ga}(\ga,\eta)$.
\end{prop}
\begin{proof}
	We first remark that since $\Ga$ and $\De$ are admissible groups, Theorem \ref{upper} implies $\CDepth_\Ga(\ga,\eta) < \infty$ and $\CDepth_{\De}(\ga,\eta) < \infty$. Suppose that $\map{\varphi}{\Ga}{Q}$ is surjective homomorphism to a finite group such that $|Q| = \CDepth_{\Ga}(\ga,\eta)$. If we let $\map{\iota}{\De}{\Ga}$ be the inclusion, then we have a surjective map $\map{\varphi \circ \iota}{\De}{\varphi(\De)}$ to a finite group where $\varphi(\iota(\ga)) \nsim \varphi(\iota(\eta))$. By definition, $\CDepth_{\De}(\ga,\eta) \leq |\varphi(\De)| \leq |Q|$. Thus, $\CDepth_{\Ga^k}(\ga,\eta) \leq \CDepth_{\Ga}(\ga,\eta)$.\end{proof}

\begin{proof}\textit{Theorem \ref{lower}(ii)}\\
We first assume that $\Ga$ is torsion free. Let $\Ga / \La$ be a choice of a maximal admissible quotient of $\Ga$. There exists a $g \in Z(\Ga) - \set{1}$ such that $\Ga / \La$ is a choice of an admissible quotient with respect to $g$. Moreover, there exists a $k$ such that $g^k = [y,z]$ for some $y \in \Ga_{c -1}$ and $z \in \Ga$. If $g$ is not primitive, then there exists a $x_\La \in Z(\Ga) - \set{1}$ such that $x_\La^s = g$ for some $s \in \Z- \set{0}$. In particular, $x_\La^{t} = [y,z]$ where $t = s \: k$.

There exists $a_\La \in \Ga_{c(\Ga) - 1}$ and $b_\La \in \Ga$ such that $[a_\La,b_\La] = x_\La^{t \: \Bas(\Ga / \La)}$. Equation \ref{presentation_heisen} implies that  $H_\La = \innp{a_\La, b_\La, x_{\La}^{t \: \Bas(\Ga / \La)}} \cong \Heis_3(\Z)$. Let $\set{p_{j,\La}}$ be an enumeration of primes greater than $\Bas(\Ga / \La)$, and let $\ga_{j,\La} = a_\La^{p_{j,\La}} x_{\La}^{t \: \Bas(\Ga / \La)}$ and $\eta_{j,\La} = a_\La^{p_{j,\La}} x_{\La}^{2 \: t \: \Bas(\Ga / \La)}$. Since the images of $\ga_{j,\La}$ and $\eta_{j,\La}$ are non-equal central elements of $\Ga / \La \cdot \Ga^{p_{j,\La}}$, it follows that $\ga_{j,\La} \nsim \eta_{j,\La}$ for all $j$.

We claim that $\ga_{j,\La}$ and $\eta_{j,\La}$ are our desired non-conjugate elements. In particular, we will demonstrate $\CDepth_\Ga (\ga_{j,\La},\eta_{j,\La}) \approx n_i^{\psi_{\text{RF}}(\Ga)(c(\Ga)-1)}$ and that $\|\ga_{j,\La}\|_S, \|\eta_{j,\La}\|_S \approx n_i$. By construction, we have that $\ga_{j,\La},\eta_{j,\La} \in \Ga_{c(\Ga) - 1}$ and $\|\ga_{j,\La}\|_{S'},\|\eta_{j,\La}\|_{S'} \approx p_{j,\La}$ where $S' = S \cap \Ga^2$. \cite[3.B2]{asymptotic_group} implies that $\|\ga_{j,\La}\|_S,\|\eta_{j,\La}\|_S \approx p_{j,\La}^{1 /(c(\Ga) - 1)}$. Therefore, $(\|\ga_{j,\La}\|_S)^{\psi_{\text{Lower}}(\Ga)},(\|\eta_{j,\La}\|_S)^{\psi_{\text{Lower}}(\Ga)} \approx p_{j,\La}^{\psi_{\text{RF}}(\Ga)}$. Hence, we need to demonstrate for all surjective homomorphisms to finite groups $\map{\varphi}{\Ga}{Q}$ where $|Q| < p_{j,\La}^{\psi_{\text{RF}}(\Ga)}$ that $\varphi(\ga_{j,\La}) \sim \varphi(\eta_{j,\La})$.

\cite[Thm 2.7]{Hall_notes} implies that we may assume that $|Q| = q^\beta$ where $q$ is prime. Since $\varphi(\ga_{j,\La}) = \varphi(\eta_{j,\La})$ when $\varphi(x_\La^{t \: \Bas(\Ga / \La)}) = 1$, we may also assume that $\varphi(x_\La^{t \: \Bas(\Ga / \La)}) \neq 1$. Suppose that $q > p_{j,\La}$. Consider the homomorphism given by $\map{\rho \circ \varphi}{\Ga}{Q / \varphi(\La)}$ where $\map{\rho}{Q}{Q / \varphi(\La)}$ is the natural projection. Since $\La \leq \ker(\rho \circ \varphi)$, we have an induced homomorphism $\map{\widetilde{\rho \circ \varphi}}{\Ga / \La}{Q / \varphi(\La)}.$ Since $h(Z(\Ga / \La)) = 1$,  $\widetilde{\rho \circ \varphi}(x^{t \: \Bas(\Ga / \La)}) \neq 1$, and $q > \Bas(\Ga / \La)$,  Proposition \ref{p_group_lower_bounds} implies that $|Q / \varphi(\La)| > p_{j,\La}^{\psi_{\text{RF}}(\Ga)}$. Hence, $|Q| > p_{j,\La}^{\psi_{\text{RF}}(\Ga)}$. Thus, we may assume that $q \leq p_{j,\La}$

Now assume that $q = p_{j,\La}$.  Suppose that $\varphi(\La)$ is a non-trivial subgroup of $Q$. As before, we have an induced homomorphism $\map{\widetilde{\rho \circ \varphi}}{\Ga / \La}{Q / \varphi(\La)}$. Since $|Q/\varphi(Q)| \leq p_{j,\La}^{\psi_{\text{RF}}(\Ga)}$, Proposition \ref{alternative} implies that $|Q / \varphi(Q)| = p_{j,\La}^{\psi_{\text{RF}}(\Ga)}$. Thus, we have that $|Q| > p_j^{\psi_{\text{RF}}(\Ga)}$. Hence, we may assume that $\varphi(\La) = \set{1}$.

Suppose that $q = p_{j,\La}$. If $\varphi(\ga_{j,\La}) \sim \varphi(\eta_{j,\La})$, then there is nothing to prove. Thus, we may assume that $\varphi(\ga_{j,\La}) \nsim \varphi(\eta_{j,\La})$. Proposition \ref{alternative} implies that if $N$ is a proper quotient of $Q$ with natural projection $\map{\theta}{Q}{N}$, then $\ker(\theta) \cap Z(Q) \cong Z(Q)$. Thus, we have that $\theta(\varphi(\ga_{j,\La})) = \theta(\varphi(\eta_{j,\La}))$ since $\theta(\varphi(x_\La)) = 1$. In particular, if $Q$ is a $p_{j,\La}$-group where $\varphi(\La) \cong \set{1}$ and $|Q| < p_{j,\La}^{\psi_{\text{RF}}(\Ga)}$, then $\varphi(\ga_{j,\La}) \sim \varphi(\eta_{j,\La})$. Hence, we may assume that $q \neq p_{j,\La}$.

Since $\varphi(H_\La)$ is a $q$-group where $q \neq p_{j,\La}$, Corollary \ref{difficult_lower} implies that there exists $g \in H_\La$ such that $\varphi(g^{-1} \: \ga_{j,\La} \: g) = \varphi( \eta_{j,\La} )$ as elements of $\varphi(H_\La)$. Thus, $\varphi(\ga_{j,\La}) \sim \varphi(\eta_{j,\La})$. Since we have exhausted all possibilities, it follows that $\CDepth_\Ga (\ga_j,\eta_j) = p_{j,\La}^{\psi_{\text{RF}}(\Ga)}$. Hence, $n^{\psi_{\text{RF}}(\Ga)(c(\Ga) - 1)} \preceq \Conj_{\Ga,S}(n)$.

Now suppose that $\Ga$ is an infinite admissible group where $|T(\Ga)| > 1$. There exists a finite index, torsion free subgroup of $\Ga$ which we denote as $\De$. Let $\De / \La$ be a choice of a maximal admissible quotient of $\De$. Using above reasoning, there exists $x_\La \in \De$ such that $\De / \La$ is a choice of an admissible quotient with respect to $x_\La$ where $x_\La^t = [y,z]$ for some $y \in \De_{c-1}$ and $z \in \De$.

Let $\{p_{j,\La}\}$ be an enumeration of primes greater than $\Bas(\De / \La)$. There exist an $a_\La \in \De_{c(\Ga)-1}$ and $b_\La \in \De$ such that $[a_\La,b_\La] = x_\La^{t \: \Bas(\De / \La)}$. Let $\ga_{j,\La} = a_\La^{p_j} x_\La^{t \: \Bas(\De / \La)}$ and $\eta_{j,\La} = a_\La^{p_j} x_\La^{2 \: t \: \Bas(\De/\La)}$ be the elements from the above construction for $\De$. 
Let $\map{\rho}{\Ga}{\Ga / T(\Ga) \cdot \Ga^p}$ be the natural projection. We have that $\rho(\ga_j) \neq \rho(\eta_j)$ and $\rho(\ga_j),\rho(\eta_j) \neq 1$ by construction. Additionally, we have that $\pi_{T(\Ga)}(\De)$ is a finite index subgroup of $\Ga / T(\Ga)$. Thus, \cite[Lem 4.8(c)]{Hall_notes} implies that $Z(\pi_{T(\Ga)}(\De)) = \pi_{T(\Ga)}(\De) \cap Z(\Ga / Z(\Ga))$. Hence, $\pi_{T(\Ga)}(x_\La) \in Z(\Ga / T(\Ga))$. Since $\rho(\ga_j)$ and $\rho(\eta_j)$ are non-equal central elements of $\Ga / T(\Ga) \cdot \Ga^p$, we have that $\ga_j \nsim \eta_j$.

 Proposition \ref{finite_index_conjugacy} implies that $\CDepth_{\De}( \ga_{j,\La},  \eta_{j,\La}) \leq \CDepth_{\Ga}(\ga_{j,\La},\eta_{j,\La})$. By the above construction, we have $n_j^{\psi_{\text{RF}}(\De)(c(\De) - 1)} \leq \CDepth_{\Ga}(\ga_{j,\La},\eta_{j,\La})$ where $\| \ga_{j,\La}\|_S, \|\eta_{j,\La}\|_{S} \approx n_j$. If $S'$ is a finite generating subset of $\Ga$, then $\| \ga_{j,\La}\|_S \approx \| \ga_{j,\La} \|_{S'}$ and  $\|\eta_{j,\La}\|_{S} \approx \|\eta_{j,\La}\|_{S'}$. Hence, $\|\ga_{j,\La}\|_{S'}, \| \eta_{j,\La} \|_{S'} \approx n_j,$ and $n_j^{\psi_{\text{RF}}(\De)(c(\De) - 1)} \preceq \CDepth_{\Ga}(\ga_{j,\La},\eta_{j,\La})$. Since the projection to the torsion free quotient $\map{\pi_{T(\Ga)}}{\Ga}{\Ga / T(\Ga)}$ is injective when restricted to $\De$, $\De$ is isomorphic to a finite index subgroup of $\Ga / T(\Ga)$, and thus, Theorem \ref{maltsev_invariant} implies that $\psi_{\text{RF}}(\Ga^\ell) = \psi_{\text{RF}}(\Ga)$. Since $c(\Ga / T(\Ga)) = c(\De)$, we have $n^{\psi_{\text{RF}}(\Ga)(c(\Ga/T(\Ga)) - 1)} \preceq \Conj_{\Ga}(n)$.
\end{proof}

\begin{proof}\textit{Theorem \ref{conj_malsev_invariant}}\\
	Suppose that $\Ga$ and $\De$ are two infinite admissible of step size $2$ or greater such that $\Ga / T(\Ga)$ and $\De / T(\De)$ has isomorphic Mal'tsev completions. Proposition \ref{same_growth_maltsev_completion} implies that $\psi_{\text{RF}}(\Ga / T(\Ga)) = \psi_{\text{RF}}(\De / T(\De))$. By definition of $\psi_{\text{RF}}(\Ga)$ and $\psi_{\text{RF}}(\De)$, we have $\psi_{\text{RF}}(\Ga) = \psi_{\text{RF}}(\De)$. Since $c(\Ga / T(\Ga)) = c(\De / T(\De))$, our theorem is now evident.
\end{proof}
\section{Proof of Theorem \ref{applications}}
\begin{proof}
For $s \in \N$, let $\La_s$ be the group given in Definition \ref{filiform}, and let $\map{\theta}{Z(\La_s)}{Z(\La_s)}$ be the identity morphism. Let $c>1$ and $m \geq 1$, and consider the group $\Ga_{cm} = (\La_{c+1} \circ_\theta)_{i=1}^m$ with a finite generating subset $S_{cm}$. Proposition \ref{central_product_growth}  implies that $h(\Ga_{cm}) =  c \: m^2 + c \: m - 1$, and since $c^2 \: m^2 + c^2 \: m - 1 \geq m$, Theorem \ref{lower}(ii) implies that $n^m \preceq \Conj_{\Ga_{cm}}(n)$ as desired.
\end{proof}


\end{document}